\documentclass{amsart}

\usepackage{amsfonts,amsthm,amssymb}
\usepackage[pdftex]{graphicx}
\usepackage{graphics}
\usepackage[matrix,arrow]{xy}
\usepackage[active]{srcltx}
\usepackage[colorlinks=true]{hyperref}

\usepackage[usenames, dvipsnames]{xcolor}

\newcommand{\C}{\mathbb{C}}
\newcommand{\R}{\mathbb{R}}
\newcommand{\Z}{\mathbb{Z}}
\newcommand{\Q}{\mathbb{Q}}
\newcommand{\N}{\mathbb{N}}
\newcommand{\D}{\mathbb D}
\newcommand{\M}{\mathcal{M}}
\newcommand{\J}{\mathcal{J}}

\newcommand{\F}{\mathcal{F}}
\newcommand{\T}{\mathcal{T}}

\newcommand{\W}{\mathcal W}
\renewcommand{\P}{\mathcal{P}}
\renewcommand{\sl}{{\rm sl}}

\newcommand{\wind}{{\rm wind}}

\newcommand{\mon}{{\rm mon}}

\newcommand{\util}{{\widetilde u}}
\newcommand{\vtil}{{\widetilde v}}
\newcommand{\wtil}{{\widetilde w}}
\newcommand{\jtil}{\widetilde J}

\newcommand{\Mob}{\text{M\"ob}}
\newcommand{\dcal}{\mathcal{D}}
\newcommand{\ncal}{\mathcal{N}}

\theoremstyle{plain}
\newtheorem{theorem}{Theorem}[section]
\newtheorem{proposition}[theorem]{Proposition}
\newtheorem{lemma}[theorem]{Lemma}
\newtheorem{corollary}[theorem]{Corollary}

\theoremstyle{definition}
\newtheorem{definition}[theorem]{Definition}

\theoremstyle{remark}
\newtheorem{remark}[theorem]{Remark}

\title[Dynamical characterization of lens spaces]{A dynamical characterization of \\ universally tight lens spaces}
\author{Umberto L. Hryniewicz}
\address[Umberto L. Hryniewicz]{Universidade Federal do Rio de Janeiro -- Departamento de Matem\'atica Aplicada, Av.\ Athos da Silveira Ramos 149, Rio de Janeiro RJ, Brazil 21941-909.}
\email{umberto@labma.ufrj.br}
\author{Joan E. Licata}
\address[Joan E. Licata]{Mathematical Sciences Institute, The Australian National University ACT 2000 Australia.}
\email{joan.licata@anu.edu.au}
\author{Pedro A. S. Salom\~ao}
\address[Pedro A. S. Salom\~ao]{Universidade de S\~ao Paulo,  Instituto de Matem\'atica e Estat\'istica -- Departamento de Matem\'atica, Rua do Mat\~ao, 1010 - Cidade Universit\'aria - S\~ao Paulo SP, Brazil 05508-090.}
\email{psalomao@ime.usp.br}

\begin{document}

\begin{abstract}
We give necessary and sufficient conditions for a closed connected co-orientable contact $3$-manifold $(M,\xi)$ to be a standard lens space based on assumptions on the Reeb flow associated to a defining contact form. Our methods also provide rational global surfaces of section for nondegenerate Reeb flows on $(L(p,q),\xi_{\rm std})$ with prescribed binding orbits.
\end{abstract}

\maketitle

\tableofcontents

\section{Introduction and main results}

In order to study relationships between dynamics and topology, one might ask to what extent dynamical properties of a vector field  on a closed connected oriented manifold determine the diffeomorphism type of the manifold. To give a simple and well-known instance of this phenomenon, surfaces are characterized by algebraically counting zeros of a vector field with isolated zeros: the Euler characteristic determines the closed connected oriented $2$-manifolds. In dimension~$3$ this question becomes more interesting, as counting zeros falls short from being enough.

Introducing extra geometric structure is a common procedure when one wishes to state and prove a characterization theorem, in the hope that the extra structure provides extra tools. We propose to see a $3$-manifold as a contact-type energy level of a Hamiltonian system with two degrees of freedom. In doing so we will end up looking for a classification theorem not only for its diffeomorphism type, but also for the induced contact structure up to contactomorphism. The tools that become available with this particular viewpoint come from the theory of pseudo-holomorphic curves in symplectizations introduced by Hofer~\cite{93}.

Our goal is to characterize universally tight lens spaces in terms of their Reeb dynamics, via the following connection with Hamiltonian systems. If $(W,\omega)$ is a symplectic $4$-manifold then a hypersurface $M \subset W$ is of \textit{contact-type} if $\omega$ has a primitive $\alpha$ near $M$ such that $\lambda := \iota^*\alpha$ is a contact form. Here $\iota : M \to W$ denotes the inclusion. The contact structure $\xi = \ker \lambda$ is then determined up to isotopy. Since $M$ is co-orientable, there exists a Hamiltonian $H$ defined near $M$ realizing $M$ as a regular energy level. The Hamiltonian vector field $X_H$ is defined by $i_{X_H}\omega = -dH$ and is tangent to $M$ since $dH$ vanishes on $TM$. The defining equations for the Reeb vector field $R$ imply that $X_H$ and $R$ have the same trajectories, i.e., the Reeb dynamics reparametrizes the Hamiltonian dynamics on~$M$. Conversely, suppose that we are given a contact form $\lambda$ on the $3$-manifold $M$. The symplectization $(\R\times M,d(e^a\lambda))$ provides an exact symplectic $4$-manifold where $M \hookrightarrow \{0\}\times M$ embeds with contact-type; here $a$ denotes the $\R$-coordinate. The Hamiltonian vector field of $H = a$ coincides with the Reeb vector field on $M$, and the induced contact structure is $\xi = \ker \lambda$. Our plan is to view the pair $(M,\lambda)$ as an energy level of a special kind in order to characterize $(M,\xi)$ in terms of the associated Reeb dynamics. In this paper we use the methods from~\cite{char1,char2} and~\cite{hryn,HS} to understand universally tight lens spaces from this point of view.

This kind of problem was first studied by Hofer, Wysocki and Zehnder in~\cite{char1} where the tight $3$-sphere was dynamically characterized. Let us recall a statement which is contained in~\cite{char1,char2} postponing precise definitions to later sections.

\begin{theorem}[Hofer, Wysocki and Zehnder]\label{theo_char1}
A closed connected contact $3$-man\-ifold $(M,\xi)$ is the tight $3$-sphere if, and only if, $\xi=\ker\lambda$ for a nondegenerate dynamically convex contact form $\lambda$ admitting a closed Reeb orbit which is unknotted, has self-linking number $-1$ and Conley-Zehnder index $3$.
\end{theorem}

Theorem~\ref{theo_char1} is a first major step in understanding the relations between the topology of tight contact $3$-manifolds and their associated Reeb dynamics. It is clear now that dynamical convexity is only a particular instance of a more general assumption which allows for dynamically characterizing tight contact $3$-manifolds.

The proof makes use of disk-filling methods similar to those used by Hofer~\cite{93} to confirm the $3$-dimensional overtwisted Weinstein conjecture. The assumptions in Theorem~\ref{theo_char1} (e.g., nondegeneracy of the contact form, dynamical convexity, Conley-Zehnder index $3$) 
are restrictive. In~\cite{hryn} the assumption on the Conley-Zehnder index of the binding orbit is removed, allowing for applications to Finsler geodesic flows on $S^2$ as in~\cite{global}, where it is proved that pinching conditions on the flag curvatures exclude certain types of closed geodesics. Below we will discuss more dynamical applications of the method.


\subsection{Lens spaces}
Equip $\C^2$ with complex coordinates $(z,w)$ and consider the 3-sphere $S^3 = \{ (z,w) \in \C^2 \mid |z|^2+|w|^2=1 \}.$ The Liouville form
\begin{equation}\label{liouville_form}
\lambda_0 = \frac{1}{4i} (\bar zdz - zd\bar z + \bar wdw - wd\bar w)
\end{equation}
is a primitive of the standard symplectic structure of $\C^2$ and restricts to a contact form on $S^3$, defining the so-called standard contact structure
\begin{equation}
\xi_{\rm std} = \ker \lambda_0|_{S^3}.
\end{equation}
Given relatively prime integers $p \geq q \geq 1$, there is a free action of $\Z_p := \Z/p\Z$ on $(S^3,\xi_{\rm std})$ generated by the contactomorphism
\begin{equation}\label{generator}
(z,w) \mapsto (e^{i2\pi/p}z,e^{i2\pi q/p}w).
\end{equation}
The orbit space of this action is the lens space
\begin{equation}
L(p,q) := S^3 / \Z_p
\end{equation}
and the standard contact structure on $S^3$ descends to a contact structure on $L(p,q)$ still denoted by $\xi_{\rm std}$ and called standard. The case $p=q=1$ is not ruled out and, according to our conventions, $L(1,1)=S^3$. We always consider $L(p,q)$ oriented by $\xi_{\rm std}$, and $\xi_{\rm std}$ is its unique positive universally tight contact structure. The $\Z_p$-action on $S^3$ preserves $\lambda_0$, which descends to a defining contact form for $\xi_{\rm std}$ on $L(p,q)$. To keep a familiar nontrivial case in mind, $L(2,1)$ is diffeomorphic to the unit tangent bundle of any metric on $S^2$, and the Reeb flow of $\lambda_0$ is a constant reparametrization of the geodesic flow of the standard Riemannian metric of constant curvature~$+1$.

Lens spaces were introduced in 1908 by Tietze~\cite{tietze} and provided the first examples of $3$-manifolds not characterized by the fundamental group. Tietze showed that the fundamental group of $L(p,q)$ is $\Z_p$, and he suspected that $L(5,1)$ and $L(5,2)$ were not homeomorphic. This fact was confirmed by Alexander in 1919, even though these manifolds have the same homology groups. In fact, the homology groups of $L(p,q)$ are independent of $q$:
\[
H_*(L(p,q),\Z) \simeq \left\{ \begin{aligned} & \Z, & *=0,3 \\ & 0, & *=2 \\ & \Z_p, & *=1 \end{aligned} \right.
\]
In 1935 Reidemeister~\cite{reid} showed that $L(p,q_1)$ is $L(p,q_2)$ if, and only if, $q_1 \equiv \pm q_2^{\pm1} \mod p$, up to piecewise linear homeomorphism. Only in 1960 with the work of Brody~\cite{brody}  was it proved that this condition classifies lens spaces up to homeomorphism. It is interesting that $L(p,q_1)$ and $L(p,q_2)$ are homotopy equivalent if, and only if, $q_1 \equiv \pm k^2q_2 \ {\rm mod} \ p$ for some $k\in\Z$ and, consequently, in general $L(p,q)$ is not determined by its homotopy type. For an account of this beautiful piece of the history of topology we refer to Dieudonn\'e~\cite{dieudonne}.

\subsection{Preliminary notions}

Here $M$ denotes a closed connected oriented smooth $3$-manifold, $\mathbb D$ denotes the closed unit disk in the complex plane with its standard orientation, and $\partial\mathbb D$ is always oriented counter-clockwise.

\begin{definition}\label{p_unknot_def} 
Let $K\subset M$ be a knot. Given an integer $p\geq 1$, a \textit{$p$-disk} for $K$ is an immersion $u:\D \to M$ such that $u|_{\D\setminus \partial\D}$ is an embedding into $M\setminus K$ and $u|_{\partial \mathbb{D}}$ is a $p:1$ covering map of $K$. When $K$ admits a $p$-disk then $K$ is said to be \textit{$p$-unknotted} and is called an \textit{order $p$ rational unknot}. We call $u$ an \textit{oriented $p$-disk} for $K$ if, in addition, $K$ is oriented and $u|_{\partial \mathbb{D}}$ is orientation-preserving.
\end{definition}

Later we will define a $\Z_p$-valued invariant associated to any order $p$ rational unknot $K$ inside an oriented $3$-manifold, called its {\it monodromy} and denoted by ${\rm mon}(K)$, see section~\ref{section_topology} and also Figure~1.

Let $\lambda$ be a defining contact form for a positive contact structure $\xi$ on $M$. By a \textit{closed Reeb orbit} we mean a pair $P = (x,T)$, where $x:\R\to M$ is a trajectory of the associated Reeb flow with period $T>0$. The set of closed Reeb orbits for $\lambda$ is denoted by $\P(\lambda)$. We call $P$ \textit{prime} if $T$ is the minimal positive period of $x$, and \textit{contractible} if the loop
\begin{equation}\label{map_x_T}
x_T : t \in \R/\Z \mapsto x(Tt) \in M
\end{equation} 
is contractible. If $m\geq 1$ is an integer then the $m$-th iterate of $P$ will be denoted by $P^m:=(x,mT)$. When $A$ is a subset of $M$ we write $P\subset A$ to indicate that the geometric image of $P$ is a subset of $A$. We may abuse the notation and identify a knot tangent to the Reeb vector field of $\lambda$ with the prime closed Reeb orbit it determines. We shall use the notation
\begin{equation}\label{integral_notation}
\int_P\lambda := \int_{\R/\Z} (x_T)^*\lambda = T.
\end{equation}

When $K\subset (M,\xi)$ is an order $p$ rational unknot transverse to $\xi$ then there is a well-known contact invariant $\sl(K)\in{\Q}$ called the {\it {rational} self-linking number}, see \S~\ref{contact_geometry_background} or \cite{BE} for a definition. This number is computed using an auxiliary $p$-disk for $K$, and if $c_1(\xi)$ vanishes on $\pi_2(M)$ then it is independent of the choice of a $p$-disk.

The contact structure $\xi$ becomes a symplectic vector bundle with the bilinear form $d\lambda$, and when the associated first Chern class $c_1(\xi)$ vanishes on $\pi_2(M)$   the Conley-Zehnder index $\mu_{CZ}(P) \in \Z$ and the transverse rotation number $\rho(P) \in\R$ of a contractible closed Reeb orbit $P$ are well-defined. These are invariants of the linearized Reeb flow along $P$ and it is possible to check that
\begin{equation}\label{mu_rho_rel}
\mu_{CZ}(P)\geq 3 \ \Leftrightarrow \ \rho(P) > 1
\end{equation}
and if $\lambda$ is nondegenerate then $\mu_{CZ}(P)=2 \Leftrightarrow \rho(P)=1$, see \S~\ref{ssec_CZ_dim_3} for the definitions.

\begin{figure}\label{fig1}
\includegraphics[width=350\unitlength]{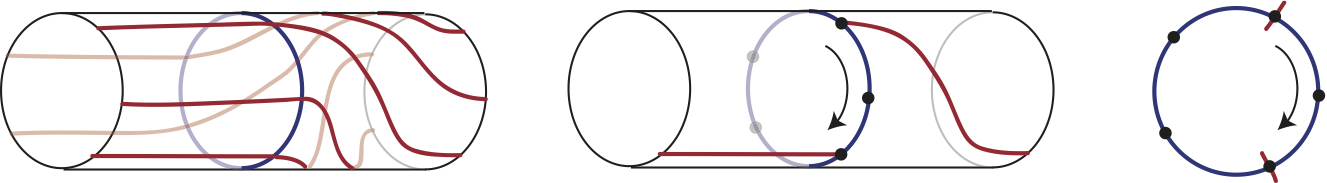}
\caption{On a tubular neighborhood of $K$, a $p$-disk intersects a meridian { $m$ in $p$ points}.  Two intersection points which are adjacent {along the boundary of the $p$-disk} will be separated by $\mon (K)$ points on the meridian.  See above for $p=5$ and monodromy $2$.}
\end{figure}

\subsection{Characterizing lens spaces}

Contact $3$-manifolds are always oriented by the contact structure. We state here our first main result.

\begin{theorem}\label{main2}
Let $(M,\xi)$ be a closed connected tight contact $3$-manifold satisfying $c_1(\xi)|_{\pi_2(M)}=0$, and let $p$ be a positive integer. Then $(M,\xi)$ is contactomorphic to $(L(p,k),\xi_{\rm std})$ for some $k$ if, and only if, $\xi=\ker\lambda$ for a contact form $\lambda$ admitting a prime closed Reeb orbit $K$ such that
\begin{itemize}
\item[i)] $K$ is $p$-unknotted, $\mu_{CZ}(K^p) \geq 3$, $\sl(K) = {\frac{-1}{p}}$, and
\item[ii)] no contractible closed Reeb orbit $P \subset M\setminus K$ with $\rho(P)=1$ is contractible in $M\setminus K$.
\end{itemize}
Moreover, if $\lambda,K$ are as above and $K$ has monodromy $-q$ then $(M,\xi)$ is contactomorphic to $(L(p,q),\xi_{\rm std})$.
\end{theorem}

Note that there are no genericity assumptions on $\lambda$. Later we will be able to say much more about Reeb dynamics of the contact form $\lambda$ admitting an orbit $K$ with the above properties.

Sufficiency in Theorem~\ref{main2} relies on the notion of a rational open book decomposition with disk-like pages, which we now recall.

\begin{definition}
Let $K$ be an oriented knot in $M$. A \it{rational}  \textit{open book decomposition with binding $K$ and disk-like pages of order $p$} is a pair $(K,\pi)$ such that $\pi : M\setminus K \to S^1$ is a smooth fibration, and the closure of each fiber $\pi^{-1}(t)$ is the image of an oriented $p$-disk for $K$.
\end{definition}

As discussed in \S~\ref{sec:monodromy}, a manifold admitting such a decomposition is necessarily a lens space. This definition represents a special case of the rational open book decompositions introduced in~\cite{BEVHM}. The key step in proving sufficiency is provided by the following proposition.

\begin{proposition}\label{proposition_global_sections}
Let $\lambda$ be a defining contact form for a tight contact structure $\xi$ on the closed connected 3-manifold $M$. Suppose that $c_1(\xi)|_{\pi_2(M)}=0$ and that there exists an order $p$ rational unknot $K$ that is tangent to the Reeb vector field of $\lambda$ and has self-linking number $ {\frac{-1}{p}}$. Consider the set $\P^* \subset \P(\lambda)$ of closed Reeb orbits in $M\setminus K$ which are contractible in $M$ and have transverse rotation number equal to~$1$. If $\mu_{CZ}(K^p) \geq 3$ and no element of $\P^*$ is contractible in $M\setminus K$, then $K$ is the binding of a  {rational} open book decomposition $(K,\pi)$ with disk-like pages of order $p$, and there exists some defining contact form $\lambda'$ for $\xi$ such that the Reeb vector field of $\lambda'$ is transverse to the pages and tangent to $K$.
\end{proposition}

Proposition~\ref{proposition_global_sections} follows directly from Propositon~\ref{proposition_global_sections2} and Remark~\ref{seqlambdak} below. The second step for sufficiency in Theorem~\ref{main2} is provided by the following proposition.

\begin{proposition}\label{proposition_classification}
Let $p\geq q\geq 1$ be relatively prime numbers. Suppose that $(M,\xi)$ is a closed connected co-orientable contact $3$-manifold that admits  {a rational} open book decomposition $(K,\pi)$ with disk-like pages of order $p$. If $\mon(K)=-q$ and the Reeb vector field of some defining contact form for $\xi$ is positively tangent to $K$ and positively transverse to the interior of the pages of $(K,\pi)$, then $(M,\xi)$ is contactomorphic to $(L(p,q), \xi_{\rm std})$.  
\end{proposition}

The existence of  {a rational} open book decomposition with disk-like pages of order $p$ and binding $K$ on $M$ is essentially telling us that $M$ is obtained by doing Dehn surgery with coefficient $-p/q$ on a Hopf fiber {;} see Figure~2.

\begin{figure}\label{fig2}
\includegraphics[width=200\unitlength]{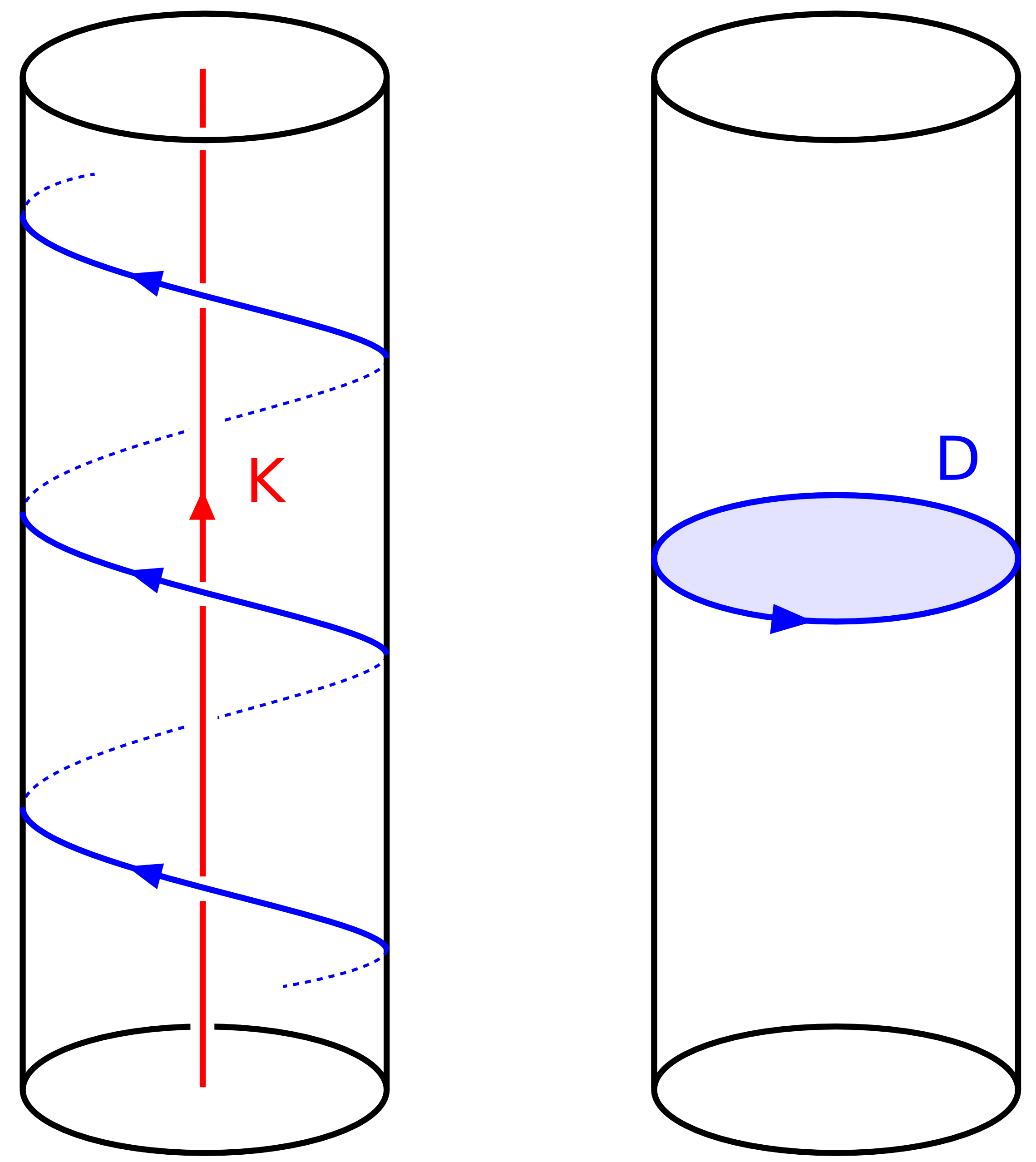}
\caption{$L(p,q)$ is obtained by Dehn surgery with coefficient $-p/q$ on a Hopf fiber.   {This induces a genus $1$ Heegaard splitting of $L(p,q)$, and the two solid tori are shown in the figure.  The boundary of the meridional disc on the right glues to the curve on the boundary of the left-hand solid torus. As the radius of the  torus on the left converges to $0$, the meridional disk $D$ in the other solid torus} converges to a $p$-disk for $K$ which is a page of the rational open book decomposition.}
\end{figure}

Combining Propositions~\ref{proposition_global_sections} and~\ref{proposition_classification}, sufficiency in Theorem~\ref{main2} is proved. Before addressing necessity we recall an important definition.

\begin{definition}[Hofer, Wysocki and Zehnder]
A defining contact form $\lambda$ for $\xi$ is \textit{dynamically convex} if $c_1(\xi)$ vanishes on $\pi_2(M)$ and every contractible closed Reeb orbit $P$ satisfies $\mu_{CZ}(P)\geq3$.
\end{definition}

The simplest example of a dynamically convex contact form is the standard Liouville form $\lambda_0$~\eqref{liouville_form} on $S^3$: the prime closed Reeb orbits are the Hopf fibers and they all have Conley-Zehnder index $3$. Fixing $p>q\geq1$ with $\gcd(p,q)=1$  {$\lambda_0$} obviously descends to a dynamically convex contact form on $L(p,q)$. The Hopf fiber $\widetilde K = S^1\times\{0\}$ descends to a transverse order $p$ rational unknot $K \subset L(p,q)$ with self-linking number $ {\frac{-1}{p}}$ and monodromy $-q$, see Lemma~\ref{lem:lpqsl}. This proves necessity in Theorem~\ref{main2}.

It is important to note that the existence of a defining dynamically convex contact form for $\xi$ already imposes contact-topological restrictions on $(M,\xi)$, as the following result from~\cite{char2} shows; see \S~\ref{contact_geometry_background} for the definition of tight contact structures.

\begin{theorem}[Hofer, Wysocki and Zehnder]\label{thm_dyn_conv_top}
If $\lambda$ is a dynamically convex contact form on a closed 3-manifold $M$, then $\ker\lambda$ is tight and $\pi_2(M)$ vanishes.
\end{theorem}

As an immediate consequence we obtain

\begin{corollary}\label{main}
Fix $p\in \mathbb{Z}^+$.  A closed connected contact $3$-manifold $(M,\xi)$ is contactomorphic to $(L(p,k),\xi_{\rm std})$ for some $k$ if, and only if, $\xi = \ker \lambda$ for a dynamically convex contact form $\lambda$ admitting a prime closed Reeb orbit $K$ which is $p$-unknotted and has self-linking number $ {\frac{-1}{p}}$. Moreover, if $\lambda,K$ are as above and $K$ has monodromy $-q$ then $(M,\xi)$ is contactomorphic to $(L(p,q),\xi_{\rm std})$.
\end{corollary}

\begin{proof}
Necessity follows as in Theorem~\ref{main2}. In view of Theorem~\ref{thm_dyn_conv_top}, $\xi$ is tight and $\pi_2(M)=0$ when $\xi=\ker\lambda$ for some dynamically convex contact form~$\lambda$. From~\eqref{mu_rho_rel} no contractible closed Reeb orbit $P$ satisfies $\rho(P)=1$. Sufficiency follows from Theorem~\ref{main2}.
\end{proof}

Another dynamical characterization is given by Hutchings and Taubes~\cite{HT}. They prove that nondegenerate $3$-dimensional Reeb flows with precisely two closed Reeb orbits only exist in lens spaces.

\subsection{Global surfaces of section on lens spaces}

From a dynamical point of view we are specially interested in  {(rational)} open book decompositions where all pages are global surfaces of section for a Reeb flow given {\it a priori}. The inherent analytical difficulties now come from the fact that, in this context, one is not allowed to change the contact form since the whole point is to study dynamics.

\begin{definition}
Assume $\lambda$ is a contact form on $M$ with Reeb vector field $R$. A \textit{disk-like global surface of section of order $p\geq 1$} for the Reeb dynamics is a $p$-disk $u:\mathbb D\to M$ such that $K = u(\partial \D)$ is tangent to $R$, $u(\D\setminus \partial \D)$ is transverse to $R$, and for every Reeb trajectory $x(t)$ in $M\setminus K$ one finds $t^\pm_n \to \pm\infty$ such that $x(t^\pm_n) \in u(\D\setminus \partial \D)$, $\forall n$.
\end{definition}

The importance of finding global surfaces of section for flows without rest points goes back to Poincar\'e and his studies of the circular planar restricted $3$-body problem (CPR3BP). When they exist one can study the flow on the $3$-dimensional energy level via the associated return map, thus allowing $2$-dimensional dynamical methods to come into play.

Global surfaces of section may exist organized in the form of pages of a  {rational} open book decomposition, as in the following definition.

\begin{definition}\label{def_ob_adapted}
Let $\lambda$ be a defining contact form for $(L(p,q),\xi_{\rm std})$ and let $K\subset L(p,q)$ be a knot. A  {rational} open book decomposition $(K,\pi)$ of $L(p,q)$ with disk-like pages of order $p$ and binding $K$ is   {{\it adapted}} to $\lambda$ if $K$ is a closed Reeb orbit and the pages are disk-like global surfaces of section of order $p$ for the Reeb flow. 
\end{definition}

When the defining contact form $\lambda$ for $(L(p,q),\xi_{\rm std})$ is assumed to be nondegenerate then we can identify the precise conditions for a given closed Reeb orbit to bound a disk-like global surface of section of order $p$. This is made precise in our next main result.

\begin{theorem}\label{main3}
Let $\lambda$ be a nondegenerate defining contact form for $(L(p,q),\xi_{\rm std})$, and let $K\subset L(p,q)$ be a prime closed Reeb orbit. Let $\P^*$ denote the set of closed   {Reeb} orbits $P'\subset L(p,q)\setminus K$ which are contractible in $L(p,q)$ and satisfy $\rho(P')=1$. Then the following assertions are equivalent:
\begin{itemize}
\item[i)] $K$ bounds a disk-like global surface of section of order $p$ for the Reeb flow.
\item[ii)] $K$ is the binding of   {a rational} open book decomposition $(K,\pi)$ with disk-like pages of order $p$ adapted to $\lambda$.
\item[iii)] $K$ is $p$-unknotted, $\sl(K)=  {\frac{-1}{p}}$, $\mu_{CZ}(K^p)\geq 3$ and no orbit in $\P^*$ is contractible in $L(p,q)\setminus K$.
\end{itemize}
\end{theorem}

\begin{proof}
${\rm iii) \Rightarrow ii)}$ follows from~Proposition~\ref{proposition_global_sections2} taking $f_n\equiv1$, $\forall n$. ${\rm ii) \Rightarrow i)}$ is obvious from Definition~\ref{def_ob_adapted} since the pages are global surfaces of section of order $p$. To prove ${\rm i) \Rightarrow iii)}$ note that the existence of a $p$-disk $\mathcal D$ for $K$ transverse to the Reeb vector field immediately implies $\sl(K)=  {\frac{-1}{p}}$. If $\mathcal D$ is a global surface of section then the linking number between $K$ and any other closed Reeb orbit is positive, in particular no orbit of $\P^*$ is contractible in $L(p,q)\setminus K$. The Reeb vector field along $K$ orients ${\rm int} (\mathcal D)$. By Stokes  {'s} theorem, the integral of $d\lambda$ over ${\rm int}(\mathcal D)$ is positive. Since the Reeb vector field is transverse to ${\rm int}(\mathcal D)$ we conclude that $d\lambda$ is positive over ${\rm int}(\mathcal D)$. The geometric description of the Conley-Zehnder index explained in \S~\ref{sssec_geom_defn} tells us that if $\mu_{CZ}(K^p)\leq 1$ then $d\lambda|_{T\mathcal D}$ is negative near $\partial\mathcal D$. Thus $\mu_{CZ}(K^p)\geq 2$. If $\mu_{CZ}(K^p)=2$ then $p\in\{1,2\}$ and $K$ is hyperbolic. Its stable manifold does not wind with respect to $\mathcal D$, and we would find trajectories that never hit $\mathcal D$ in the future, a contradiction. Thus $\mu_{CZ}(K^p)\geq 3$.
\end{proof}

The case $p=1$ was treated in~\cite{HS}. The reader will notice that one could lift the data in Theorem~\ref{main3} to the universal covering $(S^3,\xi_{\rm std})$ and apply the result from~\cite{HS} to obtain disk-like global surfaces of section bounded by the lift of $K$. However, such sections do not necessarily descend to sections on $L(p,q)$.

As explained by Hofer, Wysocki and Zehnder in~\cite{convex}, the existence of a disk-like global section has deep dynamical consequences such as the existence of two or infinitely many close Reeb orbits. In fact, Brouwer's translation theorem provides a fixed point of the return map which corresponds to a second closed Reeb orbit $K'$, geometrically distinct from the binding $K$. Then a result of Franks~\cite{franks} for area-preserving maps of the annulus implies infinitely many closed Reeb orbits if a third closed Reeb orbit exists.  We remark that one could also address the question of existence of infinitely many closed Reeb orbits assuming only the existence of the Hopf link $K \cup K'$. In fact,  if a non-resonance condition on their rotation numbers is satisfied, then one obtains infinitely many closed Reeb orbits characterized by their linking numbers with $K$ and $K'$.  This is proved in~\cite{HMS} with no genericity assumptions on the Reeb flow.

In~\cite{affhk} Albers, Fish, Frauenfelder, Hofer and van Koert use results in Symplectic Dynamics~\cite{BH} to study the CPR3BP. There are two primaries, the Sun with mass $m_S$ and the Earth with mass $m_E$, and a massless satellite. The relative position of the Earth with respect to the Sun is assumed to describe a circular trajectory, and the satellite moves in the same plane as the primaries. When the mass ratio $\mu = m_S/(m_S+m_E)$ equals $1$ the satellite moves around the Sun as in Kepler's problem. For $\mu\in(0,1)$ the Earth comes into play and, after choosing a suitable rotating system of coordinates where Sun and Earth remain at rest on a given axis, the Hamiltonian describing the system has five critical points $L_1,\dots,L_5$ with critical values $h_1,\dots,h_5$ ordered monotonically. The energy levels below $h_1$ have $3$ connected components. After regularizing collisions, two of them, denoted by $\Sigma_{\mu,c}^S$ and $\Sigma_{\mu,c}^E$, become diffeomorphic to $\R P^3 = L(2,1)$ and correspond to Hill's regions around the Sun and the Earth. Here $c<h_1$ refers to the value of the energy. The third is non-compact. Loosely speaking, it is proved in~\cite{affhk} that for $\mu\sim 1$ the Hamiltonian flow on $\Sigma_{\mu,c}^E$ is a dynamically convex Reeb flow on $(L(2,1),\xi_{\rm std})$. As a consequence, the following remarkable result from~\cite{convex} applies to the flow lifted to the universal covering $\widetilde \Sigma_{\mu,c}^E$.

\begin{theorem}[Hofer, Wysocki and Zehnder \cite{convex}]\label{thm_convexo}
The Reeb dynamics associated to any dynamically convex contact form on $S^3$ admits a disk-like global surface of section.
\end{theorem}

Theorem~\ref{thm_convexo} provides a disk-like global surface of section $\mathcal D$ on $\widetilde \Sigma_{\mu,c}^E$. However, $\mathcal D$ may not necessarily project onto a global surface of section on $\Sigma_{\mu,c}^E$, and it is not known precisely which orbits can $\partial\mathcal D$ be. 

In~\cite{HS3BP} the problem of constructing disk-like global surfaces of section directly inside $\Sigma_{\mu,c}^E$ and $\Sigma_{\mu,c}^S$ with prescribed boundary orbits is studied. For any energy value $c$ below $h_1$ and mass ratio $\mu\sim 1$, the results from~\cite{affk,affhk} are used to find disk-like global surfaces of section of order $2$ inside $\Sigma_{\mu,c}^E$ or $\Sigma_{\mu,c}^S$ bounded by a prescribed $2$-unknotted periodic orbit $P$ with self-linking number   {$\frac{-1}{2}$}. No assumption on the Conley-Zehnder index of $P$ is made. It is expected that particular periodic orbits of physical interest satisfy these assumptions.

Hamiltonian systems coming from   {p}hysics need not satisfy genericity assumptions. Hence Theorem~\ref{main3} may not be useful to the study of CPR3BP. In~\cite{HS3BP} it will be proved that the analysis performed here is still sufficient for constructing global surfaces of section even when the contact form is degenerate. As a preliminary and crucial step in this direction we collect here the following technical proposition.

\begin{proposition}\label{thm_global_sections}
Let $\lambda$ be a defining contact form for $(L(p,q),\xi_{\rm std})$, and let $K\subset L(p,q)$ be a $p$-unknotted prime closed Reeb orbit satisfying $\sl(K)=  {\frac{-1}{p}}$ and $\mu_{CZ}(K^p)\geq 3$. Consider $\P^*\subset\P(\lambda)$ the set of closed orbits $P'\subset L(p,q)\setminus K$ which are contractible in $L(p,q)$ and satisfy $\rho(P')=1$. Consider also a $p$-disk $u_0$ for $K$ which is special robust for $(\lambda,K)$. Let $e$ be the (unique) singular point of the characteristic foliation of $u_0$ and fix $V$ a small open neighborhood of $e$. Suppose that every orbit $P'\in\P^*$ satisfies
\begin{center}
$P'$ is not contractible in $L(p,q)\setminus K\ \ \ $ or $\ \ \ \int_{P'}\lambda> 1+\int_\D |u_0^*d\lambda|$.
\end{center}
Then for every sequence of smooth functions $f_n:L(p,q)\to(0,+\infty)$ satisfying $f_n|_K\equiv1$, $df_n|_K\equiv0$, $f_n|_V\equiv 1$, $f_n\to 1$ in $C^\infty$ and such that $\lambda_n:=f_n\lambda$ is nondegenerate $\forall n$, one finds $n_0$ such that $\forall n\geq n_0$ there exists   {a rational} open book decomposition $(K,\pi_n)$ with disk-like pages of order $p$ adapted to~$\lambda_n$.
\end{proposition}

See Definition~\ref{def_special_robust} for the notion of special robust $p$-disks for $(\lambda,K)$. \\

\noindent {\it Acknowledgements.} The results of this paper originated when the authors were at IAS for the thematic year on Symplectic Dynamics. We would like to thank Professor Helmut Hofer for creating such a stimulating academic environment. UH was partially supported by CNPq grant 309983/2012-6 and by NSF grant DMS-0635607. JL was partially supported by NSF grant DMS-1105312. PS was partially supported by CNPq grant 303651/2010-5 and by FAPESP grant 2011/16265-8.

\section{Preliminaries}\label{sec_preliminaries}

In this section we recall some  definitions from contact geometry and discuss the Conley-Zehnder index and the basics from pseudo-holomorphic curve theory in symplectizations. 

\subsection{Contact geometry background}\label{contact_geometry_background}

Recall that a \textit{contact form} on a $3$-manifold $M$ is a $1$-form $\lambda$ such that $\lambda \wedge d\lambda$ never vanishes.  A \textit{contact structure} $\xi$ is a $2$-plane distribution locally expressed as the kernel of locally defined contact forms. These local contact forms $\lambda$ determine volume forms $\lambda \wedge d\lambda\neq 0$ which induce an orientation of $M$, called the orientation \textit{induced by $\xi$}. When $M$ is already oriented, $\xi$ is said to be \textit{positive} when it induces the orientation of $M$.  If $\xi$ is co-orientable, then there is a globally defined contact form $\lambda$ such that $\xi = \ker \lambda$, and conversely. We refer to such $\lambda$ as a \textit{defining contact form} for $\xi$. The \textit{Reeb vector field} $R$ associated to $\lambda$ is uniquely determined by
\begin{equation}\label{eqns_Reeb_vector}
i_Rd\lambda = 0, \ \ \ i_R\lambda = 1.
\end{equation}
The first equation determines $R$ up to multiplication by a non-vanishing function since $\ker d\lambda$ is $1$-dimensional. The second equation is a normalization condition.

An \textit{overtwisted disk} is an embedded disk $D$ such that $T\partial D \subset \xi$ and $T_pD \neq \xi_p \ \forall p\in\partial D$. A contact structure $\xi$ is \textit{overtwisted} if it admits an overtwisted disk, and it is \textit{tight} otherwise. We abuse the terminology and call $\lambda$ tight when $\ker\lambda$ is tight. The contact manifold $(M,\xi)$ is called \textit{universally tight} when the lift of $\xi$ to the universal covering of $M$ is tight. It is well-known that $\xi_{\rm std}$ is the only universally tight positive contact structure on $L(p,q)$, up to contactomorphism.


Now assume $\xi$ is a positive co-orientable contact structure on $M$ and let $K$ be a  knot transverse to $\xi$ which is $p$-unknotted. Let $u$ be a $p$-disk for $K$ and let $Z$ be a smooth non-vanishing section $Z$ of $u^*\xi$. For $\epsilon>0$ small consider the loop $Z_\epsilon:\R/\Z\to M\setminus K$, $Z_\epsilon(t) := \exp(\epsilon Z(e^{i2\pi t}))$ where $\exp$ is any exponential map.

\begin{definition}\label{sl_number_def}
The   {rational} \textit{self-linking number} $\sl(K,u) \in \Q$ is 
\[
\sl(K,u) =   {\frac{1}{p^2}} \#(Z_\epsilon \cap u) \ \ \ \text{(oriented intersection number)}.
\]
 
\end{definition}
This value is independent of the choices of $\epsilon$ small, $\exp$ and $Z$.
Setting $\bar u(z) := u(\bar z)$ we have $\sl(K,u)=\sl(K,\bar u)$, and if $c_1(\xi)$ vanishes on $\pi_2(M)$, then $\sl(K,u)$ does not depend on $u$.  Setting $p=1$ recovers the ordinary self-linking number of a null-homologous knot; we will apply the term ``self-linking number" to both the classical and rational cases.

\begin{remark}
A more general definition of the rational self-linking number of rationally null-homologous transverse knots is introduced in  \cite{BE}.  \end{remark}

\subsection{The Conley-Zehnder index in dimension $3$}\label{ssec_CZ_dim_3}

We recall here two definitions of the Conley-Zehnder index in dimension $3$. The first is given in terms of spectral properties of asymptotic operators following~\cite{props2}, the second is more geometric and is taken from the appendix of~\cite{fols}.

\subsubsection{An analytical definition of the index}\label{sssec_analytical_defn}

Let $S:\R/\Z \to \R^{2\times 2}$ be a smooth map satisfying $S(t)^T=S(t) \ \forall t$. We freely identify $\R^2 \simeq \C$, $\R^{2\times 2} \simeq \mathcal L_\R(\C)$ and consider the unbounded self-adjoint operator
\begin{equation}\label{operator_normal_form}
L_S = -i \partial_t - S(t)
\end{equation}
on $L^2(\R/\Z,\C)$ equipped with the Hilbert space structure induced by the standard Euclidean inner-product on $\C$.

In~\cite{props2} a number of properties of $L_S$ are studied. Its spectrum $\sigma(L_S)$ consists of a discrete sequence of real eigenvalues accumulating at $\pm\infty$. Any non-trivial vector $e:\R/\Z \to \C$ in the eigenspace associated to some $\nu \in \sigma(L_S)$ never vanishes and has a well-defined winding number
\begin{equation}
\wind(e(t)) = \frac{1}{2\pi}(\theta(1)-\theta(0)) \in \Z
\end{equation}
where $\theta : [0,1] \to \R$ is some continuous function satisfying $e(t) \in \R^+ e^{i\theta(t)}$. It turns out that $\wind(e(t))$ does not depend on the choice of $e(t) \in \ker (L_S-\nu I)$ and we denote it by $\wind(\nu)$, the winding of the eigenvalue $\nu$. If $\nu_1\leq\nu_2$ then $\wind(\nu_1)\leq\wind(\nu_2)$, and for every $k\in\Z$ there are precisely two eigenvalues (multiplicities counted) with winding $k$. Eigenvectors associated to distinct eigenvalues with the same winding are pointwise linearly independent.

We denote by $Sp(2)$ the group of symplectic $2\times2$ matrices and consider the set $\Sigma$ of smooth maps $\varphi : \R \to Sp(2)$ satisfying
\begin{equation}\label{periodicity_path}
\varphi(t+1) = \varphi(t)\varphi(1) \ \forall t \ \ \text{and} \ \ \varphi(0) = I.
\end{equation}
Then $S(t) = -i\dot\varphi(t)\varphi(t)^{-1}$ defines a smooth $1$-periodic matrix-valued function satisfying $S(t)^T=S(t) \ \forall t$, and there is an unbounded self-adjoint operator $L_S$ as in~\eqref{operator_normal_form}. Following~\cite{props2} we consider the extremal eigenvalues
\begin{equation}
\begin{array}{ccc}
\nu^{\geq0} = \min \{ \nu \in \sigma(L_S) \mid \nu\geq0 \} & & \nu^{<0} = \max \{ \nu \in \sigma(L_S) \mid \nu<0 \},
\end{array}
\end{equation}
the associated extremal winding numbers
\begin{equation}
\begin{array}{ccc}
\wind^{\geq0}(L_S) = \wind(\nu^{\geq0}) & & \wind^{<0}(L_S) = \wind(\nu^{<0}),
\end{array}
\end{equation}
and the parity
\begin{equation}
p(L_S) = \wind(\nu^{\geq0}) - \wind(\nu^{<0}) \in \{0,1\}.
\end{equation}
The Conley-Zehnder index of the path $\varphi(t)$ is defined by
\begin{equation}
\mu_{CZ}(\varphi) = 2\wind^{<0}(L_S) + p(L_S).
\end{equation}

Consider the set of nondegenerate paths 
\[
\Sigma' = \{ \varphi \in \Sigma \mid \det (\varphi(1)-I) \not= 0\}.
\]
The Conley-Zehnder index satisfies the following axioms:
\begin{itemize}
\item ({\bf Homotopy}) If $\{\varphi_s\}_{s\in[0,1]}$ is a continuous homotopy of paths in $\Sigma'$ then $\mu_{CZ}(\varphi_1)=\mu_{CZ}(\varphi_0)$.
\item ({\bf Maslov index}) If $\varphi \in \Sigma$ and $\psi:\R/\Z\to Sp(2)$ is a smooth loop satisfying $\psi(0)=I$ then $\mu_{CZ}(\psi\varphi) = \mu_{CZ}(\varphi) + 2\text{Maslov}(\psi)$, where $\text{Maslov}(\psi)$ is the usual Maslov index of $\psi$.
\item ({\bf Inverse}) $\mu_{CZ}(\varphi^{-1}) = - \mu_{CZ}(\varphi) \ \forall \varphi \in \Sigma'$.
\item ({\bf Normalization}) The path $t \mapsto \begin{pmatrix} \cos \pi t & -\sin \pi t \\ \sin \pi t & \cos \pi t \end{pmatrix}$ has index equal to $1$.
\end{itemize}

\subsubsection{A geometric definition of the index}\label{sssec_geom_defn}

The index can be alternatively given the following geometric description. To every closed interval $J$ of length strictly less than $1/2$ satisfying $\partial J \cap \Z = \emptyset$ we associate an integer as follows.  Let $k$ range over $\mathbb{Z}$ and define
\[
\tilde\mu(J) = \left\{ \begin{aligned} & 2k \ \text{if} \ k\in J \\ & 2k+1 \ \text{if} \ J \subset (k,k+1). \end{aligned} \right.
\]
The function $\tilde\mu$ can be extended to the set of all closed  intervals of length strictly less than $1/2$ by
\[
\tilde\mu(J) = \lim_{\epsilon\to0^+} \tilde\mu(J-\epsilon).
\]
Given $\varphi\in\Sigma$, consider the map $\Delta_\varphi : \C\setminus\{0\} \to \R$ defined by
\begin{equation}\label{map_Delta_twist}
\Delta_\varphi(\zeta) = \frac{1}{2\pi} (\theta(1)-\theta(0))
\end{equation}
where $\theta:\R\to\R$ is a continuous function satisfying $\varphi(t)\zeta\in\R^+e^{i\theta(t)}$. The image $I_\varphi = \Delta_\varphi(\C\setminus\{0\})$ is a closed interval of length strictly less than $1/2$. It turns out that
\begin{equation}
\mu_{CZ}(\varphi) = \tilde\mu(I_\varphi).
\end{equation}

\subsubsection{Relative winding numbers}\label{sssec_rel_winding}

Let $\Pi:E\to\R/\Z$ be an oriented vector bundle satisfying ${\rm rank}_\R(E)=2$. We denote by $\Omega_E^+$ the set of homotopy classes of oriented trivializations of $E$.

Any non-vanishing section $t\in\R/\Z \mapsto Z(t) \in \Pi^{-1}(t)$ can be completed to a positive frame $\{Z,Z'\}$ of $E$, that is, there exists another non-vanishing section $Z'$ such that $\{Z(t),Z'(t)\}$ is an oriented basis of $\Pi^{-1}(t)$, $\forall t$. The frame $\{Z,Z'\}$ determines an oriented trivialization and its homotopy class $\beta \in \Omega_E^+$ depends only on $Z$ up to homotopy through non-vanishing sections; it is called the homotopy class induced by $Z$.

If $W$ is another non-vanishing section then $W(t) = a(t)Z(t) + b(t)Z'(t)$ for unique smooth functions $a,b$ and we write
\[
\wind(W,Z) = \frac{1}{2\pi}(\theta(1)-\theta(0)) \in \Z
\]
where $\theta:[0,1]\to\R$ is a smooth map satisfying $a(t)+ib(t) \in \R^+e^{i\theta(t)}$. This integer depends only on the homotopy classes of non-vanishing sections of $Z$ and~$W$. Denoting by $\beta' \in \Omega^+_E$ the homotopy class of oriented trivializations induced by $W$ we may also write $\wind(\beta',\beta)$ or $\wind(W,\beta)$ to denote $\wind(W,Z)$.

\subsubsection{The Conley-Zehnder index of closed Reeb orbits}\label{sssec_CZ_orbits}

Let $P=(x,T)$ be a closed orbit of the Reeb flow $\phi_t$ associated to the contact form $\lambda$ on $M$. The associated contact structure $\xi = \ker\lambda$ is preserved by $\phi_t$. 

Recall the map $x_T$~\eqref{map_x_T}. The bundle $(x_T)^*\xi \to \R/\Z$ is always assumed to be oriented by $d\lambda$. A $d\lambda$-symplectic trivialization $\Psi:(x_T)^*\xi\to\R/\Z\times\R^2$ representing some given class $\beta \in \Omega^+_{(x_T)^*\xi}$ allows us to define a path $\varphi \in \Sigma$ of symplectic $2\times2$ matrices by
\begin{equation}\label{rep_linear_flow}
\varphi(t) = \Psi_t \circ d\phi_{Tt} \circ (\Psi_0)^{-1}
\end{equation}
where $\Psi_t$ is the restriction of $\Psi$ to fiber over $t$. The Conley-Zehnder index of the pair $(P,\beta)$ is the integer
\begin{equation}\label{def_CZ_geom_orbits}
\mu_{CZ}(P,\beta) = \mu_{CZ}(\varphi).
\end{equation}
In view of the axioms described in \S~\ref{sssec_analytical_defn},  $\mu_{CZ}(P,\beta)$ does not depend on the choice of symplectic trivialization $\Psi$ in the class $\beta$. Moreover,
\[
\mu_{CZ}(P,\beta) = \mu_{CZ}(P,\beta') + 2\wind(\beta',\beta).
\]

If $P$ is contractible in $M$ and $c_1(\xi)$ vanishes on $\pi_2(M)$, then there exists a special class $\beta_{\rm disk} \in \Omega^+_{(x_T)^*\xi}$ induced by any capping disk. Indeed, if $f:\D\to M$ is a smooth map satisfying $f(e^{i2\pi t})=x(Tt)$ then $f^*\xi$ is a trivial symplectic vector bundle, and the restriction of a symplectic trivialization of $f^*\xi$ to $\partial \D$ singles out the homotopy class $\beta_{\rm disk}$, which is independent of $f$ when $c_1(\xi)$ vanishes on $\pi_2(M)$. We may write 
\begin{equation}
\mu_{CZ}(P) = \mu_{CZ}(P,\beta_{\rm disk}) \ \ \text{ when $P$ is contractible and } c_1(\xi)|_{\pi_2(M)}=0.
\end{equation}

\subsubsection{Transverse rotation number}\label{sssec_transv_rot_number}

With $P=(x,T)$, $\lambda$, $\phi_t$ and $\xi$ as above, consider a $d\lambda$-symplectic trivialization $\Psi$ of $(x_T)^*\xi$. Then $\varphi:\R\to Sp(2)$ given by~\eqref{rep_linear_flow} satisfies~\eqref{periodicity_path}. Consider the unique continuous function $\theta : \R\times\R \to \R$ satisfying $\varphi(t)e^{i2\pi s} \in \R^+e^{i\theta(t,s)}$, $\theta(0,s)=2\pi s$. Setting $f(s)=\theta(1,s)/2\pi$,  $f$ then satisfies $f(s+1)=f(s)+1$ and $f(s)-s = \Delta_\varphi(e^{i2\pi s})$, where $\Delta_\varphi$ is the map~\eqref{map_Delta_twist}. Denoting by $\beta \in \Omega^+_{(x_T)^*\xi}$ the homotopy class of $\Psi$ we consider rotation number
\begin{equation}
\rho(P,\beta) = \lim_{k\to+\infty} \frac{f^k(s)}{k}.
\end{equation}
It is not hard to show that
\[
\rho(P,\beta) = \lim_{k\to+\infty} \frac{\mu_{CZ}(P^k,\beta^{(k)})}{2k}
\]
where $P^k$ and $\beta^{(k)}$ denote the $k$-th iterates of $P$ and $\beta$, respectively. If $P$ is contractible and $c_1(\xi)$ vanishes on $\pi_2(M)$ then the class $\beta_{\rm disk}$ explained above is well-defined and we write $\rho(P) = \rho(P,\beta_{\rm disk})$.

  {Note that the transverse rotation number determines the Conley-Zehnder index, see for instance~\cite{hutchings}.}

\subsubsection{Asymptotic operators}\label{sssec_asymp_ops}

Here we still consider a closed orbit $P=(x,T)$ of the Reeb flow $\phi_t$ associated to $\lambda$ as in \S~\ref{sssec_CZ_orbits}. A $d\lambda$-compatible complex structure $J$ on $(x_T)^*\xi$ determines the $L^2$-inner product of two sections $Z,W$ of $(x_T)^*\xi$
\[
\int_{\R/\Z} d\lambda_{x_T(t)}(Z(t),J_t \cdot W(t)) dt,
\]
 where $J_t$ is the restriction of $J$ to the fiber over $t$. The corresponding Hilbert space of square-integrable sections is denoted by $L^2((x_T)^*\xi)$. With the help of a symmetric connection $\nabla$ on $TM$ one defines the so-called asymptotic operator at $P$ as the unbounded self-adjoint operator
\begin{equation}\label{defn_asymp_operator}
A_P : \eta \mapsto -J(\nabla_t\eta - T\nabla_\eta R)
\end{equation}
on $L^2((x_T)^*\xi)$. Here $R$ is the Reeb vector field and $\nabla_t$ is covariant differentiation along the curve $x_T$. Note that $A_P$ does not depend on the choice of $\nabla$.

Let $\beta \in \Omega^+_{(x_T)^*\xi}$ be fixed. With the help of a $d\lambda$-symplectic trivialization of $(x_T)^*\xi$ in the class $\beta$ the operator $A_P$ is represented as $-J(t)\partial_t-S(t)$, for smooth maps $t \in \R/\Z \mapsto J(t),S(t) \in \mathcal L_\R(\C)$, $J(t)^2=-I$ and $\det J(t)=1$, $\forall t$. If the trivialization is $(d\lambda,J)$-unitary, which means that it is simultaneously $d\lambda$-symplectic and  $J$-complex, then $J(t) \equiv i$ and $S(t)^T=S(t) \ \forall t$, i.e., $A_P$ takes the form $L_S$~\eqref{operator_normal_form} in this frame. Hence $A_P$ has all the spectral properties discussed in~\ref{sssec_analytical_defn}. In particular, given $\nu \in\sigma(A_P) = \sigma(L_S)$ an eigensection $\zeta \in \ker(A_P-\nu I)$ is represented in this trivialization as an eigenvector $e(t) \in \ker(L_S-\nu I)$ and we can consider the winding number
\begin{equation}\label{spectral_winding}
\wind(\zeta,\beta) := \wind(e).
\end{equation}
Analogously as before we define
\begin{equation}\label{extremal_windings}
\begin{aligned}
& \wind(\nu,\beta) := \wind(\zeta,\beta) \ \text{for any} \ \zeta \in \ker(A_P-\nu I), \\
& \wind^{<0}(A_P,\beta) := \wind^{<0}(L_S), \\
& \wind^{\geq0}(A_P,\beta) := \wind^{\geq0}(L_S).
\end{aligned}
\end{equation}
These winding numbers clearly do not depend on the choice of $(d\lambda,J)$-unitary trivialization representing $\beta$. The parity
\begin{equation}\label{parity}
p(A_P) = \wind^{\geq0}(A_P,\beta) - \wind^{<0}(A_P,\beta)
\end{equation}
does not depend on $\beta$. According to our definitions we have
\[
\mu_{CZ}(P,\beta) = 2\wind^{<0}(A_P,\beta) + p(A_P).
\]

  {
\begin{remark}
A proof that the above formula indeed gives the Conley-Zehnder index~\eqref{def_CZ_geom_orbits} can be found in~\cite[Section~2.1]{hryn2}.
\end{remark}
}

\subsection{Pseudo-holomorphic curves in symplectizations}

Pseudo-holomorphic curve techniques were introduced in symplectic geometry by Gromov~\cite{gromov}. These techniques were used to study Reeb flows by Hofer~\cite{93} who confirmed the three-dimensional Weinstein conjecture in numerous cases. We recall here the basic notions of this theory, and refer the reader to~\cite{abbas,AH} for introductory expositions.

Fix a closed co-orientable contact $3$-manifold $(M,\xi)$ and a defining contact form~$\lambda$. The associated Reeb vector field is denoted by $R$ and the projection onto $\xi$ along the Reeb direction is denoted by
\begin{equation}\label{proj_along_Reeb}
\pi:TM \to \xi.
\end{equation}
The symplectization $(\R\times M,d(e^a\lambda))$ has a natural $\R$-action by symplectic dilations given by translations of the $\R$-coordinate $a$.

\subsubsection{Cylindrical almost complex structures}

Let $\J_+(\xi)$ denote the set of complex structures on $\xi$ which are $d\lambda$-compatible. As is well-known, this is a contractible space when equipped with the $C^\infty_{\rm loc}$-topology.

Following Hofer~\cite{93}, for a given $J \in \J_+(\xi)$ we consider the $\R$-invariant almost complex structure $\jtil$ on $\R\times M$ defined by
\begin{equation}\label{cyl_acs}
\begin{array}{ccc}
\jtil \cdot \partial_a = R, & & \jtil|_\xi \equiv J.
\end{array}
\end{equation}
Here we see $R$ and $\xi$ as $\R$-invariant objects on $\R\times M$. It is easy to check that $\jtil$ is $d(e^a\lambda)$-compatible.

\subsubsection{Finite-energy surfaces}\label{sssec_fe_surfaces}

Let $(S,j)$ be a compact Riemann surface and $\Gamma \subset S \setminus\partial S$ be a finite set. Equipping $\R\times M$ with an almost complex structure $\jtil$ as in~\eqref{cyl_acs}, we call a non-constant map $\util : S\setminus\Gamma \to \R\times M$ a finite-energy surface if it is $\jtil$-holomorphic,  i.e., it satisfies the Cauchy-Riemann equation
\[
\bar\partial_{\jtil} (\util) = \frac{1}{2} (d\util + \jtil(\util) \circ d\util \circ j) = 0
\]
and has finite Hofer energy 
\begin{equation}\label{energy}
E(\util) = \sup_{\phi\in\Lambda} \int_{S\setminus\Gamma} \util^*d(\phi\lambda) < \infty.
\end{equation}
Here $\Lambda$ denotes the collection of smooth maps $\phi:\R\to[0,1]$ satisfying $\phi'\geq0$. In the integrand above we see $\phi$ as a function on $\R\times M$ depending only on the $\R$-coordinate and $\lambda$ as an $\R$-invariant $1$-form.

The elements of $\Gamma$ are called \textit{punctures}. Let us denote by $\util=(a,u)$ the components of the map $\util$. A point $z\in\Gamma$ is a \textit{positive} puncture if $a(\zeta) \to +\infty$ as $\zeta \to z$, it is a \textit{negative} puncture if $a(\zeta) \to -\infty$ as $\zeta \to z$, and it is called \textit{removable} if $a$ is bounded near $z$. It is a non-trivial fact that every puncture is either positive, negative or removable, see~\cite{93}. Moreover, as the terminology suggests, $\util$ can be smoothly continued across a removable puncture.

It follows from Stokes'  {s} theorem that there exists always at least one positive puncture. More precisely, if a map $\util$ as above satisfies $\bar\partial_{\jtil}(\util)=0$, $E(\util)<\infty$ and has no positive punctures then $E(\util)=0$ and therefore $\util$ is constant.

\subsubsection{Asymptotic behavior}

Let $\util = (a,u) : (S\setminus\Gamma,j) \to (\R\times M,\jtil)$ be a finite-energy surface as described above, and fix a non-removable puncture $z \in \Gamma$. Associated to a holomorphic embedding $\varphi:(\D,0)\to(S,z)$ we have positive cylindrical coordinates $(s,t) \in [0,+\infty)\times\R/\Z$ at $z$ defined by $(s,t) \simeq \varphi(e^{-2\pi(s+it)})$, and negative cylindrical coordinates $(s,t) \in (-\infty,0]\times\R/\Z$ at $z$ defined by $(s,t) \simeq \varphi(e^{2\pi(s+it)})$. The sign of the non-removable puncture $z$ is $\epsilon_z=+1$ if $z$ is positive, or $\epsilon_z=-1$ if $z$ is negative. The following important statement gives the connection between finite-energy surfaces and periodic orbits of $R$.

\begin{theorem}[Hofer~\cite{93}]\label{thm_93}
Let $(s,t)$ be positive cylindrical coordinates at $z$. Then for every sequence $s_n \to +\infty$ there exists a subsequence $s_{n_k}$, a closed Reeb orbit $P=(x,T)$ and $c\in\R$ such that $u(s_{n_k},t) \to x(\epsilon_z Tt+c)$ in $C^\infty(\R/\Z,M)$ as $k\to\infty$.
\end{theorem}

\begin{definition}[Martinet tubes]\label{def_Martinet}
Let $P = (x,T) \in \P(\lambda)$ be fixed, and consider the minimal positive period $T_{\rm min}$ of $x$. A \textit{Martinet tube} at $P$ is a pair $(U,\Phi)$ consisting of an open neighborhood $U \subset M$ of $x(\R)$ and a diffeomorphism $$ \Phi:U\to\R/\Z\times B \ \ \ \text{($B \subset \R^2$ is an open ball centered at the origin)} $$ such that $\Phi(x(T_{\rm min}t)) = (t,0,0) \ \forall t\in\R/\Z$, and $$\Phi_*\lambda = f(d\theta + x_1dx_2)$$ for some smooth function $f:\R/\Z\times B \to \R^+$ satisfying $f|_{\R/\Z\times \{(0,0)\}} \equiv T_{\rm min}$ and $df|_{\R/\Z\times \{(0,0)\}} \equiv 0$. Here $(\theta,x_1,x_2)$ are coordinates on $\R/\Z\times B$.
\end{definition}

\begin{theorem}[Hofer, Wysocki and Zehnder~\cite{props1}]\label{thm_precise_asymptotics}
Assume that $\lambda$ is nondegenerate. Let $(s,t)$ be positive or negative cylindrical coordinates at the non-removable puncture $z$ if $z$ is a positive or a negative puncture, respectively. Then there exists $P = (x,T) \in \P(\lambda)$ and $c\in\R$ such that the loops $t\in\R/\Z \mapsto u(s,t)$ converge to $t\in\R/\Z\mapsto x(Tt+c)$ in $C^\infty$ as $\epsilon_zs\to+\infty$. After rotating the cylindrical coordinates we can assume that $c=0$. Moreover, either $u(s,t) \in x(\R) \ \forall(s,t)$ when $|s|$ is large enough, or there exists an eigenvector $\eta:\R/\Z\to(x_T)^*\xi$ of $A_P$~\eqref{defn_asymp_operator} associated to an eigenvalue $\nu$ satisfying $\epsilon_z\nu<0$, and some $r>0$ for which the following holds. Consider a Martinet tube $(U,\Phi)$ at $P$, so that for $|s|\gg1$ we have well-defined components $\Phi\circ u(s,t) = (\theta(s,t),z(s,t)) \in \R/\Z\times \R^2$. Then the lift $\widetilde\theta(s,t):\R\times\R\to\R$ of $\theta(s,t)$ uniquely determined by $\widetilde\theta(s,0) \to 0$ as $|s|\to\infty$ satisfies
\[
\lim_{\epsilon_z s\to+\infty} e^{r|s|}|D^\alpha[\widetilde\theta(s,t)-kt]| = 0 \ \forall D^\alpha = \partial_s^{\alpha_1}\partial_t^{\alpha_2}
\]
where $k = T/T_{\rm min} \in \Z^+$ is the multiplicity of $P$ and $T_{\rm min}$ is its minimal positive period. The $\R$-component satisfies
\[
\lim_{\epsilon_z s\to+\infty} e^{r|s|}|D^\alpha[a(s,t)-Ts-a_0]| = 0 \ \forall D^\alpha
\]
for some $a_0\in\R$. If we use the coordinates to represent $\eta(t)$ as a smooth (non-vanishing) function $e(t):\R/\Z\to\R^2$ then
\[
z(s,t) = e^{\int_{s_0}^s h(\tau)d\tau}(e(t)+R(s,t))
\]
when $|s|>|s_0|$, for some fixed $s_0$ such that $\epsilon_zs_0\gg1$ and functions $h,R$ satisfying
\[
\lim_{\epsilon_zs\to+\infty} D^j[h(s)-\nu] = 0 \ \forall j \ \ \ \ \ \lim_{\epsilon_zs\to+\infty} D^\alpha R(s,t) = 0 \ \forall \alpha.
\]
\end{theorem}

\subsubsection{Algebraic invariants and fast planes}\label{sec_alg_invs}

We recall some algebraic invariants introduced in~\cite{props2}. Let $\util = (a,u) : (S\setminus\Gamma,j) \to (\R\times M,\jtil)$ be a finite-energy surface as described in \S~\ref{sssec_fe_surfaces}, and assume that $\Gamma$ consists of non-removable punctures. The vector bundle ${\rm Hom}_\C(T(S\setminus \Gamma),u^*\xi)$ admits a complex structure induced by $J$, and the section $\pi \circ du$ satisfies 
\begin{equation*}
J \circ \pi \circ du = \pi \circ du \circ j.
\end{equation*}
Therefore, either $\pi \circ du$ vanishes identically or its zeros are isolated. Moreover, if we assume that $\lambda$ is nondegenerate then the asymptotic behavior described in Theorem~\ref{thm_precise_asymptotics} implies that $\pi \circ du$ does not vanish near $\Gamma$ if it does not vanish identically. For this discussion we assume $\lambda$ is nondegenerate.

\begin{definition}[Hofer, Wysocki and Zehnder~\cite{props2}]\label{defn_wind_pi}
If $\pi\circ du$ does not vanish identically then we set \begin{equation}
\wind_\pi(u) = \text{algebraic count of zeros of } \pi\circ du.
\end{equation}
\end{definition}

It follows from the equation satisfied by $\pi\circ du$ that every zero counts positively, so that $\wind_\pi(u)\geq 0$ with equality if, and only if, there are no zeros.

Now consider a non-vanishing section $\sigma$ of $u^*\xi$; such a section always exists. For every $z \in \Gamma$ we fix positive cylindrical coordinates $(s,t)$ at $z$ and set
\[
\wind_\infty(\util,z,\sigma) = \lim_{s\to+\infty} \wind(t \mapsto \pi \cdot \partial_su(s,\epsilon_z t), t \mapsto \sigma(s,\epsilon_z t)) \in \Z
\]
where $\epsilon_z=+1$ or $\epsilon_z=-1$ if $z$ is a positive or a negative puncture, respectively. The bundle $u(s,\cdot)^*\xi$ is oriented by $d\lambda$, and the relative winding number  is defined as in \S~\ref{sssec_rel_winding}. This limit is well-defined since, by Theorem~\ref{thm_precise_asymptotics}, $\pi \cdot \partial_su(s,t)$ does not vanish when $s\gg1$.

  {
\begin{remark}\label{rmk_wind_infty}
Let $P=(x,T)$ be the asymptotic limit of $\util$ at a puncture $z$. Even though $\sigma(s,\epsilon_z t)$ may not have any limit as $s\to+\infty$, it still induces a homotopy class $\beta_{\sigma,z}$ of oriented trivializations of $(x_T)^*\xi$. Then one can show that $\wind_\infty(\util,z,\sigma)$ is precisely the winding of the asymptotic eigenvector given by Theorem~\ref{thm_precise_asymptotics}, taken with respect to $\beta_{\sigma,z}$. In particular, if $z$ is a positive puncture then $\wind_\infty(\util,z,\sigma) \leq \wind^{<0}(A_P,\beta_{\sigma,z})$, and if $z$ is a negative puncture then $\wind_\infty(\util,z,\sigma) \geq \wind^{\geq0}(A_P,\beta_{\sigma,z})$.
\end{remark}
}

\begin{definition}[Hofer, Wysocki and Zehnder~\cite{props2}]\label{defn_wind_infty}
Assuming that $\pi\circ du$ does not vanish identically, we set
\begin{equation}
\wind_\infty(\util) = \sum_{z\in\Gamma^+} \wind_\infty(\util,z,\sigma) - \sum_{z\in\Gamma^-} \wind_\infty(\util,z,\sigma)
\end{equation}
where $\Gamma^+,\Gamma^-\subset\Gamma$ are the sets of positive and negative punctures, respectively. 
\end{definition}

Using standard degree theory one shows that $\wind_\infty(\util)$ does not depend on the choice of $\sigma$, in fact, we have the following statement.

\begin{lemma}[Hofer, Wysocki and Zehnder~\cite{props2}]\label{lemma_wind_relation}
The invariants $\wind_\pi(\util)$ and $\wind_\infty(\util)$ satisfy $\wind_\pi(\util) = \wind_\infty(\util) - \chi(S) + \#\Gamma$ when they are defined.
\end{lemma}

\begin{definition}[Fast planes]
A   {\textit{fast plane}} is a finite-energy plane $\util$ satisfying $\wind_\pi(\util)=0$.
\end{definition}

This notion of fast planes differs from the one discussed in~\cite{hryn}, in particular, we allow asymptotic orbits which are not prime.

\section{Topological lemmas and constructions}\label{section_topology}

Here we describe some topological constructions which are important to our arguments. In particular, we prove Lemma~\ref{lemma_embedded_strips} and its Corollary~\ref{cor_special_spanning_disk} as initial steps in the proof of our main results.

\subsection{Monodromy}\label{top_lemma_section}\label{sec:monodromy}

We introduce the notion of monodromy to describe the twisting of a multisection of a transverse plane bundle over a closed curve.  Such sections arise naturally in the presence of closed orbits and  $p$-disks, and we discuss these occurrences after presenting the general definition.

Consider a knot $K$ inside an oriented $3$-manifold $M$, and let $v: S^1 \simeq \R/\Z \rightarrow K$ be a degree $p$ covering map. The orientations of $K$ induced by $v$ and of $TM|_K$  together induce an orientation of the normal bundle $\nu K := TM|_K\slash TK$. We view $K$ as the image of an embedding $v_1 : S^1\rightarrow M$, which can be arranged to satisfy $v(t)=v_1(pt)$. Choose an oriented trivialization $\Psi : v_1^*(\nu K) \to S^1\times \C$ and write $\Psi_t : v_1^*(\nu K)|_t = \nu K|_{v_1(t)} \to \{t\}\times \C \simeq \C$ for the restriction of $\Psi$ to the fiber over $t\in S^1$. Any such $\Psi$ determines an oriented trivialization $\Psi^{(p)}: v^*(\nu K) \rightarrow  S^1\times \C$  defined by requiring $\Psi^{(p)}_t=\Psi_{pt}$. To any non-vanishing section $Z$ of $v^*(\nu K)$ we may associate a number in $\mathbb{Z}\slash p\mathbb{Z}$ as follows.

\begin{definition}\label{def:monodromy}
The smooth non-vanishing section $Z:S^1\rightarrow v^*(\nu K)$ is said to have \textit{monodromy} $q\in \{0,\dots,p-1\}$ if $\wind(t \mapsto \zeta(t)) \in q + p\Z$, where the smooth map $\zeta:S^1 \to \C\setminus \{0\}$ is defined by $\Psi^{(p)}_t \cdot Z(t) = \zeta(t)$.
\end{definition}

\begin{lemma}
The monodromy of a section is independent of the choice of oriented trivialization $\Psi$ used in the definition.
\end{lemma}

\begin{proof}
Suppose $\Psi$ and $\Phi$ are oriented trivializations of $v_1^*(\nu K)$. Then there is a unique $k\in\Z$ such that $\wind(t \mapsto \Phi_t \circ \Psi^{-1}_t \cdot u) = k$, $\forall u\in\C\setminus\{0\}$.  In particular we get $\wind(t \mapsto \Phi^{(p)}_t \circ (\Psi^{(p)}_t)^{-1} \cdot u) = pk$, $\forall u\in \C\setminus \{0\}$. Hence
\[
\begin{aligned}
\wind(\Phi_t^{(p)} \cdot Z(t)) & = \wind ( \Phi_t^{(p)}\circ (\Psi_t^{(p)})^{-1} \circ \Psi_t^{(p)} \cdot Z(t)) \\
& = \wind(\Psi_t^{(p)} \cdot Z(t)) + pk.
\end{aligned}
\]

\end{proof}

Suppose now that $K$ is an order $p$ rational unknot. Then a $p$-disk $u:\mathbb{D}\rightarrow M$ for $K$ induces a section with monodromy; the map $v$ above is the restriction of $u$ to $\partial \mathbb{D}$,   {and} the bundle $(u|_{\partial\D})^*\nu K \to \partial\D$ admits a non-vanishing section defined by 
\begin{equation}\label{section_normal_disk}
z \in \partial\D \mapsto [\partial_r u(z)] \in \nu K|_{u(z)},
\end{equation}
where $r$ denotes the radial coordinate on the punctured disk $\D\setminus\{0\}$.

\begin{definition}
We define the   {\textit{monodromy}} of the order $p$ rational unknot $K$ as the element of $\Z\slash p\Z$ determined by the monodromy of the section~\eqref{section_normal_disk}. \end{definition}

It is easy to see that the $p$-disk $z\mapsto u(\bar z)$ yields the same value of the monodromy.

At first one might imagine that the monodromy depends on the choice of  $p$-disk, but the next lemma shows that this is not the case.

\begin{lemma}\label{lemma_top}
Let $p$ be a positive integer and $K$ be a knot inside the oriented $3$-manifold $M$ admitting some $p$-disk. The monodromy of $K$ is independent of the choice of $p$-disk and has a multiplicative inverse in $\Z_p$. Moreover, if $p'\neq p$, $p'\geq 1$ then $K$ does not admit a $p'$-disk.
\end{lemma}

  {To prove Lemma~\ref{lemma_top}, we will consider a curve $\gamma$ in a neighborhood of $K$ which is constructed as a push-off of $K$ along the section~\eqref{section_normal_disk}; this allows us to compute the monodromy of $K$ as an intersection number between $\gamma$ and a $p$-disk.}


\begin{proof}[Proof of Lemma~\ref{lemma_top}]
Let $u:\mathbb{D}\rightarrow M$ be a $p$-disk for $K$. Consider a regular neighborhood of $K$ equipped with coordinates $(\theta,x+iy) \in \R/\Z \times \C$  such that $d\theta \wedge dx \wedge dy > 0$, $K=\R/\Z\times \{0\}$ and $u(e^{i2\pi t}) = (pt,0)$. For $r$ close enough to~$1$, $u$ maps $re^{i\phi}$ into this neighborhood and we can write components $u(re^{i2\pi t}) = (u^1(re^{i2\pi t}),u^2(re^{i2\pi t}))$ with respect to this coordinate system.

For $\epsilon$ small consider $\ncal := \{(\theta,x+iy) \in \R/\Z \times \C \mid \rho\leq \epsilon\}$ where $\rho^2=x^2+y^2$, and write $\dcal$ for the embedded open $2$-disk $u(\mathbb{D}\setminus\partial\D)$. Then $\dcal$ intersects $\partial\ncal$ transversally since $u$ is an immersion and $\epsilon$ is small. Next, consider the intersection $\partial \mathcal{N}\cap \mathcal{D}$.  For $r$ sufficiently close to $1$, this is a connected curve, as  $u^{-1}(\partial \mathcal{N}\cap \mathcal{D})$ is connected and parallel to $\partial \D$ in $\D$.   {We view $\partial \mathcal{N}\cap \mathcal{D}$ as a curve on $\partial \mathcal{N}$ and we denote this simple closed curve by  $\gamma$.}   

Free homotopy classes of closed curves in $\R/\Z \times (\C\setminus\{0\})$ can be identified with $\Z\times\Z$ via 
\[ 
(a,b) \mapsto [t\in\R/\Z \mapsto (at,e^{i2\pi bt})], 
\] 
and we refer to the unreduced fraction $\frac{b}{a}$ as the \textit{slope} of such a curve. 
Then, perhaps after rotating our coordinate system, we can assume that the slope of the curve $\dcal\cap\ncal$ is $\frac{q}{p}$, where $q \in \{0,\dots,p-1\}$ is the monodromy of $K$ computed with respect to the disk $u$, as explained above. Since $\gamma=\dcal\cap \partial \ncal$ is connected, we conclude that $q$ and $p$ are relatively prime. In particular, $q$ has a multiplicative inverse in $\Z_p$.

Let $u'$ be a $p'$-disk for $K$. As before, there is no loss of generality to assume that $\dcal' = u'(\D\setminus\partial\D)$ intersects $\partial\ncal$ transversally and $\gamma' = \dcal' \cap \partial\ncal$ is a connected curve with slope $\frac{q'}{p'}$ for some $q'\in\Z$. As before $p'$ and $q'$ must be relatively prime. We now use the intersection pairing $H_1(M\setminus K) \otimes H_2(M,K) \rightarrow \mathbb{Z}$ to show that $u$ and $u'$ have the same slope.  
We evaluate $[\gamma] \cdot [u']$ in two ways. First, observe that the algebraic,  i.e., signed, intersection number between $\gamma$ and $u'$ is the same as the algebraic intersection number of $\gamma$ and $\gamma'$ on $\partial \mathcal{N}$. Standard results on torus knots imply that $ \gamma\cdot\gamma'=\pm|p'q-pq'|$. On the other hand, $[\gamma]=0\in H_1(M\setminus K)$, since $\gamma$ bounds a subdisc of $\dcal$.  It follows that the intersection pairing is $0$, so $\gamma$ and $\gamma'$ are parallel.  Thus the slope $\frac{q}{p}=\frac{q'}{p'}$ is uniquely determined by the rational unknot. In particular, $p=p'$ and $q=q'$.
\end{proof}

From now on we will denote the monodromy of an order $p$ rational unknot $K$ by
\begin{equation}\label{mon_notation}
\mon(K) \in \Z\slash p\Z.
\end{equation}

Let $\Pi:S^3\rightarrow L(p,q) = S^3/\Z_p$ be the quotient map. 

\begin{lemma}\label{lem:lpqmonodromy} Let $K:=\Pi\{ (e^{i\theta_1},0)\in S^3 \}$.  Then $\mon (K)=-q$.  
\end{lemma}

\begin{proof}[Proof of Lemma~\ref{lem:lpqmonodromy}] We exploit the presentation of the monodromy as a slope which was introduced in the preceding proof of Lemma~\ref{lemma_top}.  

Consider the decomposition of $S^3=\{ (r_1e^{i\theta_1}, r_2e^{i\theta_2}) \ | \ r_1^2+r_2^2=1\}$ into the solid tori $\widehat{H_1}$ and $\widehat{H_2}$ defined by the inequalities $r_i\leq \frac{1}{\sqrt{2}}$.  Within $\widehat{H_1}$, let $\widehat{\D_1}$ denote the $p$-tuple of meridional disks defined by setting $\theta_2=\frac{2\pi k}{p}$, for $k=0,1, \dots, p-1$. Similarly, let $\widehat{\D_2}$ denote the $p$-tuple of disks in $\widehat{H_2}$ defined by setting $\theta_1=\frac{2\pi k}{p}$, for $k=0,1, \dots, p-1$. Observe that $\partial \widehat{\D_1} \cap \partial \widehat{\D_2}$ consists of  $p^2$ points on $\partial \widehat {H_1}=\partial \widehat{H_2}$. 
 
In $L(p,q)$,  $\Pi(\widehat{H_i}, \widehat{\D_i})$  is  a solid torus $H_i$ with a preferred meridional disk $\D_i$; this recovers the familiar fact that lens spaces admit genus one Heegaard splittings.    The disks $\D_1$ and $\D_2$ intersect in $p$ points, and $K$ intersects $\D_2$ once transversely.  It follows that the cone of $\D_1$ over $K$ is a  $p$-disk for $K$ in $L(p,q)$.

We can read off the monodromy of $K$ from the  slope of $\partial \D_1$ on $\partial H_2$.  First, observe that  $\{ (\frac{1}{\sqrt{2}}e^{i\theta_1},\frac{1}{\sqrt{2}}e^{i\theta_2}) : (\theta_1,\theta_2)\in[0, \frac{2\pi}{p}) \times [0,2\pi) \} \subset \partial \widehat{H_2}$ is a fundamental domain for the $\mathbb{Z}\slash p\mathbb{Z}$ action, so we may identify its image under $\Pi$ with $\partial H_2$.  The intersection of $\partial \widehat{\D_1}$ with this rectangle consists of $p$  horizontal segments.  In the quotient, the edges of the rectangle are identified via 
\[ (t, 0) \sim (t,2\pi) \ \ \  and  \ \ \ (0,t) \sim \big(\frac{2\pi}{p},  {t+} \frac{2q\pi}{p}\big).\]

After rescaling the $\theta_1$ dimension to $2\pi$, we see that $\partial \D_1$ is a connected curve with slope $-q$ on $\partial H_2$, as desired.  
\end{proof}

\begin{remark}
An analogous argument shows that for $K'= \Pi\{ (0, e^{i\theta_2})\in S^3 \}$, $\mon(K')=-q'$ for $qq'\equiv 1 \mod p$.  As noted in the introduction, $L(p,q)$ is homeomorphic to $L(p,q')$.  
\end{remark}

Before continuing, we remark on some immediate consequences of the Heegaard splitting $H_1\cup H_2$  constructed above.  Namely, if $K\subset M$ is an order $p$ rational unknot with $\mon(K)=-q$, then $M=L(p,q)\# M'$.  This is a standard fact in three-manifold topology, but we recall the proof here as part of our proof of  Proposition~\ref{proposition_classification}. 

\begin{proof}[Proof of Proposition~\ref{proposition_classification}] 
In the notation introduced above,  the complement of $\D_1$ in $H_1$ is a three-ball, so it follows that the lens space is  determined  up to homeomorphism by $H_2\cup \D_1$. Thus, a regular neighborhood of any $p$-disk in a three-manifold is $L(p,r)\setminus B^3$ for some $r$.  Combining Lemmas~\ref{lemma_top} and \ref{lem:lpqmonodromy}, we see immediately that $L(p,r)$ is homeomorphic to $L(p,q)$.  

Furthermore, we see that any manifold admitting   {a rational} open book decomposition with disk-like pages is in fact a lens space.  Given such a decomposition $(K, \pi)$, $M\setminus K$ is the mapping torus of a disk and is therefore a solid torus.  Letting $\mathcal{N}(K)$ denote a regular neighborhood of $K$, we have the genus one Heegaard splitting   $M=\mathcal{N}(K) \cup \big( (M\setminus K) \cap (M \setminus \mathcal{N}(K))\big)$, so $M$ is a lens space.  

 It remains to show that if the Reeb vector field of a defining contact form $\lambda$ is positively tangent to $K$ and positively transverse to the interior of the pages, then  $(M, \xi)$ is contactomorphic to $(L(p,q), \xi_\text{std})$.  However, this is equivalent to the statement that the rational open book supports the contact structure $\xi$, and Baker-Etnyre-Van Horn-Morris  \cite{BEVHM} have shown that this implies $\xi$ is universally tight.   
\end{proof}

\begin{lemma}\label{lemma_no_p'_disk}
Let  $K$ be an order $p$ rational unknot. If $1\leq p'<p$ then the $p'$-th iterate of $K$ is not contractible in $M$.
\end{lemma}

\begin{proof}
Let $u:\D \rightarrow M$ be a $p$-disk for $K$.  As noted above, a regular neighborhood of $K\cup u(\D)$ is the complement of a ball in $L(p,q)$ for some $q$, where $K$ is the core of a solid torus in a genus one Heegaard splitting for the lens space.  The cores of these solid tori have order $p$ in $\pi_1(L(p,q))$.  Since $M=L(p,q)\#M'$ for some three-manifold $M'$, it follows that $\pi_1(M)=\pi_1(L(p,q))*\pi_1(M')$.  Thus $K$ has order $p$ in $\pi_1(M)$.  A disk for the $p'$-th iterate of $K$ would provide a homotopy from $K^{p'}$ to a point, so no such disk can exist for $p'<p$.
\end{proof}

\subsection{Embedded strips and special spanning $p$-disks}\label{ssec_transv_strips}

For the present discussion we fix a contact form $\lambda$ on the $3$-manifold $M$, which defines a contact structure $\xi = \ker\lambda$ and a Reeb vector field $R$. Let us consider a closed orbit $P=(x,T)$ of the Reeb flow. Write $T=pT_{\rm min}$ where $p \in \Z^+$ and $T_{\rm min}$ is smallest positive period of $x$. As described in the previous section, an oriented trivialization $\Psi$ of $(x_{T_{\rm min}})^*\xi$ induces an oriented trivialization $\Psi^{(p)}$ of $(x_T)^*\xi$, called the $p$-th iterate of $\Psi$.  Here $x_T$ is the map~\eqref{map_x_T}.

The following statement will be used later in the construction of certain special pseudo-holomorphic curves. We denote by $$ \pi : TM \to \xi $$ the projection onto $\xi$ along $R$.

\begin{lemma}\label{lemma_embedded_strips}
Let $Z$ be a non-vanishing section of $(x_T)^*\xi$ with monodromy $q \in \{0,\dots,p-1\}$ satisfying $\gcd(q,p)=1$, and let $\beta$ be the homotopy class of oriented trivializations of $((x_T)^*\xi,d\lambda)$ induced by $Z$. Fix some $d\lambda$-compatible complex structure $J$ on $\xi$ and let $A_P$ be the corresponding asymptotic operator at $P$. If $\exists\nu \in \sigma(A_P) \setminus \{0\}$ with $\wind(\nu,\beta) = 0$ then one finds $0<r_0<1$ and an immersion $$ \psi:(r_0,1]\times\R/\Z \to M $$ satisfying
\begin{itemize}
\item $\psi(1,t) = x_T(t)$ $\forall t$, and $\psi|_{(r_0,1)\times\R/\Z}$ is an embedding,
\item $\wind(t\mapsto \pi \cdot \partial_r\psi(1,t),t\mapsto Z(t))=0$, and
\item For every sequence $\lambda_k = h_k\lambda$ of defining contact forms for $(M,\xi)$ such that $h_k\to1$ in $C^\infty$ as $k\to\infty$, $h_k|_{x(\R)}\equiv1$ and $dh_k|_{x(\R)}\equiv0$ $\forall k$, the Reeb vector field $R_k$ of $\lambda_k$ is transverse to $\psi((r_0,1)\times\R/\Z)$ for all $k$ large enough.
\end{itemize}
\end{lemma}

\begin{proof}
Without loss of generality we assume that the minimal positive period $T_{\rm min}$ of $x$ is equal to $1$, so that $T = p$. By our assumptions on the functions $h_k$, it follows that $\lambda_k$, $d\lambda_k$ and $R_k$ agree with $\lambda$, $d\lambda$ and $R$ over $x(\R)$, respectively, $\forall k$. Therefore $t \in \R \mapsto x(t) \in M$ is a trajectory for the flow of $R_k$ with minimal positive period~$1$,~$\forall k$.

Consider a tubular neighborhood $U \subset M$ of $x(\R)$ and a diffeomorphism $U \simeq \R/\Z \times B$, where $B \subset \R^2$ is an open ball centered at the origin. Equipping $\R/\Z \times B$ with coordinates $(\theta,x_1,x_2)$, we can construct this diffeomorphism in such a way that $x(T_{\rm min}t) = x(t) \simeq (t,0,0)$ $\forall t\in\R/\Z$, $\lambda|_{\R/\Z\times\{(0,0)\}} \simeq d\theta$, $d\lambda|_{\R/\Z\times\{(0,0)\}} \simeq dx_1\wedge dx_2$ and $J \cdot \partial_{x_1} = \partial_{x_2}$ along $\R/\Z\times\{(0,0)\}$. Hence $R|_{\R/\Z\times\{(0,0)\}} \simeq \partial_\theta$, $\lambda_k|_{\R/\Z\times\{(0,0)\}} \simeq d\theta$, $d\lambda_k|_{\R/\Z\times\{(0,0)\}} \simeq dx_1\wedge dx_2$ and ${R_k}|_{\R/\Z\times\{(0,0)\}} \simeq \partial_\theta$ for every $k$. Moreover, $\{\partial_{x_1}|_{(t,0,0)},\partial_{x_2}|_{(t,0,0)}\}$ is a unitary frame of $(x_{T_{\rm min}})^*\xi$ inducing a trivialization $\Psi$ of $(x_{T_{\rm min}})^*\xi$ by 
\[
\Psi_t \cdot (a\partial_{x_1}|_{(t,0,0)} + b\partial_{x_2}|_{(t,0,0)}) = (a,b).
\]

Let us denote by $\phi^k_t$ the flow of $R_k$ and by $\phi_t$ the flow of $R$. Since Reeb flows preserve corresponding contact forms, the transverse linearized flows give paths of symplectic maps $d\phi^k_{t},d\phi_t:\xi|_{x(0)} \to \xi|_{x(t)}$, represented by the smooth paths $\varphi_k,\varphi : \R \to Sp(2)$
\[
\begin{array}{cc}
\varphi_k(t) = \Psi_t \circ d\phi^k_t|_{x(0)} \circ (\Psi_0)^{-1}, & \varphi(t) = \Psi_t \circ d\phi_t|_{x(0)} \circ (\Psi_0)^{-1}.
\end{array}
\]
We have $\varphi,\varphi_k \in \Sigma$, see \S~\ref{sssec_analytical_defn}. Hence 
\[
\begin{array}{cc}
\varphi_k^{(p)}(t) := \varphi_k(pt), & \varphi^{(p)}(t) := \varphi(pt)
\end{array}
\]
also belong to $\Sigma$ and, in our coordinates, the total linearized flows along the orbit $P$ are represented by a smooth paths of $3\times3$ matrices written in blocks as
\begin{equation}\label{rep_lin_flow}
\begin{aligned}
& d\phi^k_{Tt}|_{x(0)} = d\phi^k_{pt}|_{x(0)} \simeq \begin{pmatrix} 1 & 0 \\ 0 & \varphi_k^{(p)}(t) \end{pmatrix}, \\
& d\phi_{Tt}|_{x(0)} = d\phi_{pt}|_{x(0)} \simeq \begin{pmatrix} 1 & 0 \\ 0 & \varphi^{(p)}(t) \end{pmatrix}.
\end{aligned}
\end{equation}
Let $S_k,S : \R/\Z \to \R^{2\times 2} \simeq \mathcal L_\R(\C)$ be smooth paths of symmetric matrices
\[
\begin{aligned}
& S = -i\frac{d\varphi^{(p)}}{dt}(\varphi^{(p)})^{-1} = -ip\dot\varphi(pt)\varphi^{-1}(pt), \\
& S_k = -i\frac{d\varphi_k^{(p)}}{dt}(\varphi_k^{(p)})^{-1} = -ip\dot\varphi_k(pt)\varphi^{-1}_k(pt).
\end{aligned}
\]
Then the asymptotic operator $A_P$~\eqref{defn_asymp_operator} associated to the orbit $P$ and the contact form $\lambda$ is represented in the $p$-th iterate $\Psi^{(p)}$ of $\Psi$ as the operator $L_S = -i\partial_t-S$.

Let $\zeta$ be an eigenvector in $\ker (A_P-\nu I)$, where $\nu \in \sigma(A_P)\setminus\{0\} = \sigma(L_S)\setminus\{0\}$ satisfies $\wind(\nu,\beta) = 0$. In the trivialization $\Psi^{(p)}$ it gets represented as an element $t \in \R/\Z \mapsto n(t) = (n_1(t),n_2(t)) \in \R^2$ of $\ker(L_S-\nu I)$ with coordinates determined by $$ \zeta(t) = n_1(t) \partial_{x_1}|_{(pt,0,0)} + n_2(t) \partial_{x_2}|_{(pt,0,0)}. $$ Finally we define
\begin{equation}\label{map_psi}
\psi(r,t) = (pt,(1-r)n(t)) \in \R/\Z \times B
\end{equation}
for $r \in (r_0,1]$ and $t\in\R/\Z$, where $r_0<1$ is close to $1$. We will now verify that $\psi$ satisfies the required properties.

The frame $\{\partial_\theta,\partial_{x_1},\partial_{x_2}\}$ trivializes $T(\R/\Z\times B)$ and represents the vector fields $R = R(\theta,x_1,x_2)$, $R_k = R_k(\theta,x_1,x_2)$ as vectors in $\R^3$. Also, the flows $\phi_t,\phi^k_t$ get represented in these coordinates for an arbitrarily large time on a small neighborhood of $\R/\Z\times\{(0,0)\}$. In view of~\eqref{rep_lin_flow} we can evaluate identities 
\[
\frac{d}{dt}d\phi_t = (DR \circ \phi_t)d\phi_t, \ \ \ \frac{d}{dt}d\phi^k_t = (DR_k \circ \phi^k_t)d\phi^k_t
\]
at the point $(0,0,0)$ and at time $t=\theta$ ($T_{\rm min}=1$) to get
\begin{equation}
\begin{aligned}
& DR(\theta,0,0) = \begin{pmatrix} 0 & 0 \\ 0 & \dot\varphi(\theta)\varphi^{-1}(\theta) \end{pmatrix} \\
& DR_k(\theta,0,0) = \begin{pmatrix} 0 & 0 \\ 0 & \dot\varphi_k(\theta)\varphi_k^{-1}(\theta) \end{pmatrix}
\end{aligned}
\end{equation}
in blocks. Note that $\dot\varphi\varphi^{-1}$, $\dot\varphi_k\varphi_k^{-1}$ are $1$-periodic. Using Taylor's expansion
\begin{equation}\label{taylor_Reeb}
\begin{aligned}
& R_k(\psi(r,t)) \\
& = R_k(pt,0,0) + (r-1)DR_k(pt,0,0) \cdot \partial_r\psi(1,t) + O(|1-r|^2) \\
& = \begin{pmatrix} 1 \\ 0 \end{pmatrix} + (1-r) \begin{pmatrix} 0 & 0 \\ 0 & i\frac{1}{p}S_k(t) \end{pmatrix} \begin{pmatrix} 0 \\ n(t) \end{pmatrix} + O(|1-r|^2) \\
& = \begin{pmatrix} 1 \\ (1-r)i\frac{1}{p}S_k(t)n(t) \end{pmatrix} + O(|1-r|^2).
\end{aligned}
\end{equation}
Since $\lambda_k \to \lambda$ in $C^\infty$, we have that $S_k \to S$ in $C^\infty$ and that the absolute value of the terms $O(|1-r|^2)$ in the above expressions are bounded by $C|1-r|^2$, for some $C>0$ independent of $k$. Therefore we can compute
\begin{equation}\label{est_twist}
\begin{aligned}
& \det(\partial_r\psi(r,t),\partial_t\psi(r,t),R_k(\psi(r,t)) \\
& = \det \begin{pmatrix} 0 & p & 1 \\ -n_1 & (1-r)\dot n_1 & (1-r)\frac{1}{p}(iS_kn)_1 \\ -n_2 & (1-r)\dot n_2 & (1-r)\frac{1}{p}(iS_kn)_2 \end{pmatrix} + O(|1-r|^2) \\
& = (1-r)\det(n,iS_kn) - (1-r)\det(n,\dot n) + O(|1-r|^2) \\
& = (1-r)\det(n,iS_kn-\dot n) + O(|1-r|^2) \\
& = (1-r)(\det(n,-i\nu n) + \det(n,i[S_k-S]n)) + O(|1-r|^2) \\
& = (1-r)(\det(n,i[S_k-S]n) - \nu|n|^2) + O(|1-r|^2).
\end{aligned}
\end{equation}
In the second line we wrote $(iS_kn)_1,(iS_kn)_2$ for the components of $iS_kn \in \R^2$ and $\det$ is the usual determinant of $3\times3$ matrices. In the first, third, fourth, fifth and sixth lines $\det$ denotes the determinant as a multilinear function. In the fifth line we used that $n \in \ker (L_S - \nu)$, which gives 
\[
iSn-\dot n = -i\nu n \Rightarrow iS_kn-\dot n = i(S_k-S)n-i\nu n.
\]
Since $n(t)$ never vanishes and $\nu\neq0$ and 
\[
|\det(n,i(S_k-S)n)| \leq \|S_k-S\||n|^2
\]
we get from~\eqref{est_twist} that $\{\partial_r\psi(r,t),\partial_t\psi(r,t),R_k(\psi(r,t)\}$ is linearly independent for every $(r,t) \in (r_0,1)\times \R/\Z$ when we set $r_0$ close enough to $1$ and large $k$. In other words, $R_k$ is transverse to the immersed strip $\psi((r_0,1)\times\R/\Z)$ when $k$ is large enough.

Let us represent $Z(t)$ by $\vec Z(t) = (Z_1(t),Z_2(t))$ with coordinates $$ Z(t) = Z_1(t) \partial_{x_1}|_{(pt,0,0)} + Z_2(t) \partial_{x_2}|_{(pt,0,0)} $$ and write $W(t) = \partial_{x_1}|_{(pt,0,0)}$. Since $Z$ has monodromy $q$ there exists $m \in \Z$ such that $\wind(Z,W) = \wind(\vec Z) = q + pm$. We claim that $\wind(n) = q+pm$. In fact,
\[
\begin{aligned}
0 = \wind(\nu,\beta) & = \wind(\zeta,Z) \\
& = \wind(\zeta,W) - \wind(Z,W) \\
& = \wind(n) - q - pm.
\end{aligned}
\]
It follows that $\wind(t\mapsto\partial_r\psi(1,t),Z(t))=0$.

Let $(r_1,t_1),(r_2,t_2) \in (r_0,1) \times \R/\Z$ satisfy $\psi(r_1,t_1) = \psi(r_2,t_2)$. Assume, by contradiction, that $t_1,t_2$ are distinct points of $\R/\Z$. Perhaps after interchanging $t_1$ and $t_2$ we can represent $t_1,t_2$ as numbers $0\leq t_1<t_2<1$. Since $pt_1 \in pt_2 + \Z$ we must have $t_2 = t_1+k/p$ for some $k \in \{1,\dots,p-1\}$. Moreover, there exists $c>0$ such that $n(t_2) = cn(t_1)$. These facts follow from the expression~\eqref{map_psi} for~$\psi$. We can see $n(t)$ as a $1$-periodic $\R^2$-valued function defined on the entire real line. Let $\vartheta : \R \to \R$ be a smooth function satisfying $n(t) \in \R^+ e^{i\vartheta(t)} \ \forall t\in\R$. Then $\vartheta(t+1) = \vartheta(t) + 2\pi(q + pm) \ \forall t\in\R$. Note that $n$ solves a linear ODE with coefficients which are $1/p$-periodic. Then $n(t_2 = t_1 + k/p) = cn(t_1)$ implies that $n(t+k/p) = cn(t) \ \forall t\in\R$. Then there is a well-defined integer $l \in \Z$ given by
\[
l = \frac{\vartheta(t+k/p)-\vartheta(t)}{2\pi}, \ \ \forall t\in\R.
\]
We obtain
\[
\begin{aligned}
k(q+pm) & = \frac{1}{2\pi}(\vartheta(t+k) - \vartheta(t)) \\
& = \frac{1}{2\pi} \sum_{j=0}^{p-1} \vartheta \left(t+\frac{k(j+1)}{p}\right) - \vartheta\left(t+\frac{kj}{p}\right) \\
& = pl.
\end{aligned}
\]
This implies that $kq = p(l-km)$, proving that $p$ divides $kq$. Since $q$ and $p$ are relatively prime, $p$ must divide $k$ which is impossible. This contradiction shows that $t_1 \equiv t_2 \ {\rm mod} \ \Z$, forcing $r_2 = r_1$. Therefore $\psi|_{(r_0,1)\times\R/\Z}$ is an embedding.
\end{proof}

In the statement below, $r$ denotes the radial coordinate on $\D\setminus \{0\}$.

\begin{corollary}\label{cor_special_spanning_disk}
Orient the knot $x(\R)$ by $\lambda$ and assume that it admits an oriented $p$-disk $u : \D \to M$. Consider the non-vanishing section $Z(t) = \pi \cdot \partial_ru(e^{i2\pi t})$ of $(x_T)^*\xi$ and denote by $\beta_u \in \Omega^+_{(x_T)^*\xi}$ the homotopy class induced by~$Z$. If $\rho(P,\beta_u) \neq 0$ then $u$ can be slightly modified in the $C^0$-topology near $\partial\D$ into a new $p$-disk $u'$ for $x(\R)$ satisfying the following property: $\exists \ \epsilon>0$ such that for every sequence of smooth functions $h_k : M \to (0,+\infty)$ satisfying $h_k \to 1 \ \text{in} \ C^\infty_{\rm loc}$, $h_k|_{x(\R)}\equiv1$ and $dh_k|_{x(\R)} \equiv 0 \ \forall k$, one can find $k_0$ such that 
\[
1-\epsilon < |z| < 1,\ k\geq k_0 \Rightarrow R_k|_{u'(z)} \not\in du'_z(T_z\D).
\]
Here we denoted by $R_k$ the Reeb vector field associated to the contact form $h_k\lambda$.
\end{corollary}

\begin{proof}
Note that if $\rho((x,T),\beta_u)>0$ then the eigenvalues of the asymptotic operator $A_{(x,T)}$ with zero winding with respect to $\beta_u$ are strictly negative, and if $\rho((x,T),\beta_u)<0$ then the eigenvalues of $A_{(x,T)}$ with zero winding with respect to $\beta_u$ are strictly positive. Thus we can apply Lemma~\ref{lemma_embedded_strips} with $Z(t)$ to obtain a special embedded strip $\psi(r,t)$, $(r,t)\in(1-\delta,1]\times \R/\Z$. Since $\pi \cdot \partial_r\psi(1,t)$ does not wind with respect to $Z$ we can cut out a neighborhood of the boundary of $u$ and patch what is left with the strip $\psi$ to obtain the desried $p$-disk $u'$. In this cut-and-paste process we make use of Dehn's lemma.
\end{proof}

\subsection{Self-linking number in terms of winding numbers}\label{sec:selflink}

 Let $(M,\xi)$ be a co-oriented contact $3$-manifold.

\begin{lemma}\label{lemma_self_link_props}
Let the knot $K \subset M$ be $p$-unknotted, transverse to $\xi$, and oriented by the co-orientation of $\xi$. Let $u:\D\to M$ be an oriented $p$-disk for $K$, $Z$ be a non-vanishing section of $u^*\xi$, and $N$ be a non-vanishing section of $u^*\xi|_{\partial \D}$ satisfying $N(z) \in du_z(T_z\D) \ \forall z\in\partial \D$. Then $\sl(K,u) =   {\frac{1}{p}}\ \wind(Z|_{\partial\D},N)$. 
\end{lemma}

Winding numbers are computed with respect to the orientation of $\xi$ induced by its co-orientation.

\begin{proof}
Set $n=\wind(Z|_{\partial\D},N)$ and consider coordinates $(\theta,w_1+iw_2) \in \R/\Z \times B$ on a small tubular neighborhood of $K$, where $B\subset\C$ is a small open ball centered at the origin, such that $K \simeq \R/\Z\times 0$, $\xi|_x \simeq 0\times \C \ \forall x\in K$, $d\theta\wedge dw_1\wedge dw_2$ induces the same orientation as $\xi$, and  $u(e^{i2\pi t})\simeq (pt,0)$.

Rotating the coordinate system along $w_1+iw_2$ and deforming $u$ we may assume that $u(re^{i2\pi t})\simeq (pt,(1-r)e^{i2\pi qt})$ where $q\in\{0,\dots,p-1\}$ represents $\mon(K)$ and $|1-r|$ is small. There is no loss of generality to assume that $Z(t) \simeq (0,e^{i2\pi (q+n)t})$ over the point $(pt,0)$. Exponentiating a small multiple of this vector we obtain a loop homotopic to $\gamma(t)=(pt,\epsilon e^{i2\pi (q+n)t})$ on $M\setminus K$. Thus $\sl(K,u)$ is   {$\frac{1}{p^2}$ times} the oriented intersection number of $\gamma$ and $u$, which is equal to $$ \sl(K,u) =   {\frac{\pm1}{p^2}} \# \left\{ (t_0,t_1) \in \R/\Z \times\R/\Z \left| \begin{aligned} pt_0=pt_1 &\mod 1 \\ (q+n)t_0=qt_1 &\mod 1 \end{aligned} \right. \right\}. $$ The first equation cuts out $p$ disjoint circles inside $\R/\Z \times\R/\Z$ homotopic to the diagonal. The second equation cuts out $d:=\gcd(q+n,q)=\gcd(n,q)$ disjoint circles, all of which intersect the diagonal $(q+n)/d-q/d=n/d$ times. We get $\sl(K,u)=  {\frac{1}{p^2}}p d \frac{n}{d}=  {\frac{n}{p}}$.
\end{proof}

An immediate consequence of Lemma~\ref{lemma_self_link_props} is the following statement.

\begin{corollary}\label{cor_rel_windings}
Let $\lambda$ be a defining contact form for $\xi$, and let $x$ be a periodic Reeb trajectory with minimal period $T_{\rm min}>0$ such that $x(\R)$ is an order $p$ rational unknot. Orienting $x(\R)$ by $\lambda$, consider an oriented $p$-disk $u$ for $x(\R)$. Setting $T=pT_{\rm min}$ there is no loss of generality to assume that $u(e^{i2\pi t}) = x_T(t)$. Let $\beta_u \in \Omega^+_{(x_T)^*\xi}$ be induced by a non-vanishing section $N$ satisfying $$ N(t) \in du_{e^{i2\pi t}}(T_{e^{i2\pi t}}\D) \cap \xi|_{u(e^{i2\pi t})}, \ \ \forall t\in \R/\Z $$ and let $\beta_{\rm disk} \in \Omega^+_{(x_T)^*\xi}$ be induced by a trivialization of $u^*\xi$. If $\sl(x(\R),u)=  {\frac{-1}{p}}$ then $\wind(\beta_u,\beta_{\rm disk})=+1$. In particular we get $\rho((x,T),\beta_u)+1=\rho((x,T),\beta_{\rm disk})$.
\end{corollary}

We conclude this section with an application of the results above to the specific case of computing the self-linking number of an order $p$ rational unknot in $L(p,q)$.  As in Section~\ref{sec:monodromy}, let $\widetilde{K}=\{ (z,w) \in S^3\ | \ |z|=1\}$ and denote the image of $\widetilde{K}$ in $L(p,q)$ by $K$.  Recall that we let $\lambda_0$ denote both the Liouville form  on $S^3$ and its image in $L(p,q)$; we note for later use that  $\lambda_0$ is dynamically convex in $L(p,q)$.

\begin{lemma}\label{lem:lpqsl} The knot $K$ is tangent to the Reeb vector field of $\lambda_0$ and has self-linking number $  {\frac{-1}{p}}$.
\end{lemma}

\begin{proof}
We begin by introducing some notation that will allow us to explicitly parameterize a $p$-disk for $K$.   Define the map $P : \C^2 \setminus 0 \to S^3$ by
\begin{equation}
(z,w) \mapsto \frac{(z,w)}{\sqrt{|z|^2 + |w|^2}}
\end{equation}
and once again denote the quotient projection by $\Pi:S^3 \to L(p,q).$

One sees easily that $\widetilde{K}$ is unknotted and has self-linking number $-1$ since it bounds the embedded disk $D_1=\overline{P(\C\times\{1\})}$.  The disk $D_1$ is the image of the map $\widetilde{u}$ defined by
\[
re^{i\theta} \in \D \mapsto (f(r)e^{i\theta},\sqrt{1-f(r)^2}) \in S^3
\]
where $f:[0,1]\to\R$ is smooth, $f'>0$, $f(r)=r$ near $0$, $f(r) = \cos \frac{\pi}{2}(1-r)$ near $1$. Away from its boundary, $D_1$ is transverse to the Reeb vector field of $\lambda_0$, and $\widetilde{u}$ is a disk-like global surface of section with respect to the standard Reeb flow on $S^3$.

We construct a $p$-disk for $K$ whose image is $\Pi(D_1\setminus\partial D_1)$. We can define such a map $u : \D \to L(p,q)$ by
\[
 re^{i\theta} \mapsto \Pi(f(r)e^{i\theta},\sqrt{1-f(r)^2}).
\]
 Orienting $K$ by the Reeb vector field $R_0$ of $\lambda_0$ on $L(p,q)$, this $p$-disk becomes oriented as in Definition~\ref{p_unknot_def}. Moreover, $u$ is a disk-like global surface of section of order $p$ for the flow of $R_0$.

To compute the self-linking number of $K$, note that 
\[
(z,w) \in S^3 \mapsto W(z,w) := (-\bar w,\bar z) \in \xi_{\rm std}|_{(z,w)}
\]
is a global non-vanishing section of the standard contact structure of $S^3$. Since $\Pi \circ \widetilde u = u$ we have that $z\in \D \mapsto X(z) := d\Pi \cdot W(\widetilde u(z))$ is a smooth non-vanishing section of $u^*\xi_{\rm std}$. We have
\[
\begin{aligned}
& \wind(t \in \R/\Z \mapsto X(e^{i2\pi t}), t \in \R/\Z \mapsto \partial_ru(e^{i2\pi t})) \\
& = \wind(t \in \R/\Z \mapsto e^{-i2\pi t},\ \text{constant vector}) = -1.
\end{aligned}
\]
By Lemma~\ref{lemma_self_link_props} we obtain $\sl(K_1,u) =   {\frac{-1}{p}}$. 
\end{proof}

\begin{remark}
Reversing the roles of the two factors of $\mathbb{C}^2$, let $K'$ denote the knot $P(S^3\cap(\{0\}\times \mathbb{C}))$. The reader may similarly compute $ \sl(K') =   {\frac{-1}{p}}$.
\end{remark}

\section{Bubbling-off analysis}\label{section_bubb_off_analysis}

In this section $(M,\xi)$ is a closed co-oriented contact $3$-manifold such that $c_1(\xi)$ vanishes on $\pi_2(M)$. We fix a nondegenerate contact form $\lambda$ which defines $\xi$ and induces the given co-orientation. We will call a number $T>0$ a period if there exists a $T$-periodic closed Reeb orbit.

\subsection{Germinating sequences}

Fix $C>0$ and choose a number $\sigma(C)$ such that
\begin{equation}\label{eqsigC} 
0< \sigma(C) < \inf \{T',\ |T'-T''|\} 
\end{equation} 
where the infimum is taken over all periods $T',T''$ satisfying $\max\{T',T''\}\leq C$ and $T' \neq T''$. The choice of $\sigma(C)$ is always possible since $\lambda$ is assumed to be nondegenerate. Throughout we fix $J \in \J_+(\xi)$ arbitrarily and consider sequences of $\jtil$-holomorphic maps $$ \vtil_n=(b_n,v_n):B_{R_n}(0) \subset \C \to \R \times M $$ with $R_n \to \infty$, satisfying 
\begin{equation}\label{ineq1} 
E(\vtil_n)\leq C,   \hspace{1cm}  \int_{B_{R_n}(0) \setminus \D} v_n^* d\lambda \leq \sigma(C),
\end{equation} 
for all  $n$, and 
\begin{equation}\label{ineq2} 
\{b_n(2)\} \mbox{ is uniformly bounded.} 
\end{equation}
Such a sequence $\vtil_n$ of $\jtil$-holomorphic maps satisfying \eqref{ineq1} and \eqref{ineq2} will be referred to as a {\bf germinating sequence}.

\begin{proposition}\label{proplimit}
Let $\vtil_n$ be a germinating sequence. Then there exists a finite set $\Gamma \subset \D$, a $\jtil$-holomorphic map $\vtil=(b,v):\C \setminus \Gamma  \to \R \times M$ and a subsequence of $\vtil_n$, still denoted by $\vtil_n$, such that $\vtil_n \to \vtil$ in $C^\infty_{\rm loc}(\C \setminus \Gamma, \R \times M)$. Moreover, $E(\vtil) \leq C$. \end{proposition}

\begin{proof}
Let $\Gamma_0\subset \C$ be the set of points $z\in \C$ such that there exists a subsequence $\vtil_{n_j}$ and points $\zeta_j \in B_{R_{n_j}}(0)$ satisfying $\zeta_j \to z$ and $|d \vtil_{n_j}(\zeta_j)| \to \infty$ as $j\to\infty$.  If $\Gamma_0 = \emptyset$ then from usual elliptic estimates  {, see~\cite[chapter~4]{mcdsal},} and~\eqref{ineq2} we get $C^\infty_{\text{loc}}$-bounds for $\vtil_n$ and find a $\jtil$-holomorphic map $\vtil:\C \to \R \times M$ so that, up to a subsequence, $\vtil_n \to \vtil$ in $C^\infty_{\text{loc}}(\C,\R \times M)$. In this case $\Gamma = \emptyset$.

By results from~\cite{93}, if $\Gamma_0 \neq \emptyset$ and $z_0\in \Gamma_0$ then we find a period $T_0>0$ and sequences $r_j \to 0^+$, $n_j\to\infty$ such that $$ \int_{B_{r_j}(z_0)} v_{n_j}^* d \lambda \to T_0 \ \mbox{ as } \ j\to \infty. $$ We consider $\vtil_{n_j}$ as the new sequence $\vtil_n$. Now let $\Gamma_1\subset \C \setminus \{z_0\}$ be the set of points $z_1 \neq z_0$ so that there exists a subsequence $\vtil_{n_j}$ and points $\zeta_j \in B_{R_{n_j}}(0)$ satisfying $\zeta_j \to z_1$ and $|d \vtil_{n_j}(\zeta_j)| \to \infty$. As before, if $\Gamma_1 = \emptyset$ then we get $C^\infty_{\text{loc}}$-bounds in $\C \setminus \{z_0\}$ and find a $\jtil$-holomorphic map $\vtil:\C \setminus \{z_0\} \to \R \times M$ so that, up to extraction of a subsequence, $\vtil_n \to \vtil$ in $C^\infty_{\text{loc}}(\C \setminus \{z_0\}, \R \times M)$. Then we set $\Gamma =\Gamma_0=\{z_0\}$. If $\Gamma_1 \neq \emptyset$ and $z_1 \in\Gamma_1$ we find a period $T_1>0$ and sequences $r_j \to 0^+$, $n_j\to\infty$ such that $$ \int_{B_{r_j}(z_1)} v_{n_j}^* d \lambda \to T_1 \ \mbox{ as } \ j\to \infty. $$ Considering $\vtil_{n_j}$ as the new sequence $\vtil_n$, we define $\Gamma_2 \subset \C \setminus \{z_0,z_1\}$ in the same way and proceed as before. 

Since each $z_i\in \Gamma_i\setminus\Gamma_{i-1}$ takes away a period from the $d\lambda$-energy, we get from~\eqref{ineq1} that $\Gamma_i \subset \D$, $\forall i$, and there exists $i_0$ such that $\Gamma_{i_0} \neq \emptyset$ and $\Gamma_{i_0+1} = \emptyset$. We end up with a finite set  $\Gamma = \{z_0,\ldots, z_{i_0}\}\subset \D$ and a $\jtil$-holomorphic map $\vtil: \C \setminus \Gamma \to \R \times M$ such that, up to a subsequence, $\vtil_n \to \vtil$ in $C^\infty_{\text{loc}}(\C \setminus \Gamma, \R \times M)$.

The inequality $E(\vtil) \leq C$ follows from~\eqref{ineq1} and Fatou's Lemma.
\end{proof}

\begin{definition}
A $\jtil$-holomorphic map $\vtil:\C \setminus \Gamma \to \R \times M$ as in Proposition \ref{proplimit} is called a {\bf limit of the germinating sequence $\vtil_n$}.
\end{definition}

  {It follows from the asymptotic behavior at a negative puncture} that if $\Gamma \neq \emptyset$ then $\vtil$ is non-constant. In this case all of its punctures $z\in \Gamma$ are negative and $\infty$ is a positive puncture.

\subsection{Soft rescaling near a negative puncture}\label{secsoft}

Let $\vtil=(b,v):\C \setminus \Gamma \to \R \times M$ be a non-constant limit of the germinating sequence $\vtil_n=(b_n,v_n)$ as in Proposition~\ref{proplimit}. Assume that $\Gamma \neq \emptyset$ and let $z \in \Gamma$. We already know that $z$ is necessarily a negative puncture of $\vtil$. We define its mass by 
\begin{equation}\label{massa} 
m(z)=\lim_{\epsilon \to 0^+} \int_{\partial B_\epsilon(z)} v^* \lambda=T_z>\sigma(C)>0,
\end{equation} 
where $\partial B_\epsilon(z)$ is oriented counter-clockwise and $T_z$ is the period of the asymptotic limit $P_z$ of $\vtil$ at $z$. Fix $\epsilon>0$ small enough so that 
\begin{equation}\label{sig2} 
0< \int_{\partial B_\epsilon(z)} v^* \lambda- m(z)\leq \frac{\sigma(C)}{2}.
\end{equation} 
Next we choose sequences $z_n \in B_\epsilon(z)$ and $0<\delta_n<\epsilon$, $\forall n$, so that 
\begin{eqnarray} 
\label{eqcond1} 
b_n(z_n) \leq b_n(\zeta),\ \forall \zeta \in B_\epsilon(z) \mbox{ and } \\ 
\label{eqcond2} 
\int_{B_\epsilon(z) \setminus B_{\delta_n}(z_n)} v_n^* d \lambda = \sigma(C). \end{eqnarray}
Since $z$ is a negative puncture,  \eqref{eqcond1} implies that $z_n \to z$. Hence the existence of $\delta_n$ as in~\eqref{eqcond2} follows from~\eqref{massa}. We claim that $\delta_n \to 0$. Otherwise, up to a subsequence, we may assume $\delta_n \to \delta' >0$ and choose $0<\epsilon'<\delta'$. From~\eqref{sig2}, we get the contradiction  $$\frac{\sigma(C)}{2} \geq \lim_{n \to \infty} \int_{B_\epsilon(z) \setminus B_{\epsilon'}(z)} v_n^* d\lambda \geq \lim_{n \to \infty} \int_{B_\epsilon(z) \setminus B_{\delta_n}(z_n)}
v_n^* d\lambda = \sigma(C).$$

Now take any sequence $R_n \to +\infty$ satisfying 
\begin{equation}\label{soft_sequences}
\delta_nR_n < \epsilon/2
\end{equation}
and define the sequence of $\jtil$-holomorphic maps 
\begin{equation}\label{eqwn} 
\wtil_n(\zeta) = (b_n(z_n + \delta_n \zeta)-b_n(z_n + 2 \delta_n), v_n(z_n + \delta_n \zeta)), \ \forall \zeta \in B_{R_n}(0).
\end{equation}

\begin{proposition}\label{proplimit2} 
The sequence $\wtil_n = (d_n,w_n)$ in~\eqref{eqwn} is a germinating sequence. Moreover, if $\wtil = (d,w)$ is a limit of $\wtil_n$ then $\wtil$ is non-constant, $E(\wtil) \leq C$ and the asymptotic limit $P_\infty$ of $\wtil$ at $\infty$ coincides with the asymptotic limit $P_z$ of $\vtil$ at the negative puncture $z\in \Gamma$. 
\end{proposition}

\begin{proof}
From~\eqref{eqcond2} and~\eqref{soft_sequences} we see that $$ \int_{B_{R_n}(0) \setminus \D} w_n^* d\lambda \leq \sigma(C), \forall n.$$ Moreover $E(\widetilde w_n) \leq E(\vtil_n) \leq C$ by construction. Since $w_n(2)\in \{0\} \times M$, $\tilde w_n$ is a germinating sequence.

Let $\widetilde w:\C \setminus \Gamma' \to \R \times M$ be a limit of $\widetilde w_n$ as in Proposition \ref{proplimit}, where $\Gamma' \subset \D$ is finite. If $\Gamma' \neq \emptyset$ then $\widetilde w$ is non-constant. If $\Gamma' = \emptyset$ we see from \eqref{massa} and \eqref{eqcond2} that $$\int_\D w^* d\lambda = \lim_{n \to \infty} \int_{B_\epsilon(z)} v_n^* d\lambda - \int_{B_\epsilon(z) \setminus B_{\delta_n}(z_n)} v_n^* d\lambda \geq T_z - \sigma(C) >0,$$ thus $\widetilde w$ is non-constant as well. From Fatou's Lemma we get $0<E(\wtil)\leq C$. This also implies that the periods of the asymptotic limits of $\widetilde w$ are bounded by $C$.

Identifying $S^1 = \R/\Z$, let $\W \subset C^\infty(S^1,M)$ be an open neighborhood of the set of periodic orbits $P=(x,T)$, with $T\leq C$, viewed as maps $x_T:S^1 \to M$ defined as in~\eqref{map_x_T}. We assume that $\W$ is $S^1$-invariant, meaning that $y_c( \cdot):=y( \cdot +c) \in \W \Leftrightarrow y(\cdot) \in \W$, $\forall c\in S^1$. We choose $\W$ small enough so that each of its connected components contains at most one  periodic orbit modulo $S^1$-reparametrizations. This is always possible since $\lambda$ is non-degenerate. Let $\W_\infty,\W_z\subset \W$ be the components containing $P_\infty,P_z$, respectively.

Since $\vtil_n \to \vtil$ we can choose $0<\epsilon_0<\epsilon$ small enough so that if $0<\rho\leq \epsilon_0$ is fixed then
\begin{equation}\label{eqviz}
\text{the loop} \ t\in\R/\Z \mapsto v_n(z_n+\rho e^{i2\pi t}) \ \text{belongs to} \ \W_z, \ \text{for} \ n \ \text{large.}
\end{equation}
Since $\wtil_n \to \wtil$, we can choose $R_0>1$ large enough so that if $R\geq R_0$ fixed then
\begin{equation}\label{eqR0}
\text{the loop} \ t\in\R/\Z \mapsto v_n(z_n+R\delta_n e^{i2\pi t}) \ \text{belongs to} \ \W_\infty, \ \text{for} \ n \ \text{large.}
\end{equation}
By~\eqref{sig2} and~\eqref{eqcond2} we can also choose $e>0$ small enough so that
\begin{equation}\label{eqe}
0<e < \int_{\partial B_{\delta_n R_0}(z_n)} v_n^* \lambda, \ \forall n.
\end{equation}

To finish the proof we still need the following lemma from~\cite{fols}.

\begin{lemma}\cite[Lemma 4.9]{fols}\label{lemlong} 
Given $C$, $\sigma(C)$, $e>0$ and $\W\subset C^\infty(S^1,M)$ as above,  there exists $h>0$ such that the following holds. If $\util=(a,u):[r,R]\times S^1 \to \R \times M$ is a $\jtil$-holomorphic cylinder satisfying $$ E(\util) \leq C, \hspace{0.5cm} \int_{[r,R] \times S^1} u^* d\lambda \leq \sigma(C) \hspace{0.3cm} \mbox{ and }\hspace{0.3cm}  \int_{\{r\} \times S^1} u^* \lambda \geq e, $$ then each loop $t\in S^1\mapsto u(s,t)$ is contained in $\W$ for all $s\in[r+h,R-h]$.
\end{lemma}

Consider for each $n$ the $\jtil$-holomorphic cylinder $$ (s,t)\in \left[ \frac{\ln R_0 \delta_n}{2\pi}, \frac{\ln \epsilon_0}{2\pi}  \right] \times S^1 \mapsto \widetilde C_n(s,t) = (c_n(s,t),C_n(s,t)) := \vtil_n(z_n + e^{2 \pi (s+i t)}). $$ Observe from \eqref{eqcond2} that \begin{equation}\label{hiplem}\int_{\left[ \frac{\ln R_0 \delta_n}{2\pi}, \frac{\ln \epsilon_0}{2\pi}  \right] \times S^1} C_n^* d\lambda \leq \sigma(C).\end{equation} Using \eqref{eqe} and \eqref{hiplem} we apply Lemma \ref{lemlong} to find $h>0$ so that the loop $t\in S^1\mapsto C_n(s,t)$ is contained in $\W$ for all $s\in[ (2\pi)^{-1}\ln R_0 \delta_n+h, (2\pi)^{-1}\ln \epsilon_0-h]$ and all $n$ large. Using \eqref{eqviz} and \eqref{eqR0}, we see that these loops are all contained in $\W_\infty$ and $\W_z$ for $n$ large. This implies that $\W_\infty =\W_z$ and therefore $P_\infty = P_z$. \end{proof}

\begin{remark}\label{remper}
An important feature of the limit $\wtil = (d,w):\C \setminus \Gamma' \to \R \times M$ in Proposition \ref{proplimit2} is that either 
\begin{itemize} 
\item[(i)] $\int_{\C \setminus \Gamma'} w^*d\lambda >0$ or 
\item[(ii)] $\int_{\C \setminus \Gamma'} w^*d\lambda =0$ and $\# \Gamma' \geq 2$. 
\end{itemize} 
  {In fact, $\Gamma'\neq\emptyset \Rightarrow 0\in\Gamma'$; this follows from~\eqref{eqcond1}-\eqref{eqwn}. Arguing by contradiction, assume that $\int_{\C \setminus \Gamma'} w^*d\lambda =0$ and $\# \Gamma' = 1$. Thus $\Gamma'=\{0\}$ and we can estimate
\begin{equation}
\begin{aligned}
m(z) &= T_z = \int_{\partial\D}w^*\lambda = \lim_{n\to\infty} \int_{\partial\D} w_n^*\lambda \\
&= \lim_{n\to\infty} \int_{\partial B_{\delta_n}(z_n)} v_n^*\lambda = \lim_{n\to\infty} \int_{\partial B_\epsilon(z)} v_n^*\lambda - \sigma(C) \\
&= \int_{\partial B_\epsilon(z)} v^*\lambda - \sigma(C) \leq m(z)-\frac{\sigma(C)}{2}
\end{aligned}
\end{equation} 
which is a contradiction; here we used~\eqref{sig2},~\eqref{eqcond2} and Proposition~\ref{proplimit2}.} It follows that in both cases (i) and (ii) the period $T'$ of the asymptotic limit at a negative puncture $z'\in \Gamma'$ is strictly less than the period $T_\infty$ of the asymptotic limit $P_\infty$ at $\infty$. In fact, from \eqref{eqsigC} one has the following lower bound on the difference of periods \begin{equation}\label{eqdifT} T' < T_\infty - \sigma(C). \end{equation} This inequality will be useful in the next section.
\end{remark}

\subsection{The bubbling-off tree}\label{secbubtree}

Now we describe the compactness properties of germinating sequences in the sense of Symplectic Field Theory (SFT)~\cite{sftcomp}.

\begin{definition}  
A   {{\it bubbling-off tree}} associated to a germinating sequence $\vtil_n=(b_n,v_n)$ is a finite set of finite energy spheres $\{\util_q: q \mbox{ is a vertex of } \T\}$, modelled on a finite tree $\T=(V,r,E)$, where $V$ is the set of vertices, $r$ is a special vertex called the   {{\it root}}, and $E$ is the set of edges, oriented from the root, satisfying the following requirements:
\begin{itemize}
\item[(i)] Each $\util_q$, $q\in V$, is a non-constant finite energy sphere with one positive puncture and a finite number of negative punctures. Moreover, if $q\neq r$ then $\util_q$ is not a trivial cylinder over a periodic orbit.
\item[(ii)] The edges issuing from a vertex $q\in V$ are in bijection with the negative punctures of $\util_q$. If $e$ is such an edge from $q$ to $q'$ and $z$ is the negative puncture corresponding to $e$ then the asymptotic limit $P_z$ of $\util_q$ at $z$ coincides with the asymptotic limit $P_\infty$ of $\util_{q'}$ at its positive puncture.
\item[(iii)] There exists a subsequence of $\vtil_n$, still denoted by $\vtil_n$, such that $\vtil_n \to \util_r$ in $C^\infty_{\text{loc}}$ away from the negative punctures of $\util_r$, and for each $r\neq q \in V$, we find sequences $\{z_n\} \subset \C$, $r_n \to 0^+$ and $\{c_n\} \subset \R$, such that $$ \widetilde U_n(z):=(b_n(z_n +r_nz)+c_n,v_n(z_n+r_n z)) \to \util_q \mbox{ in } C^\infty_{\text{loc}}. $$
\end{itemize}
\end{definition}

It follows from this definition that all asymptotic limits of all curves $\util_q$, $q\in V$, are contractible closed Reeb orbits.   {The following proposition is a particular and simpler case of the SFT Compactness Theorem~\cite{sftcomp}, we include a proof here for the sake of completeness.}

\begin{proposition}\label{propbubtree}
Let $\vtil_n$ be a germinating sequence with a non-constant limit $\vtil$. Then associated to $\vtil_n$ there exists a bubbling-off tree $\{\util_q: q \mbox{ is a vertex of }\T\},$ modelled on the tree $\T=(V,r,E)$, with $\util_r = \vtil$. Moreover, $E(\util_q)\leq \sup_n E(\vtil_n)$ for every vertex $q$.
\end{proposition}

\begin{proof} 
Let $\vtil:\C \setminus \Gamma \to \R \times M$ be a non-constant limit of $\vtil_n$ as in Proposition~\ref{proplimit}. We can assume that $\vtil_n \to \vtil$ in $C^{\infty}_{\text{loc}}$ after selection of a subsequence. We start with a tree containing just the root $r$ and let $\util_r = \vtil$. If $\Gamma= \emptyset$ then we are done and $\T=(\{r\},r,\emptyset)$. Otherwise let $z\in \Gamma$ be a negative puncture of $\vtil$. We find a germinating sequence $\widetilde w_{z,n}$ near $z$, defined as in \eqref{eqwn}, such that, up to extraction of a subsequence, $\widetilde w_{z,n} \to \widetilde w_z$ in $C^\infty_{\text{loc}}$. See Proposition \ref{proplimit2}. Moreover, the asymptotic limit of $\vtil$ at $z$ coincides with the asymptotic limit of $\widetilde w_z$ at its positive puncture. Also, $\widetilde w_z$ is not a cylinder over a periodic orbit, see Remark \ref{remper}. Then we add a vertex $q_z$ to the tree and let $\util_{q_z} = \widetilde w_z$. An edge from $r$ to $z$ is also added to the tree. We do the same for all $z\in \Gamma$.

Now given $z\in \Gamma$ we take a negative puncture $z'\in \D$ of $\widetilde w_z$ and use the sequence $\widetilde w_{z,n}$ as above in order to define near $z'$ a new germinating sequence $\widetilde w_{z'z,n}$ via soft-rescalling as in the proof of Proposition~\ref{proplimit2}. Then after taking a subsequence we assume that $\widetilde w_{z'z,n} \to \widetilde w_{z'z}$ in $C^{\infty}_{\text{loc}}$. We add a new vertex $q_{z'}$ below $q_z$ corresponding to $z'$ and let $\util_{q_{z'}}=\widetilde w_{z'z}$. We also add an edge from $q_z$ to $q_{z'}$.  We do the same for all negative punctures of $\widetilde w_z$.

We keep doing this step by step for all negative punctures. As explained in Remark \ref{remper} this process has to terminate after a finite number of steps since the periods of the corresponding asymptotic limits strictly decrease when going down the tree. See \eqref{eqdifT}. We conclude that after taking a subsequence of $\vtil_n$ we end up with a finite set of finite energy spheres modelled on a tree $\T=(V,r,E)$, so that all the requirements (i), (ii) and (iii) above are satisfied, thus obtaining the desired bubbling-off tree associated to the germinating sequence $\vtil_n$.     \end{proof}

\begin{figure}
\includegraphics[width=250\unitlength]{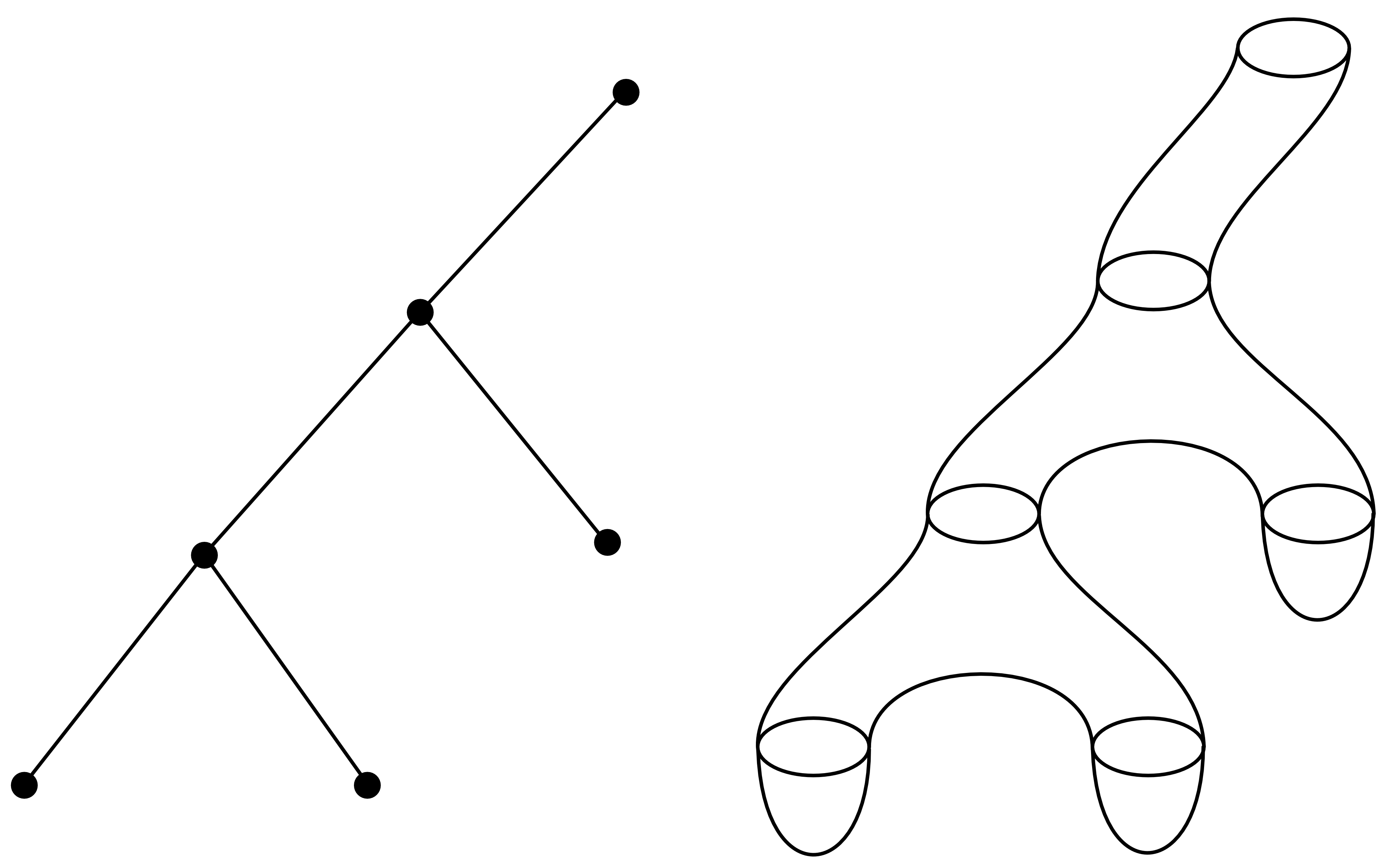}
\caption{A bubbling-off tree.}
\end{figure}

\subsection{Some index estimates}\label{sec_index_estimates}

Let $\vtil=(a,v):\C \setminus \Gamma \to \R \times M$ be a non-constant finite energy sphere, where $\Gamma$ is finite and  consists of negative punctures. Assume that for each $z\in \Gamma$ the asymptotic limit $P_z = (x_z,T_z)$  at $z$ is contractible. In particular, $P_\infty = (x_\infty, T_\infty)$, the asymptotic limit at the positive puncture $\infty$, is also contractible.

Let $u_z:\D \to M$ be a spanning disk for $P_z$ for $z\in\Gamma$, i.e., $u_z$ is continuous and $u_z(e^{2 \pi i t}) = x_z(tT_z)$, $\forall t\in S^1$.  A trivialization $\Phi_z: u_z ^* \xi \to \D \times \C$  induces a homotopy class $\beta_z$ of trivializations of $x_z(\cdot T_z)^*\xi \to S^1$. Since $c_1(\xi)$ vanishes on $\pi_2(M)$, $\beta_z$ does not depend on the choice of $u_z$.

Let $\bar v$ be a continuous extension of $v$ to the compactification of $\C \setminus \Gamma$ by adding copies of $S^1$ at its punctures in $\Gamma \cup \{\infty\}$. From the asymptotic behavior of $\vtil$ at the punctures, we can assume that $u_\infty := \bar v \cup_{z \in \Gamma} u_z$ is a spanning disk for $P_\infty$ and observe that a trivialization $\Phi_\infty: u_\infty^* \xi \to \D \times \C$ restricts to a trivialization of $x_z(\cdot T_z)^*\xi \to S^1$ in class $\beta_z$. Denote by $\beta_\infty$ to homotopy class of trivializations of $x_\infty(\cdot T_\infty)^* \xi \to S^1$ induced by $u_\infty$.

Let $A_{P_\infty},A_{P_z}$, $z\in \Gamma$, be the asymptotic operators associated to $P_\infty ,P_z,z\in \Gamma,$ defined in \S~\ref{sssec_asymp_ops}. Denote by
\begin{equation}\label{notation_wind_infty}
\wind_\infty(z):=\wind_\infty(\vtil,z,\sigma) \ \ \ \text{and} \ \ \ \wind_\infty(\infty):=\wind_\infty(\vtil,\infty,\sigma)
\end{equation}
the asymptotic winding numbers defined in \S~\ref{sec_alg_invs} where $\sigma$ is a non-trivial section of $u_\infty^*\xi$.   {These} numbers are well-defined if $\pi \circ dv$ does not vanish identically. We denote by $\mu_{CZ}(P_z)$, $z\in \Gamma$, $\mu_{CZ}(P_\infty)$ the Conley-Zehnder indices of $P_z$, $z\in \Gamma$, $P_\infty$ computed with respect to $\beta_z$, $z\in\Gamma$, $\beta_\infty$ respectively.

\begin{lemma}\label{lemindex}
Let $\widetilde v=(a,v):\C \setminus \Gamma \to \R \times M$ be a non-constant finite energy $\jtil$-holomorphic map as above.
\begin{itemize} 
\item[(i)] Assume that $\mu_{CZ}(P_\infty) \leq 1$. Then $\Gamma \neq \emptyset$ and there exists $z\in \Gamma$ so that $\mu_{CZ}(P_z) \leq 1$. 
\item[(ii)] Assume that $\mu_{CZ}(P_\infty) = 2$ and $\mu_{CZ}(P_z) \geq 2$, $\forall z\in \Gamma$. Then $\mu_{CZ}(P_z)=2$, $\forall z\in\Gamma$.
\item[(iii)] Assume that $\int_{\C \setminus \Gamma} v^*d\lambda >0$, $\wind_{\infty}(\infty) \leq 1$ and  $\mu_{CZ}(P_z)\geq 2$, $\forall z\in \Gamma$. Then $\mu_{CZ}(P_z)=2$, $\forall z\in \Gamma$. 
\end{itemize} 
\end{lemma}

\begin{proof}
Assume that $\mu_{CZ}(P_\infty)\leq 1$. If $\Gamma= \emptyset$ then it is possible to show in a standard way that $\int_\C v^* d\lambda>0$, since $\vtil$ is non-constant and has finite energy. Thus we have well defined winding numbers $\wind_\pi(\vtil)$ and $\wind_\infty(\infty)$. From $\mu_{CZ}(P_\infty) \leq 1$ we have $\wind_\infty(\infty) \leq \wind^{<0}(A_{P_\infty}) \leq 0$  {, see Remark~\ref{rmk_wind_infty}}. Hence $$ 0 \leq \wind_\pi(\vtil) = \wind_\infty(\infty) -1 \leq-1, $$ a contradiction which proves that $\Gamma \neq \emptyset$.

If $\Gamma \neq \emptyset$ and $\int_{\C \setminus \Gamma} v^*d \lambda >0$, then we also have well defined winding numbers $\wind_\infty(z)$, $z\in \Gamma\cup\{\infty\}$, and 
\begin{equation}\label{windinf1}
\begin{aligned} 
0\leq \wind_\pi(\vtil) & =  \wind_\infty(\vtil) + \#\Gamma -1 \\ 
& =  \wind_\infty(\infty) - \sum_{z\in \Gamma} \wind_\infty(z) + \#\Gamma-1.
\end{aligned} 
\end{equation} 
As before we have $\wind_\infty(\infty) \leq 0.$ Assuming that $\mu_{CZ}(P_z) \geq 2$, $\forall z \in \Gamma,$ we get $\wind_\infty(z) \geq \wind^{\geq 0}(A_{P_z}) \geq 1,\forall z\in \Gamma$  {, see again Remark~\ref{rmk_wind_infty}}. Using \eqref{windinf1} we find $$0\leq 0 - \sum_{z\in \Gamma} 1 +\#\Gamma -1 = -\#\Gamma + \#\Gamma -1=-1,$$ a contradiction. Thus there exists at least one $z\in \Gamma$ such that $\mu_{CZ}(P_z) \leq 1$.

Now if $\Gamma \neq \emptyset$ and $\int_{\C \setminus \Gamma} v^* d\lambda =0$, then we find a simple periodic orbit $Q$ and positive integers $k_z$, $z\in \Gamma$, $k_\infty$ such that $P_z=Q^{k_z}$, $P_\infty=Q^{k_\infty}$, where $k_\infty = \sum_{z\in \Gamma} k_z$. For any $z\in \Gamma$, let $l_z\in \N$ be the least common multiple of $k_z$ and $k_\infty$. Then $Q^{l_z}=P_z^{\frac{l_z}{k_z}}=P_\infty^{\frac{l_z}{k_\infty}}$ is contractible and we may compute $\mu_{CZ}(Q^{l_z})$ using trivializations induced by a $\frac{l_z}{k_z}$-cover of $u_z$ or by a $\frac{l_z}{k_\infty}$-cover of $u_\infty$. It is independent of this choice since $c_1(\xi)$ vanishes on $\pi_2(M)$. From the hypothesis $\mu_{CZ}(P_\infty) \leq1$ and the definition of the Conley-Zehnder index we get $$\mu_{CZ}(Q^{l_z})=\mu_{CZ}(P_\infty^{\frac{l_z}{k_\infty}}) \leq 2\frac{l_z}{k_\infty} -1. $$ Arguing indirectly, assume that  $\mu_{CZ}(P_z) \geq 2$. Then $$\mu_{CZ}(Q^{l_z}) =\mu_{CZ}(P_z^{\frac{l_z}{k_z}}) \geq 2\frac{l_z}{k_z}\geq 2\frac{l_z}{k_\infty},$$ since $k_z \leq k_\infty$. This contradiction proves that $\mu_{CZ}(P_z) \leq 1,\forall z\in \Gamma$, proving~(i).

The proof of (ii) is similar. If $\int_{\C\setminus \Gamma} v^* d\lambda =0$ then as before $\Gamma \neq \emptyset$ and $P_\infty = Q^{k_\infty}$, $P_z=Q^{k_z}$ for a simple periodic orbit $Q$ and positive integers $k_\infty,k_z,z\in \Gamma,$ satisfying $k_\infty = \sum_{z\in \Gamma}k_z$. If $\mu_{CZ}(P_z)>2$ for some $z\in \Gamma$, then we consider  $Q^{l_z}=P_z^{\frac{l_z}{k_z}}=P_\infty^{\frac{l_z}{k_\infty}}$ where $l_z=\text{lcm}(k_\infty,k_z)$. From $\mu_{CZ}(P_\infty)=2,$ we get $\mu_{CZ}(Q^{l_z})=2\frac{l_z}{k_\infty}.$ From $\mu_{CZ}(P_z)>2$ we get $\mu_{CZ}(Q^{l_z}) \geq 2\frac{l_z}{k_z}+1 \geq 2\frac{l_z}{k_\infty}+1,$ a contradiction.

If $\int_{\C\setminus \Gamma} v^* d\lambda >0$ then $\mu_{CZ}(P_\infty)=2 \Rightarrow \wind_\infty(\infty) \leq 1$. If there exists $z\in \Gamma$ such that $\mu_{CZ}(P_z)>2$ then $\wind_\infty(z) \geq  2.$ For all the other punctures we have $\wind_\infty(z) \geq 1$ since $\mu_{CZ}(P_z) \geq 2$. It follows that $$0\leq \wind_\pi(\vtil) = \wind_\infty(\vtil) +\#\Gamma  -1 \leq 1 -2-(\#\Gamma -1) + \#\Gamma -1 =-1,$$ a contradiction. This finishes the proof of (ii) and also proves (iii).   \end{proof}

\subsection{Existence of a periodic orbit with Conley-Zehnder index $2$}

We start proving the following lemma. Here we use the notation defined in \S~\ref{sec_index_estimates}.

\begin{lemma}\label{lemprinc}
Let $\vtil=(b,v)$ be a non-constant limit of a germinating sequence $\vtil_n=(b_n,v_n)$ on $\R\times M$. Then the asymptotic limit $P_\infty$ at the positive puncture of $\vtil$ satisfies $\mu_{CZ}(P_\infty) \geq 2$. 
\end{lemma}

\begin{proof}
Let $\{\util_q:q \mbox{ is a vertex of } \T\}$, $\T=(V,r,E)$, be the bubbling-off tree associated to the germinating sequence $\vtil_n$ so that $\util_r = \vtil$. See Proposition \ref{propbubtree}. We claim  that the asymptotic limits of $\util_q$ satisfy $\mu_{CZ} \geq 2$, $\forall q$. Otherwise we find $q\in V$ such that the asymptotic limit of $\util_q$ at its positive puncture has index $\mu_{CZ} \leq 1$. By Lemma \ref{lemindex}-(i) the asymptotic limit of a negative puncture $z$ of $\util_q$ also has $\mu_{CZ} \leq 1$. Thus if $q'$ is the vertex immediately below $q$, corresponding to $z$, then the asymptotic limit of the positive puncture of $\util_{q'}$ has $\mu_{CZ}\leq 1$. Now we proceed in the same way starting from the vertex $q'$ and find a vertex $q''$ immediately below $q'$ such that the asymptotic limit of $\util_{q''}$ at its positive puncture has $\mu_{CZ} \leq 1$. Keeping track of such punctures with asymptotic limits having $\mu_{CZ} \leq 1$ we end up finding a vertex $q_b$ at the bottom of the tree so that $\util_{q_b}$ is a finite energy plane with positive asymptotic limit $\tilde P_\infty$ satisfying $\mu_{CZ}(\tilde P_\infty) \leq 1$. This is impossible and this contradiction shows that all the asymptotic limits of all $\util_q$, $q\in V$, have $\mu_{CZ} \geq 2$. In particular, we have $\mu_{CZ}(P_\infty)\geq 2,$ where $P_\infty$ is the asymptotic limit of $\vtil$ at its positive puncture.
\end{proof}


\begin{proposition}\label{main_prop_compactness}
Let $\vtil=(b,v)$ be a limit of a germinating sequence $\vtil_n=(b_n,v_n)$ on $\R\times M$ and assume that $\vtil$ has a negative puncture. Assume further that one of the following conditions holds:
\begin{itemize}
\item[(a)] $\vtil$ is asymptotic at its positive puncture to $P_\infty$ satisfying $\mu_{CZ}(P_\infty)\leq2$.
\item[(b)] $\int v^*d\lambda>0$ and $\wind_\infty(\infty)\leq 1$, see~\eqref{notation_wind_infty}.
\end{itemize}
Then there is a finite-energy $\jtil$-holomorphic plane $\util_0=(a_0,u_0):\C\to\R\times M$ asymptotic to a closed Reeb orbit $P_0$ satisfying $\mu_{CZ}(P_0)=2$ such that $E(\util_0)\leq \sup_n E(\vtil_n)$. Moreover, if the images of the maps $v_n$ do not intersect a given periodic orbit $P$ then $u_0(\C) \cap P = \emptyset$.
\end{proposition}

\begin{proof}
Let $\{\util_q:q \mbox{ is a vertex of } \T\}$, $\T=(V,r,E)$, be the bubbling-off tree associated to the germinating sequence $\vtil_n$ so that $\util_r = \vtil$. See Proposition \ref{propbubtree}. We have $\mu_{CZ}(P)\geq 2$, where $P$ is any asymptotic limit of any $\tilde u_q$, $q \in V$. This follows from Lemma~\ref{lemprinc} because every vertex $q\in V$ is the root of the subtree below $q$, see the proof of Lemma~\ref{lemprinc}. Under hypothesis (a) or (b) we start from the root $r$ applying Lemma \ref{lemindex}-(ii) or (iii), respectively, to conclude that all asymptotic limits of $\util_r$ at its negative punctures satisfy $\mu_{CZ}=2$. Then we apply Lemma \ref{lemindex}-(ii) successively from the root to the bottom to conclude that all asymptotic limits of $\util_q$, $r\neq q\in V$, have $\mu_{CZ}=2$. In particular, any vertex $q_0$ in the bottom of the tree corresponds to a finite energy plane $\util_0 = (a_0,u_0)$ with asymptotic limit $P_0$ at $\infty$ with $\mu_{CZ}(P_0)=2$. We necessarily have $\int_\C u_0^*d\lambda>0$ and $\wind_\infty(\util_0)=1$. Thus $0\leq \wind_\pi(\util_0) = \wind_\infty(\util_0)-1 = 0,$ which implies that $\wind_\pi(\util_0)=0$. It follows that $u_0$ is an immersion transverse to the Reeb vector field and if $u_0$ intersects a given periodic orbit $P$ then $v_n$ also intersects $P$ for all large $n$, concluding the proof of the proposition.
\end{proof}

\section{Compactness of fast planes}

As in the previous section, we fix a closed co-oriented   {tight} contact $3$-manifold $(M,\xi)$ such that $c_1(\xi)$ vanishes on $\pi_2(M)$. We also fix a nondegenerate contact form $\lambda$ which defines $\xi$ and induces the given co-orientation.

Let $P = (x,T) \in \P(\lambda)$. Consider $p$ defined by $T=pT_{\rm min}$, where $T_{\rm min}>0$ is the minimal period of $x$.

\begin{definition}[Hofer, Wysocki and Zehnder~\cite{props2}]\label{def_p_unknotted_disk}
Let $h:\D\to M$ be a continuous capping disk for $P$ in the sense that $h(e^{i2\pi t}) = x(Tt+c)$ for some $c\in\R$. We will say that $P$ is $p$-unknotted in the homotopy class of $h$ if there exists a $p$-disk for $x(\R)$ representing the same class in $\pi_2(M,x(\R))$ as $h$.
\end{definition}

\begin{remark}
Let $\util=(a,u):\C\to \R\times M$ be a finite-energy plane asymptotic to $P$. The open domain $\C$ can be compactified to a closed disk $\C \sqcup S^1$ by adding a circle at $\infty$, and $u$ can be extended to a smooth capping disk $\bar u: \C \sqcup S^1 \to M$ for $P$. We say that $P$ is   {{\it $p$-unknotted in the homotopy class of $u$}} when it is $p$-unknotted in the homotopy class of $\bar u$ in the sense of Definition~\ref{def_p_unknotted_disk}.
\end{remark}

Choose $J\in \J_+(\xi)$ and a compact set $H\subset \R\times (M\setminus x(\R))$. We denote by
\begin{equation}\label{set_of_fast_planes}
\Lambda(H,P,\lambda,J)
\end{equation}
the set of embedded finite-energy $\jtil$-holomorphic planes $\util = (a,u) : \C\to \R\times M$ asymptotic to $P$ at its positive puncture $\infty$, such that $x(\R)$ is $p$-unknotted in the homotopy class of $u$, the plane $\util$ is fast  i.e., $\wind_\pi(\util)=0$,
\begin{equation}\label{normalization_fast_planes}
\begin{array}{ccc}
\util(0) \in H & \text{and} & \int_{\C\setminus\D} u^*d\lambda = \sigma(T). \end{array}
\end{equation}
The positive constant $\sigma(T)$ is defined as in~\eqref{eqsigC}.

\begin{theorem}\label{thm_comp_fast}
Let $x : \R\to M$ be a periodic trajectory for the Reeb flow of $\lambda$, and denote by $T_{\rm min}$ its minimal period. If $J \in \J_+(\xi)$, $H\subset \R\times(M\setminus x(\R))$ is a compact set, $x(\R)$ is $p$-unknotted, and every contractible closed Reeb orbit $P_* \subset M\setminus x(\R)$ satisfying $\mu_{CZ}(P_*,\beta_{\rm disk})=2$ and $\int_{P_*}\lambda \leq pT_{\rm min}$ is not contractible in $M\setminus x(\R)$, then, defining $P=(x,T:=pT_{\rm min})$, the set $\Lambda(H,P,\lambda,J)$ is $C^\infty_{\rm loc}$-compact.
\end{theorem}

Until the end of this section we fix $x$, $J$, $H$, $p$, $T_{\rm min}$ and $T=pT_{\rm min}$ as in Theorem~\ref{thm_comp_fast}. An important ingredient in the proof of Theorem~\ref{thm_comp_fast} is the following immediate consequence of the proof of Theorem~4.10 from~\cite{props2}.

\begin{theorem}\label{thm_fast_embedded}
If $\util =(a,u) \in \Lambda(H,P,\lambda,J)$ then $u(\C) \cap x(\R) = \emptyset$ and $u:\C \to M\setminus x(\R)$ is an embedding. In particular, $\bar u$ is a $p$-disk for $x(\R)$.
\end{theorem}

\begin{proof}[Proof of Theorem~\ref{thm_fast_embedded}]
In~\cite[Theorem~4.10]{props2} it is assumed that $\mu_{CZ}(P,\beta_{\rm disk})\leq 3$. But inspecting its proof one sees that the only point where this inequality is used is to guarantee that $\wind_\infty=1$, which is automatically true under our assumption that $\util$ is fast.   {The assumption that $\xi$ is tight is crucial.}
\end{proof}

Let us prove Theorem~\ref{thm_comp_fast}. In the following we fix, as in section~\ref{section_bubb_off_analysis}, an $\R$-invariant Riemannian metric on $\R\times M$. Domains in $\C$ and $\R\times S^1$ will be equipped with its standard complex structure and euclidean metric. Norms of vectors and linear maps are taken with respect to these choices.

Consider an arbitrary sequence
\begin{equation*}
\util_k = (a_k,u_k) \in \Lambda(H,P,\lambda,J)
\end{equation*}
and define $$ \Gamma = \{ z\in\C \mid \exists k_j \to \infty \ \text{and} \ z_j\to z \ \text{such that} \ |d\util_{k_j}(z_j)|\to\infty \}. $$ There is no loss of generality to assume that $\Gamma$ is finite. This can be achieved by passing to a subsequence of $\util_k$, still denoted by $\util_k$, and is proved using Hofer's lemma   {\cite[Lemma~4.6.4]{mcdsal}, which allows us to rescale the curves near a bubbling-off point to produce a finite-energy plane,} and the fact that a finite-energy plane that bubbles-off from points of $\Gamma$ has a minimum positive quantum of $d\lambda$-area. By similar reasons it follows from~\eqref{normalization_fast_planes} and from the definition of the number $\sigma(T)$ that $\Gamma \subset\D$. Thus the sequence $|d\util_k|$ is $C^0_{\rm loc}$-bounded on $\C\setminus \Gamma$. Defining $\vtil_k$ by
\begin{equation*}
\vtil_k(z) = (b_k(z),v_k(z)) := (a_k(z)-a_k(2),u_k(z))
\end{equation*}
we can use   {a} standard elliptic boot-strapping argument  {, see~\cite[chapter~4]{mcdsal},} to obtain a subsequence of $\vtil_k$, still denoted by $\vtil_k$, and a finite-energy $\jtil$-holomorphic map
\begin{equation*}
\vtil =(b,v) : \C\setminus \Gamma \to \R\times M
\end{equation*}
such that
\begin{equation}\label{conv_subseq_fast}
\vtil_k \to \vtil \ \text{in} \ C^\infty_{\rm loc}(\C\setminus \Gamma) \ \text{as} \ k\to\infty.
\end{equation}
  {Up to taking a further subsequence we may also assume, without loss of generality, that $\Gamma$ consists of non-removable punctures of $\vtil$.} The map $\vtil$ is not constant. This is obvious if $\Gamma\neq \emptyset$ and
\begin{equation}\label{area_no_neg_punctures}
\Gamma = \emptyset \Rightarrow \int_\D v^*d\lambda = \lim_{k\to\infty} \int_\D v_k^*d\lambda = T - \sigma(T) > 0.
\end{equation}
It is easy to see that points in $\Gamma$ are negative punctures of $\vtil$, from where it follows that $\infty$ is a positive puncture. The nontrivial Lemma~\ref{lemlong}, the properties of the number $\sigma(T)$ and the normalization conditions~\eqref{normalization_fast_planes} together tell us that for every $\R/\Z$-invariant neighborhood $\W$ of $P$ in the loop space $C^\infty(\R/\Z,M)$, there exists $k_0$ and $R_0> 1$ such that the loops $t\mapsto v_k(R_1 e^{i2\pi t})$, $t\mapsto v(R_1 e^{i2\pi t})$ belong to $\W$ whenever $k\geq k_0$ and $R_1\geq R_0$. Since $\lambda$ is nondegenerate we can take $\W$ with the property that if $\gamma \in \W$ is a periodic Reeb trajectory then $\exists c'\in\R$ such that $\gamma(t) = x(Tt+c')$. In particular, $P$ is the asymptotic limit of $\vtil$ at its (unique) positive puncture $\infty$.

We claim that $\pi \circ dv$ does not vanish identically. By~\eqref{area_no_neg_punctures} this is true when $\Gamma=\emptyset$. Assume by contradiction that $\Gamma\neq \emptyset$ and $\pi \circ dv \equiv 0$. Then we find a non-constant complex polynomial $Q$ of degree $p$ such that $\Gamma = Q^{-1}(0)$ and
\begin{equation*}
\vtil = Z \circ Q
\end{equation*}
where $Z:\C\setminus\{0\} \to \R\times M$ denotes the trivial cylinder given by the formula $$ Z(e^{2\pi(s+it)}) = (T_{\rm min}s,x(T_{\rm min}t)). $$ Consequently $0 \in\Gamma$ since, otherwise, $\lim_k \util_k(0)=(b(0)+c,v(0)) \in H \cap (\R\times x(\R))$, contradicting the properties of the set $H$, where $c=\lim_k a_k(2)$ (which only exists up to a subsequence and when $0\not\in\Gamma$). We claim that $\#\Gamma\geq 2$. If not then $\Gamma = \{0\}$ and 
\begin{equation*}
T = \int_{\partial\D} v^*\lambda = \lim_{k\to\infty} \int_\D v_k^*d\lambda = T-\sigma(T)
\end{equation*}
which is absurd. Choosing $z_*\in \Gamma$, $\exists 1\leq p'<p$ such that $c(t) = v(z_* + \epsilon e^{i2\pi t})$ is a reparametrization of $t\mapsto x(p'T_{\rm min}t)$, where $0<\epsilon\ll{\rm dist}(z_*,\Gamma\setminus\{z_*\})$. The sequence of loops $c_k(t) := v_k(z_* + \epsilon e^{i2\pi t})$ $C^\infty$-converges to $c$ and clearly each $c_k$ is contractible. Thus so is $c$, contradicting Lemma~\ref{lemma_no_p'_disk}. We showed that $\vtil$ has positive $d\lambda$-area.

%
%

Let $R_1>R_0$ satisfy the following property:
\begin{equation}\label{special_radius}
\pi \circ dv \ \text{does not vanish on} \ \{z\in\C : |z|\geq R_1\}.
\end{equation}
The existence of $R_1$ with this property follows from Theorem~\ref{thm_precise_asymptotics} and from the already established fact that $\pi \circ dv$ does not vanish identically. Now consider the loop $\gamma_\infty :\R/\Z\to M$, $\gamma_\infty(t) = v(R_1e^{i2\pi t})$. The sequence of loops $\gamma_k(t) = v_k(R_1e^{i2\pi t})$ converges in $C^\infty$ to $\gamma_\infty$ because $\Gamma \subset \D$, $R_1>1$ and by~\eqref{conv_subseq_fast}. By the same reason the vector field $t\mapsto \pi \cdot \partial_r v(R_1e^{i2\pi t})$ along $\gamma_\infty$ and the vector field $t\mapsto \pi\cdot \partial_r v_k(R_1e^{i2\pi t})$ along $\gamma_k$ are arbitrarily $C^\infty$-close to each other when $k$ is large enough. These vector fields define non-vanishing sections of $\gamma_\infty^*\xi$ and of $\gamma_k^*\xi$, respectively. Consider $C^\infty$-small smooth homotopies $g_k : [0,1]\times \R/\Z \to M$ satisfying $g_k(0,t) = \gamma_k(t)$ and $g_k(1,t) = \gamma_\infty(t)$. Then, when $k$ is large enough, the vector fields $\pi \cdot \partial_r v(R_1 e^{i2\pi t})$ and $\pi\cdot \partial_rv_k(R_1 e^{i2\pi t})$ can be   {extended} smoothly to a non-vanishing section of $g_k^*\xi$ since they are $C^\infty$-close.

For every $k$ large enough, consider a non-vanishing section $Z_k$ of $v_k^*\xi$. The sections $Z_k|_{\overline B}$ of $(v_k|_{\overline B})^*\xi$, with $B = B_{R_1}(0)$, extend to a non-vanishing section $\mathcal Z_k$ of $\xi$ over a piecewise smooth capping disk $\mathcal D_k$ for $\gamma_\infty$ defined by attaching $g_k$ to $v_k|_{\overline B}$. We can arrange $\mathcal Z_k$ to be smooth over $\gamma_\infty = \partial\mathcal D_k$. It is important now to finally note that, since $c_1(\xi)$ vanishes on $\pi_2(M)$, we can extend $\mathcal Z_k|_{\gamma_\infty}$ to a smooth non-vanishing section of $v^*\xi$, still denoted by $\mathcal Z_k$, with the following property:
\begin{itemize}
\item[$(*)$] Let $P_z$ be the asymptotic limit of $\vtil$ at each negative puncture $z\in\Gamma$ and denote $P_\infty=P$. Choose capping disks $D_z$ for each $P_z$, $\forall z\in \Gamma\cup\{\infty\}$. Then the section $\mathcal Z_k$ extends along each $z\in \Gamma\cup\{\infty\}$ to non-vanishing sections of $\xi$ along $P_z$ that can be further extended along $D_z$ without vanishing.
\end{itemize}
The capping disks $D_z$ do exist since all orbits $P_z$ are necessarily contractible. Property $(*)$ follows easily from the fact that $c_1(\xi)$ vanishes on $\pi_2(M)$. In particular, all $\mathcal Z_k$ are homotopic to each other through non-vanishing sections of $v^*\xi$. Now we can compute
\begin{equation}\label{estimate_wind_infty}
\begin{aligned}
\wind_\infty(\vtil,\infty,\mathcal Z_k) &= \wind(t\mapsto \pi\cdot \partial_rv (R_1e^{i2\pi t}), t\mapsto \mathcal Z_k(R_1e^{i2\pi t})) \\
&= \wind(t\mapsto \pi\cdot \partial_rv_k (R_1e^{i2\pi t}), t\mapsto Z_k(R_1e^{i2\pi t})) \\
&\leq \wind_\infty(\vtil_k) = 1.
\end{aligned}
\end{equation}
The first equality follows from~\eqref{special_radius}. The second equality follows from invariance of winding numbers under deformation through non-vanishing sections. The inequality in the third line follows since all zeros of $\pi \circ dv_k$ are isolated and count positively to the algebraic count of zeros of $\pi \circ dv_k$.

Inequality~\eqref{estimate_wind_infty} is the key step in proving that $\Gamma=\emptyset$. In fact, suppose by contradiction that $\Gamma\neq\emptyset$. The curve $\vtil$ is a limit of the germinating sequence $\vtil_k$ satisfying (b) in Proposition~\ref{main_prop_compactness}, with $E(\vtil_k) = pT_{\rm min}$, $\forall k$. As a consequence we find a finite-energy $\jtil$-holomorphic plane $\util_*=(a_*,u_*)$ asymptotic to a closed Reeb orbit $P_*$ satisfying $\mu_{CZ}(P_*,\beta_{\rm disk})=2$, $\int_{P_*}\lambda \leq pT_{\rm min}$ and $u_*(\C) \cap x(\R) = \emptyset$. We claim that $P_*$ and $P$ are geometrically distinct. If not then $P_* = (x,p'T_{\rm min})$. Moreover, $p'<p$ since $\mu_{CZ}(P,\beta_{\rm disk})\geq 3$ and $\mu_{CZ}(P_*,\beta_{\rm disk})=2$, see the proof of Lemma~\ref{lemindex}. The closure of the open disk $u_*(\C)$ is a disk for the $p'$-th iterate of $x(\R)$, a contradiction to Lemma~\ref{lemma_no_p'_disk}. This shows that $P_* \subset M\setminus x(\R)$ and, consequently, that $P_*$ is contractible in $M\setminus x(\R)$, contradicting the hypotheses of Theorem~\ref{thm_comp_fast}. Thus $\Gamma=\emptyset$.

We claim that $\vtil:\C\to \R\times M$ is an embedding. The identity $\wind_\infty(\vtil)\leq 1$ implies that $\wind_\infty(\vtil)=1$ and $\wind_\pi(\vtil)=0$, from where it follows that $\vtil$ is an immersion. If $\vtil$ is not somewhere injective then, by results of~\cite{props2}, one finds a complex polynomial $Q$ of degree at least $2$ and a non-constant $\jtil$-holomorphic map $F:\C\to \R\times M$ such that $\vtil = F\circ Q$. The zeros of $Q'$ would force zeros of $d\vtil$ which is impossible because $\vtil$ is an immersion. We conclude that $\vtil$ is a somewhere injective immersion. In view of positivity and stability of isolated self-intersections, a self-intersection point of $\vtil$ would force self-intersections of $\vtil_k$ for large values of $k$, which is impossible since $\util_k\in \Lambda(H,P,\lambda,J) \ \forall k$. We conclude that $\vtil$ is an embedding.

Thus, up to a subsequence, $\util_k \to \util \in \Lambda(H,P,\lambda,J)$ in $C^\infty_{\rm loc}$, where $\util =(a,u)$ is the plane defined by $a = b+c$ and $u=v$, where the constant $c$ is given by $c=\lim_k a_k(2)$. The proof is complete.

\section{Existence of fast planes}\label{section_existence}

Our goal in this section is to prove an existence result for special fast planes needed in the proof of our main results, see Proposition~\ref{prop_existence_fast}.

\subsection{Characteristic foliation}

Let $(M,\xi)$ be a co-oriented tight contact manifold with dimension $3$, and let $\lambda$ be any defining contact form for $\xi$. We stress that for this discussion, including the statement of Proposition~\ref{prop_nice_disk}, the contact form $\lambda$ may not be nondegenerate.

If $F\subset M$ is an oriented embedded surface then $(\xi\cap TF)^{d\lambda}$ defines a singular distribution on $F$, as is well-known. It can be parametrized by a smooth vector field $V$ on $F$ vanishing precisely at the so-called singular points $p$ where $\xi|_p = T_pF$. At a singular point $p$ the space $\xi|_p = T_pF$ has two orientations: one coming from $F$ and denoted by $o_p$, and another induced by $d\lambda$ denoted by $o'_p$. Moreover, at such $p$ there is a well-defined linearization $DV_p:T_pF\to T_pF$, and $p$ is a nondegenerate rest point of the dynamics of $V$ if, and only if, $DV_p$ is an isomorphism; in this case $p$ is elliptic if $DV_p$ preserves orientations, nicely elliptic if it is elliptic and $DV_p$ has real eigenvalues, and hyperbolic if $DV_p$ reverses orientations.

If $\partial F$ is connected and transverse to $\xi$ then we orient $F$ by orienting $\partial F$ with the co-orientation of $\xi$. A nondegenerate singular point $p$ is called positive or negative if $o_p=o'_p$ or $o_p=-o'_p$, respectively. The following is contained in~\cite[Chapter 4]{AH}, see also~\cite[Sections 5.1, 5.2]{93} and~\cite[Section 3]{char1}.

\begin{lemma}\label{good_char_fol}
Let $F\subset M$ be an embedded (closed) disk bounding a knot $\partial F$ transverse to $\xi$ and oriented by $\lambda$. If $\sl(\partial F,F)=-1$ then there exists $\epsilon_0>0$ with the following property. Denoting $U_{\epsilon_0} = \{p\in M \mid {\rm dist}_M(p,\partial F)<\epsilon_0\}$, there exists an arbitrarily $C^0$-small perturbation of $F$ into a new embedded disk $F'$ such that $F\cap U_{\epsilon_0}=F'\cap U_{\epsilon_0}$ and  $(\xi\cap TF')^{d\lambda}$ has precisely one singular point $e$. Moreover, $e \in M\setminus\overline U_{\epsilon_0}$, $e$ is positive and it is possible to obtain coordinates $(x,y,z)$ on an arbitrarily small open neighborhood $V$ of $e$ such that $e \simeq (0,0,0)$, $\lambda \simeq dz+xdy$ and $F' \cap V \simeq \{z = -\frac{1}{2}xy\}$.
\end{lemma}

\begin{remark}
The proof of the above statement makes use of Giroux's elimination lemma and relies on the assumption that $\xi$ is tight. The point $e$ given by Lemma~\ref{good_char_fol} is necessarily nicely elliptic.
\end{remark}

Let $R$ denote the Reeb vector field of $\lambda$ and assume that $R$ is tangent to an order $p$ rational unknot $K$.

\begin{definition}\label{def_special_robust}
A $p$-disk $u_0:\D\to M$ for $K$ is said to be   {{\it special robust}} for $(\lambda,K)$ if it satisfies the following properties:
\begin{itemize}
\item[(a)] The singular characteristic distribution of $u_0(\D\setminus\partial\D)$ has precisely one singular point $e$, which is positive. Moreover, it is possible to find coordinates $(x,y,z)$ on an arbitrarily small neighborhood $V$ of $e$ such that $e \simeq (0,0,0)$, $\lambda \simeq dz+xdy$ and $$ u_0(\D) \cap V \simeq \{z = -\frac{1}{2}xy\}. $$
\item[(b)] $\exists \epsilon>0$ such that for every sequence of smooth functions $h_k:M \to (0,+\infty)$ satisfying $h_k \to 1 \ \text{in} \ C^\infty$, $h_k|_{K}\equiv1$ and $dh_k|_{K} \equiv 0$ $\forall k$, there exists $k_0\geq 1$ such that $$ 1-\epsilon < |z| < 1, \ k\geq k_0 \Rightarrow R_k|_{u_0(z)} \not\in d{u_0}|_z(T_z\D) $$ where we denoted by $R_k$ the Reeb vector field associated to $h_k\lambda$.
\end{itemize}
\end{definition}

We get the following statement partly as a consequence of Corollary~\ref{cor_special_spanning_disk} and of Lemma~\ref{good_char_fol}.

\begin{proposition}\label{prop_nice_disk}
Let $x_0$ be a periodic trajectory of the Reeb flow of $\lambda$ such that $x_0(\R)=K$. We orient $x_0(\R)$ by $\lambda$, denote its minimal period by $T_{\rm min}>0$ and set $T_0:=pT_{\rm min}$. Let $u$ be an oriented $p$-disk for $x_0(\R)$ and let $\beta_{\rm disk} \in \Omega^+_{({x_0}_{T_0})^*\xi}$ be induced by a $d\lambda$-symplectic trivialization of $u^*\xi$. Assume that $$ \begin{array}{ccc} \sl(K,u)=  {\frac{-1}{p}} & \text{and} & \rho((x_0,T_0),\beta_{\rm disk}) \neq 1. \end{array} $$ Then there exists a $p$-disk $u_0$ which is special robust for $(\lambda,K)$. Moreover, $u_0$ and $u$ define the same element in $\pi_2(M,K)$.
\end{proposition}

\begin{proof}
There is no loss of generality to assume that $u(e^{i2\pi t}) = x_0(T_0t)$. Consider $\beta_u$ the homotopy class of the oriented (by $d\lambda$) trivializations of $({x_0}_{T_0})^*\xi$ with respect to which the section $Z(t) = \pi \cdot \partial_ru(e^{i2\pi t})$ has zero winding, where $\pi$ is the projection~\eqref{proj_along_Reeb}. Then $\rho((x_0,T_0),\beta_u)\neq 0$ by Corollary~\ref{cor_rel_windings}. Using Corollary~\ref{cor_special_spanning_disk} we find a $p$-disk for $x_0(\R)$ satisfying (b) in Definition~\ref{def_special_robust}, which will be still denoted by $u$.

Denoting $F = u(\{|z|\leq r_0\})$ and $\partial F = u(\{|z|= r_0\})$, for some $r_0<1$ very close to $1$, one computes $\sl(\partial F,F)=-1$: this follows from Lemma~\ref{lemma_self_link_props}. Then by Lemma~\ref{good_char_fol} we can modify $F$ by a $C^0$-small perturbation compactly supported on $F\setminus \partial F$ in order to obtain a new embedded disk ${F_0}$ such that the singular characteristic distribution of ${F_0}$ has a unique singular point $e$ with the properties described in Defintion~\ref{def_special_robust} item (a). Also $e$ is positive and $e\in {F_0}\setminus \partial {F_0}$. By the properties of ${F_0}$ we can patch ${F_0}$ with the strip $u(\{r_0\leq|z|\leq 1\}$ to obtain an immersed disk ${u_0}$ bounding $t\mapsto x_0(T_0t)$. Again by Lemma~\ref{good_char_fol} one can be sure that the immersed disk ${u_0}$ is a $p$-disk for $x_0(\R)$ if the perturbation of $F$ is taken sufficiently $C^0$-small.
\end{proof}

\begin{definition}
Let $u_0$ be a special robust $p$-disk for $(\lambda,K)$. We define
\begin{equation}\label{C_special_robust}
C(\lambda,K,u_0) = 1 + \int_\D |u_0^*d\lambda|.
\end{equation}
\end{definition}

\begin{definition}\label{def_special}
We say that a $p$-disk $u_0:\D\to M$ for $K$ is   {{\it special}} for $(\lambda,K)$ if it satisfies the following properties:
\begin{itemize}
\item[(a)] The singular characteristic distribution of $u_0(\D\setminus\partial\D)$ has precisely one singular point $e$, which is positive. Moreover, it is possible to find coordinates $(x,y,z)$ on an arbitrarily small neighborhood $V$ of $e$ such that $e \simeq (0,0,0)$, $\lambda \simeq dz+xdy$ and $$ u_0(\D) \cap V \simeq \{z = -\frac{1}{2}xy\}. $$
\item[(b)] $\exists \epsilon>0$ such that $1-\epsilon < |z| < 1 \Rightarrow R|_{u_0(z)} \not\in d{u_0}(T_z\D)$ where we denoted by $R$ the Reeb vector field associated to $\lambda$.
\item[(c)] If $x$ is a periodic trajectory of $R$ contained in $u_0(\D)$ then $x(\R) = K$.
\end{itemize}
\end{definition}

\begin{proposition}\label{prop_perturb_special_robust}
Assume that $M$ is compact and let $u_0$ be a special robust $p$-disk for $(\lambda,K)$. Consider a sequence of smooth functions $f_n:M\to (0,+\infty)$, $f_n\to 1$ in $C^\infty$, $f_n|_K \equiv 1$, $df_n|_K \equiv 0$ and $f_n|_V\equiv1$ on an open small neighborhood $V$ of the singular point $e=u_0(0)$ of the characteristic foliation of $u_0(\D\setminus\partial\D)$. Assume that $\lambda_n:=f_n\lambda$ is nondegenerate $\forall n$. There exists an arbitrarily $C^\infty$-small perturbation of $u_0$ into a new special robust $p$-disk $u_0'$ for $(\lambda,K)$ and $n_0$ such that $u_0'$ is special for $(\lambda_n,K)$, for all $n\geq n_0$. Moreover, $$ \int_\D |(u_0')^*d\lambda_n| \leq C(\lambda,K,u_0) \ \ \forall n\geq n_0 $$ where $C(\lambda,K,u_0)$ is the constant~\eqref{C_special_robust}.
\end{proposition}

\begin{proof}
Let $0<\delta<1/2$ and consider the set $$ A_\delta = \{z\in\D : |z| \in [0,\delta] \cup [1-\delta,1] \}. $$ Let $X_\delta$ be the set of smooth maps $\D\to M$ which agree with ${u_0}$ on $A_\delta$. Denote by $R_n$ the Reeb vector field of $\lambda_n$. In view of Definition~\ref{def_special_robust} we find that ${u_0}$ is transverse to $R_n$ on $A_\delta\setminus \partial \D$, when $n$ is large enough and $\delta$ is small enough. Note that the assumptions on $f_n$ guarantee that $\lambda_n|_V \equiv \lambda|_V$. Hence, we find $n_0$ and $\delta$ such that if $P'=(x',T') \in \P(\lambda_n)$ satisfies $x'(\R) \subset {u_0}(\D)$ and $x'(\R)\neq K$ then $x'(\R) \cap {u_0}(A_\delta) = \emptyset$, when $n\geq n_0$. For every $n\geq n_0$ and every periodic $\lambda_n$-Reeb trajectory $\hat x$ satisfying $\hat x(\R) \neq K$ consider the set $X_{\delta,n,\hat x}\subset X_\delta$ of disks $u'$ such that $\hat x(\R) \not\subset u'(\D)$. Clearly $X_{\delta,n,\hat x}$ is open and dense in $X_\delta$. Since each $\lambda_n$ is nondegenerate there are only countably many $\lambda_n$-Reeb trajectories geometrically distinct from $K$. By compactness of $M$ the $C^\infty$-topology in $X_\delta$ is induced by a complete metric. We conclude from Baire's category theorem that $\cap_{n\geq n_0,\hat x} X_{\delta,n,\hat x}$ is dense. Thus ${u_0}$ can be slightly $C^\infty$-perturbed away from $\partial\D\cup\{0\}$ as desired into the new disk $u_0'$ which is clearly special for $(\lambda_n,K)$, $\forall n\geq n_0$. Perhaps after making $n_0$ larger we can assume that $(u_0')^*d\lambda_n$ is $C^\infty$-close to $u_0^*d\lambda$, completing the proof.
\end{proof}

\subsection{Statement of existence result}

\begin{proposition}\label{prop_existence_fast}
Let $\lambda$ be a nondegenerate defining contact form for a tight contact structure $\xi$ on the closed 3-manifold $M$. Suppose that $c_1(\xi)$ vanishes on $\pi_2(M)$ and that there exists a closed Reeb orbit $P_0=(x_0,T_0)$ with minimal positive period~$T_{\rm min}$, such that $K=x_0(\R)$ is an order $p$ rational unknot with self-linking number   {$\frac{-1}{p}$}, where $p=T_0/T_{\rm min}$. Suppose further that $\mu_{CZ}(P_0,\beta_{\rm disk}) \geq 3$. Consider the set $\P^* \subset \P(\lambda)$ of closed Reeb orbits $P$ satisfying
\begin{itemize}
\item $P$ is contractible in $M$,
\item $P$ is contained $M\setminus x_0(\R)$,
\item $\mu_{CZ}(P,\beta_{\rm disk})=2$, or equivalently, $\rho(P,\beta_{\rm disk})=1$.
\end{itemize}
Consider also a $p$-disk $u_0$ which is special for $(\lambda,K)$ in the sense of Definition~\ref{def_special} and set
\begin{equation}\label{constant_C_0}
C_0 = \int_\D |u_0^*d\lambda|.
\end{equation}
If every orbit $P \in \P^*$ satisfies
\begin{center}
$\int_P \lambda > C_0$ or $P$ is not contractible in $M\setminus K$,
\end{center}
then for some $J\in\J_+(\xi)$ there exists an embedded fast $\jtil$-holomorphic finite-energy plane $\util=(a,u)$ asymptotic to $P_0$. Moreover, $K$ is $p$-unknotted in the homotopy class of $u$ in the sense of Definition~\ref{def_p_unknotted_disk}.
\end{proposition}

The remaining paragraphs of this section are devoted to the proof of the above statement.


\subsection{Filling by holomorphic disks}\label{filling_disks}

From now on we assume that $(M,\xi)$, $\lambda$, $P_0=(x_0,T_0)$, $K=x_0(\R)$, $u_0$ satisfy all the hypotheses of Proposition~\ref{prop_existence_fast}. Let us denote by $R$ the Reeb vector field of $\lambda$. We denote by $T_{\rm min}>0$ the minimal positive period of $x_0$ and set $p=T_0/T_{\rm min}$. From now on we denote $$ F_0 := u_0(\D\setminus\partial \D) $$ which is an open disk properly embedded in $M\setminus x_0(\R)$. The unique singular point of its characteristic distribution is denoted by $e$ and there is no loss of generality to assume that
\begin{equation}\label{norm_map_u_0}
\begin{array}{ccc} e=u_0(0), & & u_0(e^{i2\pi t}) = x_0(T_0t). \end{array}
\end{equation}

\subsubsection{Lifting the characteristic foliation}\label{lift_fol}

We need an auxiliary construction in order to adapt the arguments from~\cite{93,char1,char2} to our present context. The reason is that we will work with holomorphic disks having their boundaries mapped to $u_0(\D)$ but this set is not an embedded surface.

Since $u_0|_{\D\setminus\partial \D}$ is an embedding onto $F_0$, the characteristic distribution on $F_0$ can be pulled-back to $\D\setminus\partial \D$. Note that $u_0^*\lambda$ is a smooth $1$-form on $\D$ and $\ker u_0^*\lambda$ extends this pulled-back characteristic distribution to $\D$. The associated singular foliation on $\D$ will be denoted by $\F_0$. The leafs of $\F_0$ hit $\partial\D$ transversely since $u_0(\partial\D)$ is transverse to $\xi$. A leaf of $\F_0$ is either singular and equal to $\{0\}$, or is an embedded copy of a half-closed interval which converges to $0$ along its open end and to a point in $\partial \D$ along its closed end. The length of such a nonsingular leaf is finite because $e$ is nicely elliptic.

This construction yields a smooth submersion
\begin{equation}\label{map_L}
L : \D\setminus\{0\} \to \R/\Z
\end{equation}
given by $L(z) = t$ where the leaf of $\F_0$ through $z$ hits $\partial\D$ at the point $e^{i2\pi t}$. A continuous map $\util=(a,u):\D\to\R\times M$ satisfying $u(\partial\D)\subset F_0$ determines a unique continuous loop
\begin{equation}\label{loop_gamma}
\begin{array}{ccc} \gamma_\util:\R/\Z\to \D & \text{defined by} & \gamma_\util(t) = u_0^{-1} \circ u(e^{i2\pi t}), \ \forall t \end{array}
\end{equation}
since $u_0|_{\D\setminus\partial\D}$ is an embedding. If $\util$ is smooth then so is $\gamma_\util$.

\subsubsection{Consequences of the strong maximum principle}\label{max_prin}

On $\D\setminus\{0\}$ we have polar coordinates $(r,\theta) \simeq re^{i\theta}$ and denote by $\partial_r,\partial_\theta$ the corresponding partial derivatives. Let $\util=(a,u):\D\to\R\times M$ be any non-constant $\jtil$-holomorphic map satisfying $a(\partial\D)\equiv0$, for some $J \in \J_+(\xi)$. The function $a$ is subharmonic and $a|_{\partial\D}\equiv0$, so the strong maximum principle implies
\begin{equation}\label{str_max_princ_conseq}
\partial_ra|_{\partial\D} = (\lambda(u)\cdot \partial_\theta u|_{\partial\D}) > 0.
\end{equation}
In particular, if $u(\partial\D) \subset F_0 = u_0(\D\setminus\partial \D)$ then $\gamma_\util$ defined by~\eqref{loop_gamma} does not touch $0$ and is transverse to $\F_0$. Moreover, $\gamma_\util$ is an embedding if, and only if, it winds once around $0$, and $L\circ \gamma_\util:\R/\Z\to\R/\Z$ is an orientation preserving covering map which is a diffeomorphism if, and only if, its degree is $1$. Here we used the map $L$~\eqref{map_L}.

\subsubsection{Defining the Bishop family}\label{def_bishop_family}

Following~\cite{char1} we consider for any $J \in \J_+(\xi)$ the boundary value problem
\begin{equation}\label{bvp_bishop}
\begin{aligned}
& \util = (a,u) : \D \to \R\times M \ \text{satisfies} \ \bar\partial_{\jtil}(\util)=0, \\
& \util \text{ is an embedding,} \ a|_{\partial \D}\equiv0 \ \text{and} \ u(\partial\D) \subset F_0, \\
& t\mapsto u(e^{i2\pi t}) \ \text{winds once around} \ e.
\end{aligned}
\end{equation}
The last condition is understood as follows: by~\eqref{str_max_princ_conseq} the loop $\gamma_\util$ defined in~\eqref{loop_gamma} does not touch $0$, and we ask that it winds once around $0$. As remarked before, this is equivalent to $L\circ\gamma_\util$ having degree $1$, and when this is the case we know that $\gamma_\util$ is actually embedded in $\D\setminus (\partial\D\cup\{0\})$. We shall be interested in the subset of solutions $\util=(a,u)$ of~\eqref{bvp_bishop} satisfying the following extra condition:
\begin{enumerate}
\item[(C)] If $D\subset \D$ is the disk bounded by $\gamma_\util$ then $u:\D\to M$ and $u_0|_D:D\to M$ are homotopic keeping their common boundaries $u(\partial\D)=u_0(\partial D)$ fixed.
\end{enumerate}
The set of solutions of~\eqref{bvp_bishop} which in addition satisfy (C) above is denoted by $\M(J)$ and will be endowed with the $C^\infty$-topology. It follows easily from the definition of Hofer's energy that
\begin{equation}\label{crucial_energy_estimate}
E(\util) = \int_\D u^*d\lambda \leq C_0, \ \ \ \forall \util=(a,u) \in \M(J)
\end{equation}
where $C_0>0$ is the constant satisfying~\eqref{constant_C_0}.

Hofer~\cite{93} showed that $\M(J)\neq \emptyset$ for some $J$. In fact, consider the coordinates $(x,y,z)$ near the point $e$ described in Definition~\ref{def_special_robust}. In these coordinates the vectors $\partial_x$, $\partial_y-x\partial_z$ span $\xi$ and we find $J$ satisfying $J\cdot \partial_x=\partial_y-x\partial_z$ near $e$. A $1$-parameter family of solutions of~\eqref{bvp_bishop} is explicitly given in these coordinates by $$ s+it \in\D \mapsto \left(\frac{\tau^2}{4}(s^2+t^2-1),\tau s,\tau t,-\frac{\tau^2}{2}st\right). $$ Note that as $\tau\to0^+$ these solutions converge to the constant $(0,e)$.

The solutions in $\M(J)$ are automatically Fredholm regular, and the linearized Caucy-Riemann operator at these (parametrized) solutions has index $4$. This is proved in~\cite{93}. There is a $3$-dimensional group $\Mob$ of holomorphic self-diffeomor\-phisms of $\D$ acting freely and properly on $\M(J)$. This shows that $\M(J)/\Mob$ is the $1$-dimensional base space of the principal bundle $\Pi:\M(J) \to \M(J)/\Mob$ with group $\Mob$; here $\Pi$ is the quotient projection.

A map $\util:\D\to\R\times M$ is called an   {{\it embedding near the boundary}} if $\exists \epsilon<1$ such that $\util^{-1}(\util(\{1-\epsilon\leq|z|\leq1\})) = \{1-\epsilon\leq|z|\leq1\}$ and $\util|_{\{1-\epsilon\leq|z|\leq1\}}$ is an embedding. With this notion in mind,   {we prove the following}

\begin{lemma}\label{lemma_intersections}
Let $\util_k \in \M(J)$ converge to some $\util=(a,u):\D\to\R\times M$ in $C^\infty$ which is non-constant and satisfies $u(\partial\D) \subset F_0$. Then $\util \in \M(J)$.
\end{lemma}

\begin{proof}
Clearly $\bar\partial_{\jtil}(\util)=0$. The curve $\gamma_\util$ is well-defined and, since $\util$ is not constant, the maximum principle tells us that $\gamma_\util$ does not pass through $0$. Thus $t\mapsto L\circ \gamma_\util(t)$ has a well-defined degree. By continuity $L\circ\gamma_\util$ must have degree $1$. Then $u|_{\partial\D}$ is an embedding because, as remarked before, so is $\gamma_\util$. This fact together with~\eqref{str_max_princ_conseq} implies that $\util$ is an embedding near the boundary.   {Since $\util$ is the $C^\infty$-limit of the embeddings $\util_k$ and is an embedding near the boundary, } results of~\cite{dusa}   {imply that $\util$ is} an embedding   {as well}. Condition (C) is obviously closed under limits.
\end{proof}

Let $\ell$ be the length of the leaf $L^{-1}(0)$ of $\F_0$. There is a map
\begin{equation}\label{length_function}
\tau : \M(J)/\Mob \to (0,\ell)
\end{equation}
defined as follows. For every $\util=(a,u) \in \M(J)$ the map $L \circ \gamma_\util:\R/\Z\to\R/\Z$ is a diffeomorphism and $\gamma_\util$ hits $L^{-1}(0)$ at a unique point $z_*$. Then we set $\tau(\Pi(\util))$ to be the length of the piece of $L^{-1}(0)$ connecting $0$ to $z_*$. It is easy to check that $\tau$ is smooth. In~\cite{93} it is proved that if $s$ is a smooth section of $\M(J)$ defined on a small open neighborhood $U$ of some $y_0\in \M(J)/\Mob$, then
\begin{equation}\label{local_embedding_93}
\begin{array}{ccc} \Phi : U \times \D \to \R\times M & \text{defined by} & (y,z) \mapsto s(y)(z) \end{array}
\end{equation}
is a smooth embedding onto its image. Thus $\tau$ is a local diffeomorphism.

\subsubsection{Intersections}

We start with a technical lemma which is contained in the proof of Theorem 2.1 from~\cite{char1}. Consider coordinates $(z_1,z_2)$ on $\C \times B$ where $B\subset\C$ is the unit open ball centered at $0$. In the following we will write $B^+_r = \{ z\in \C \mid |z|<r, \ \Im z\geq0 \}$.

\begin{lemma}\label{lemma_int}
Let $j$ be a smooth almost complex structure on $\C \times B$ satisfying
\begin{equation}\label{hyp_j}
j|_{\C \times\{0\}} = \begin{pmatrix} i & 0 \\ 0 & i \end{pmatrix} \in \mathcal L_\R(\C^2).
\end{equation}
Let $v = (v_1,v_2):B^+_r \to \C \times B$ be a smooth $j$-holomorphic map satisfying 
\begin{equation}\label{hyp_v_2}
\begin{array}{cccc}
v_2((-r,r)) \subset [0,1), & v_2(0)=0 & \text{and}  & v_2 \not\equiv 0. \end{array}
\end{equation}
If there is a sequence of continuous maps $v_2^k : B^+_r \to \C$ satisfying
\begin{equation}\label{hyp_v_k_2}
v^k_2((-r,r)) \subset [0,+\infty), \ \ v_2^k \to v_2 \text{ in } C^0_{\rm loc},
\end{equation}
then for $k$ large $\exists \zeta_k \in B_r^+$ such that $v^k_2(\zeta_k)=0$.
\end{lemma}

For a proof see the Appendix~\ref{app_lemma}.

\begin{corollary}\label{cor_int}
Let $J\in\J_+(\xi)$. Consider a point $z_0 \in \partial\D$, an open neighborhood $U$ of $z_0$ in $\D$ and a sequence $\util^k=(a^k,u^k):U\to\R\times M$ of $\jtil$-holomorphic maps satisfying $u^k(U\cap\partial\D) \subset F_0 = u_0(\D\setminus\partial\D) \ \forall k$. Assume that $\util=(a,u):U\to\R\times M$ is a $\jtil$-holomorphic map such that $\util^k \to \util$ in $C^\infty_{\rm loc}(U)$ and $\int_U u^*d\lambda>0$. If $u(z_0) \in x_0(\R)$ then for $k$ large enough there exists $\zeta_k \in U$ satisfying $u^k(\zeta_k) \in x_0(\R)$.
\end{corollary}

\begin{proof}
It follows from the hypotheses that $u(\partial\D\cap U) \subset u_0(\D)$. Up to a subsequence, we can assume that $u_0^{-1}(u^k(z_0))$ converges to a point $z_* \in \partial\D$ such that $u_0(z_*) = u(z_0)$.

Consider a small tubular neighborhood $N$ of $x_0(\R)$ equipped with coordinates $(\theta,x+iy) \in \R/\Z\times B$, where $B\subset \C$ is the open unit ball centered at $0$. We can construct these coordinates in such a way that 
\begin{itemize}
\item[($\rm a_1$)] $x_0(\R) \simeq \R/\Z\times \{0\}$, $\lambda|_{\R/\Z\times\{0\}} \simeq d\theta$, $\xi|_{\R/\Z\times\{0\}} \simeq \{0\}\times \C$ and $J|_{\R/\Z\times\{0\}} \simeq i$.
\item[($\rm a_2$)] $\exists$ open neighborhood $V$ of $z_*$ on $\D$ such that $u_0(V)$ corresponds to a subset of $\R/\Z \times [0,1)$.
\end{itemize}
Thus $\R\times N$ has coordinates $(a,\theta,z_2=x+iy) \in \R\times\R/\Z\times B$ which we lift to $\R^2\times B$. From now on we identify $\R^2\times B \simeq \C\times B$ by $(a,\theta,z_2) \simeq (z_1 = a+i\theta,z_2)$. 

Without loss of generality assume that $T_{\rm min}=1$. Pulling back $\jtil|_{\R\times N}$ to $\R\times\R/\Z\times B$ and lifting to $\C\times B$ we obtain an almost complex structure $j$ satisfying~\eqref{hyp_j}. Possibly after shrinking $U$ we can assume that $u^k(U\cap\partial\D) \subset u_0(V)$ and find a biholomorphism $$ (U,U\cap\partial\D,z_0) \simeq (B_1^+,(-1,1),0). $$ Composing $\util^k|_U,\util|_U$ with this biholomorphism we obtain maps $v^k,v$ defined on $B_1^+$. From now on we equip $B_1^+$ with a complex coordinate $s+it$. Moreover, using our coordinates we can represent $v^k,v$ as maps taking values on $\R\times\R/\Z\times B$, which can be lifted to maps taking values on $\C\times B$. Summarizing $v^k,v$ are $j$-holomorphic maps defined on $B_1^+$ and if we write $v=(v_1,v_2)$, $v^k=(v^k_1,v^k_2)$ then $v^k_2$ satisfies~\eqref{hyp_v_k_2} because $v^k((-1,1)) \subset \C\times[0,1)$. Thus $v_2$ satisfies~\eqref{hyp_v_2}. By Lemma~\ref{lemma_int} we find $\zeta_k \in B_1^+$ such that $v^k_2(\zeta_k) =0$, for $k$ large. This amounts to saying that $\util^k$ intersects $\R\times x_0(\R)$.
\end{proof}

\begin{theorem}\label{bishop_disks_dont_intersect}
Let $\mathcal Y \subset \M(J)/\Mob$ be a connected component containing disks close to the constant $(0,e)$. If $\util:\D\to \R\times M$ is a smooth map and $\util_k \in \Pi^{-1}(\mathcal Y)$ is a sequence satisfying $\util_k \to \util$ in $C^\infty$ then $\util(\D) \cap (\R\times x_0(\R)) = \emptyset$.
\end{theorem}

\begin{proof}
The set $\mathcal C$ of disks $\vtil \in \Pi^{-1}(\mathcal Y)$ satisfying $\vtil(\D) \cap (\R\times x_0(\R)) \neq\emptyset$ is clearly closed. For any $\vtil \in \mathcal C$ there exists $z\in\D$ such that $\vtil(z) \in \R\times x_0(\R)$ and, by definition of $\M(J)$, we have $z\in \D\setminus\partial\D$. By positivity and stability of intersections, $\vtil$ is an interior point of $\mathcal C$. This shows that $\mathcal C$ is also open. It follows that $\mathcal C = \emptyset$ because $\Pi^{-1}(\mathcal Y)$ is connected and contains disks close to $(0,e)$. 

Consider $\util,\util_k$ as in the statement. Clearly $\util$ is $\jtil$-holomorphic. If there exists $z\in\D\setminus\partial\D$ such that $\util(z) \in \R\times x_0(\R)$ then, arguing as above using positivity and stability of intersections, $\util_k(\D)\cap\R\times x_0(\R)\neq\emptyset$ for $k$ large, a contradiction. If $\util(z) \in \R\times x_0(\R)$ for some $z\in\partial \D$ then we use Corollary~\ref{cor_int} to find that $\util_k(\D)\cap\R\times x_0(\R)\neq\emptyset$ for $k$ large, again a contradiction.
\end{proof}

\subsubsection{$C^1$-estimates near the boundary}

The following statement is a version of the non-trivial Theorem 2.1 from~\cite{char1} adapted to our set-up. Here ${\rm dist}_M$ denotes any fixed distance function on $M$.


\begin{theorem}\label{C1_estimates}
Let $\mathcal Y \subset \M(J)/\Mob$ be a component containing disks close to $(0,e)$, and $\util_k=(a_k,u_k) \in \Pi^{-1}(\mathcal Y)$, $k\geq 1$, be disks satisfying $\gamma_{\util_k}(0) \in L^{-1}(0)$, $\gamma_{\util_k}(1/4) \in L^{-1}(1/4)$, $\gamma_{\util_k}(1/2) \in L^{-1}(1/2)$ and $\inf_k {\rm dist}_M(u_k(\partial\D),e) > 0$. Then there exists $\epsilon>0$ such that $\sup_k \ \sup \{|d\util_k(z)| : 1-\epsilon\leq|z|\leq1\} < \infty$.
\end{theorem}

\begin{proof}
As remarked after the proof of Lemma~\ref{lem:lpqmonodromy}, a regular neighborhood $\mathcal N$ of the (image of the) $p$-disk $u_0(\D)$ is diffeomorphic to $L(p,q)\setminus B^3$, where $B^3$ denotes a $3$-ball. Hence $M\simeq L(p,q)\# M'$ for some $3$-manifold $M'$, and $\mathcal N$ is $p$-covered by the complement in $S^3$ of $p$ disjoint copies of~$B^3$ which we denote for simplicity by $S^3\setminus pB^3$. There is absolutely no restriction on what manifold $M'$   {can} be, but we can use these facts to construct a $p$-covering space over $M$ by gluing in $p$ copies of $M'\setminus B^3$ with $S^3\setminus pB^3$ to obtain a manifold $\widetilde M$. In fact, the group of deck transformations of the $p$-covering $S^3\setminus pB^3 \to \mathcal N$, which is isomorphic to $\Z_p$, permute the $p$ copies of $S^2$ composing $\partial(S^3\setminus pB^3)$ and, consequently, its generator can be extended to a diffeomorphism of $\widetilde M$ which interchange the $p$ copies of $M'\setminus B^3$ accordingly. This diffeomorphism generates a free action of $\Z_p$ on $\widetilde M$ and $M$ is the orbit space. The quotient map is a $p$-covering map denoted by $\widetilde \pi$ with group of deck transformations $\Z_p$. The lifted contact structure to $\widetilde M$ will be denoted by $\widetilde \xi$, and the lifting of the contact form $\lambda$ to $\widetilde M$ by $\widetilde \lambda$. 

The map $u_0$ lifts to an embedding $U_0:\D\to \widetilde M$. The Reeb flow of $\widetilde\lambda$ projects onto the Reeb flow of $\lambda$, so that the closed $\lambda$-Reeb orbit $(x_0,T_0)$ is lifted to a closed $\widetilde\lambda$-Reeb orbit $(\widetilde x_0,T_0)$ with minimal positive period $T_0$. The image $\mathcal D_0 = U_0(\D)$ is an embedded disk with boundary $\widetilde x_0(\R)$, and the transverse unknot $\widetilde x_0(\R)$ has self-linking number $-1$. This last claim follows from Lemma~\ref{lemma_self_link_props}. The characteristic foliation $(\widetilde \xi\cap T\mathcal D_0)^\bot$ has precisely one sigular point $\widetilde e$, which is nicely elliptic, and one finds suitable coordinates near $\widetilde e$ where $\widetilde \lambda$ and $\mathcal D_0$ assume a precise normal form exactly as for the original data in $M$. The Bishop disks $\util_k$ lift to holomorphic disks $\widetilde U_k$ with boundary on $\mathcal D_0$ since the almost complex structure lifts to $\R\times\widetilde M$ keeping $\R$-invariance. At this point we note that we can apply the non-trivial Theorem~2.1 from~\cite{char1} to find $\epsilon>0$ such that $$ \sup_k \{ |d\widetilde U_k(z)| : 1-\epsilon\leq |z| \leq 1 \} < \infty. $$ Here we lifted the $\R$-invariant Riemannian metric from $\R\times M$ to $\R\times \widetilde M$. The desired conclusion about the disks $\util_k$ follows.
\end{proof}

\subsubsection{Generic almost complex structures}\label{fred_argument}

Consider a finite set $\Gamma \subset \D\setminus\partial \D$ and some $J\in\J_+(\xi)$. Consider the following mixed boundary value problem studied in~\cite{props3}:
\begin{equation}\label{mixed_bnd_value_prob}
\left\{
\begin{aligned}
& \util = (a,u) : \D\setminus \Gamma \to \R\times M \ \text{is an embedding}, \\
& \bar\partial_{\jtil}(\util)=0 \ \text{and} \ 0<E(\util)<\infty, \\
& a\equiv0 \ \text{on} \ \partial\D \ \text{and} \ u(\partial\D) \subset F_0 = u_0(\D\setminus \partial\D), \\
& \int_{\D\setminus\Gamma} u^*d\lambda>0 \ \text{and each} \ z\in\Gamma \ \text{is a negative puncture}, \\
& t\mapsto u(e^{i2\pi t}) \ \text{winds once around} \ e.
\end{aligned}
\right.
\end{equation}
Here the complex structure on $\D$ is not normalized in any way. We need the following statement which is a consequence of the Fredholm theory developed in~\cite{props3}.

\begin{proposition}\label{prop_generic_J}
There exists a dense set $\J_{\rm reg} \subset \J_+(\xi)$ with the following property. If $J\in\J_+(\xi)$ and $\util$ is a $\jtil$-holomorphic solution of~\eqref{mixed_bnd_value_prob} such that at each $z\in\Gamma$ the map $\util$ is asymptotic to a contractible closed Reeb orbit $P_z$ satisfying $\mu_{CZ}(P_z,\beta_{\rm disk})\geq 2$, and there is at least one $z^*\in\Gamma$ satisfying $\mu_{CZ}(P_{z^*},\beta_{\rm disk})\geq 3$, then $J \not\in \J_{\rm reg}$.
\end{proposition}

The non-trivial proof follows a well-known pattern, see~\cite{props3} for more details.

\subsubsection{Limiting behavior of the Bishop family}

According to \S~\ref{def_bishop_family} one finds $J \in \J_+(\xi)$ and a disk $\util_* \in \M(J)$ arbitrarily $C^0$-close to the constant $(0,e)$. Every such $\util_*$ is automatically Fredholm regular for a Fredholm theory of pseudo-holomorphic disks with boundary in the embedded surface $\{0\}\times (F_0\setminus\{e\})$ which is totally real in $(\R\times M,\jtil)$; this is proved exactly as in~\cite{93}. If $\J_{\rm reg} \subset \J_+(\xi)$ is the set given by Proposition~\ref{prop_generic_J} and $J'\in\J_{\rm reg}$ is a $C^\infty$-small perturbation of $J$ then we can find a disk $\util'_* \in \M(J')$ as a small $C^\infty$-perturbation of $\util_*$. Reverting the notation back to $J$ and $\util_*$, it follows that we could have assumed $J\in\J_{\rm reg}$ from the beginning.

Write $\util_*=(a_*,u_*)$. Consider the connected component $\mathcal Y \subset \M(J)/\Mob$ containing $\Pi(\util_*)$ and choose $0<\bar\delta_0<{\rm dist}_M(u_*(\partial\D),e)$. We define
\begin{equation}
\bar \ell = \sup \{ \tau(\Pi(\util)) \mid \util \in \Pi^{-1}(\mathcal Y) \ \text{and} \ {\rm dist}_M(u(\partial\D),e)\geq\bar \delta_0 \ \text{where} \ \util=(a,u) \}.
\end{equation}
Here $\tau$ denotes the local diffeomorphism~\eqref{length_function}. The next step is to follow Hofer, Wysocki and Zehnder~\cite{char1,char2} closely and prove

\begin{proposition}\label{prop_limit_bishop_disks}
Under the hypotheses of Proposition~\ref{prop_existence_fast} the following holds. If $\util_k = (a_k,u_k) \in \Pi^{-1}(\mathcal Y)$ is a sequence satisfying $$ \begin{array}{cc} \tau(\util_k) \to \bar \ell, & \inf_k {\rm dist}_M(u_k(\partial\D),e)\geq \bar\delta_0 \end{array} $$ and the normalization conditions $\gamma_{\util_k}(\tau) \in L^{-1}(\tau)$ for $\tau\in\{0,1/4,1/2\}$ 
then, up to selection of subsequence still denoted $\util_k$, the set
\begin{equation}\label{bubb_off_pts_limit_bishop}
\Gamma := \{ z\in \D \mid \exists k_j \to \infty \ \text{and} \ z_j \to z \ \text{such that} \ |d\util_{k_j}(z_j)|\to\infty \}
\end{equation} consists of a single point in $\D\setminus \partial\D$. After composing with suitable biholomorphisms of $\D$ we may assume $\Gamma = \{0\}$ so that, up to a further subsequence, $\util_k$ converges to a finite-energy non-constant $\jtil$-holomorphic map $\util_\infty = (a_\infty,u_\infty) :\D\setminus\{0\}\to \R\times M$. This map satisfies $$ \int_{\D\setminus\{0\}} u_\infty^*d\lambda = 0 $$ and $\exists c\in\R$ such that $\util_\infty(e^{2\pi(s+it)}) = (T_0s,x_0(T_0t + c))$ $\forall(s,t) \in (-\infty,0]\times \R/\Z$. In particular we get $\bar\ell = \ell$, where $\ell$ is the length of $L^{-1}(0)$.
\end{proposition}

We now turn to the proof of Proposition~\ref{prop_limit_bishop_disks} and consider a sequence $\util_k=(a_k,u_k)$ as in the above statement. According to Theorem~\ref{C1_estimates} the set $\Gamma$ defined in~\eqref{bubb_off_pts_limit_bishop} is contained in $\D\setminus\partial\D$. After a rescaling using Hofer's lemma and selection of a subsequence, each point in $\Gamma$ produces a finite-energy plane which takes a positive quantum of $d\lambda$-area depending on the minimal period among all closed Reeb orbits. Thus, up to a subsequence still denoted by $\util_k$, we may assume that $\Gamma$ is finite since, otherwise, we would obtain a contradiction to the fact that there is a uniform bound on the $d\lambda$-area of the disks in the Bishop family.

Thus, again up to a subsequence, we may assume that $\util_k$ converges in $C^\infty_{\rm loc}(\D\setminus\Gamma)$ to a finite-energy $\jtil$-holomorphic map
\begin{equation}\label{limit_bishop_disks}
\util_\infty = (a_\infty,u_\infty) :\D\setminus \Gamma \to \R\times M.
\end{equation}
We claim that $\util_\infty$ is not constant. In fact, if $\util_\infty$ is constant then, since $\Gamma\cap\partial\D = \emptyset$, we would be able to conclude that the loops $t\mapsto u_k(e^{i2\pi t})$ do not wind around $e$ for large values of $k$, and this is absurd.

Next we claim that $\Gamma\neq\emptyset$. In fact, if $\Gamma=\emptyset$ then $\util_\infty$ is a non-constant disk and, as such, it must necessarily satisfy $$ \int_\D u_\infty^*d\lambda>0. $$ If $x_0(\R) \cap u_\infty(\partial\D) \neq\emptyset$ then by Corollary~\ref{cor_int} we get intersections of $u_k(\D)$ with $x_0(\R)$, contradicting Theorem~\ref{bishop_disks_dont_intersect}. Hence $u_\infty(\partial\D) \subset F_0$. By Lemma~\ref{lemma_intersections} we must have $\util_\infty \in \mathcal Y$. Clearly ${\rm dist}_M(u_\infty(\partial\D),e)\geq \bar\delta_0$ and $\tau(\util_\infty) = \bar\ell$. Since~\eqref{local_embedding_93} is an embedding we get, using the implicit function theorem, a contradiction to the definition of $\bar\ell$ which shows that $\Gamma\neq\emptyset$.

\begin{lemma}\label{zero_dlambda_area}
$\int_{\D\setminus\Gamma} u_\infty^*d\lambda=0$.
\end{lemma}

\begin{proof}
Assume, by contradiction, that $\pi \circ du_\infty$ does not vanish identically. The map $\util_\infty$ must be an embedding. To see this first note that it must be somewhere injective. In fact, as explained before, we must have $u_\infty(\partial\D) \cap x_0(\R)=\emptyset$ because, otherwise, Corollary~\ref{cor_int} would give intersections of $\util_k(\D)$ with $\R\times x_0(\R)$ for $k$ large, which is again absurd. Here the fact that $\pi \circ du_\infty$ does not vanish identically was strongly used. Thus $t\mapsto u_\infty(e^{i2\pi t})$ winds once around $e$ inside $F_0$. These facts and a strong maximum principle for $a_\infty$ imply that $\util_\infty$ is an embedding near $\partial\D$ and, in particular, $\util_\infty$ is somewhere injective. Now results of McDuff~\cite{dusa} will show that self-intersections of $\util_\infty$ must be isolated,   {and} moreover, self-intersections or critical points of $\util_\infty$ will force self-intersections of $\util_k$ for large values of $k$, which is impossible. Thus $\util_\infty$ is an embedding.

It is easy to conclude that every point in $\Gamma$ must be a negative puncture since, by the maximum principle, $\sup_k a_k(\D)\leq 0$. Fix $z\in \Gamma$ and let $P$ be the asymptotic limit of $\util_\infty$ at $z$. Arguing as in Section~\ref{section_bubb_off_analysis} and doing ``soft-rescaling'' at any $z\in\Gamma$ one obtains a germinating sequence having a limit $\vtil$. The punctured finite-energy sphere $\vtil$ is asymptotic at its unique positive puncture to $P$, and by Lemma~\ref{lemprinc} we must have $\mu_{CZ}(P,\beta_{\rm disk})\geq 2$. Since $J\in\J_{\rm reg}$ then by Proposition~\ref{prop_generic_J} we get $\mu_{CZ}(P,\beta_{\rm disk})=2$ for every such $P$. If $\vtil$ has a negative puncture then this germinating sequence having $\vtil$ as a limit satisfies the hypotheses of Proposition~\ref{main_prop_compactness} item (a). Hence we find a finite-energy plane $\util_*=(a_*,u_*)$ asymptotic to a closed Reeb orbit $P_*$ satisfying $\mu_{CZ}(P_*,\beta_{\rm disk})=2$ and $u_*(\C) \cap x_0(\R) = \emptyset$. If $\vtil$ has no negative punctures then we obtain the same conclusion setting $\util_*=\vtil$ and $P_*=P$: it only needs to be proved that $u_*(\C) \cap x_0(\R)=\emptyset$ which follows easily by positivity of intersections since the sets $u_k(\D)$ do not intersect $x_0(\R)$. If $P_*=(x_0,jT_{\rm min})$ then $1\leq j<p=T_0/T_{\rm min}$ because $\mu_{CZ}(P_0,\beta_{\rm disk})\geq 3$, see the proof of Lemma~\ref{lemindex}. Hence $u_*$ provides a disk for the $j$-th iterate of $(x_0,T_{\rm min})$ contradicting Lemma~\ref{lemma_no_p'_disk} since $j<p$. This shows that $P_* \subset M\setminus x_0(\R)$ and $P_*$ is contractible in $M\setminus x_0(\R)$, contradicting the assumptions of Proposition~\ref{prop_limit_bishop_disks}.
\end{proof}

With the help of Lemma~\ref{zero_dlambda_area} we conclude that $u_\infty(\partial\D)$ is a closed Reeb trajectory contained in $u_0(\D)$. Thus $u_\infty(\partial\D)=x_0(\R)$ in view of property (c) in Definition~\ref{def_special} of the $p$-disk $u_0$. In particular $\bar\ell=\ell$. By the similarity principle we find a holomorphic map $g:\D\to\D$ satisfying
\begin{itemize}
\item $g^{-1}(\partial\D)=\partial\D$
\item $g|_{\partial\D}: \partial\D\to \partial\D$ has degree $p$
\item $\Gamma = g^{-1}(0)$
\end{itemize}
such that
\begin{equation}\label{half_trivial_cyl}
\util_\infty = F\circ g
\end{equation}
where $F:\D\setminus\{0\} \to \R\times M$ is the map $$ F(s,t) = (T_{\rm min}s,x_0(T_{\rm min}t)). $$

\begin{lemma}\label{lemma_1_puncture_1}
$\#\Gamma=1$.
\end{lemma}

\begin{proof}
Arguing indirectly, assume that $\#\Gamma\geq 2$ and let $z_*\in\Gamma$. If $r>0$ is fixed small enough then the loop $\gamma_\infty:\R/\Z\to M$ given by $\gamma_\infty(t)= u_\infty(z_*+re^{i2\pi t})$ is a reparametrization of the orbit $P' = (x_0,p'T_{\rm min})$ along the Reeb vector field, for some $1\leq p'<p$. This follows from~\eqref{half_trivial_cyl}. For each $k$ denote by $\gamma_k$ the loop $\gamma_k(t) = u_k(z_*+re^{i2\pi t})$. The loops $\gamma_k$ converge to $\gamma_\infty$ in $C^\infty$ since $\util_k$ converges to $\util_\infty$ in $C^\infty_{\rm loc}(\D\setminus \Gamma)$. Thus $\gamma_\infty$ is contractible since each $\gamma_k$ is. This contradicts Lemma~\ref{lemma_no_p'_disk} and concludes the proof. 
\end{proof}

By the above lemma we can compose $g$ with a M\"obius transformation and assume, without loss of generality, that $\Gamma=\{0\}$ and $g(z) = \tilde z z^p$ for some $\tilde z$ satisfying $|\tilde z|=1$. The proof of Proposition~\ref{prop_limit_bishop_disks} is complete.

\subsection{Obtaining the fast plane}

Let $\lambda$ be a defining contact form for $(M,\xi)$ satisfying the hypotheses of Proposition~\ref{prop_existence_fast}. More precisely, $\xi$ is tight, $c_1(\xi)$ vanishes on $\pi_2(M)$, $\lambda$ is nondegenerate and induces the {\it a priori} given co-orientation on $\xi$, and there exists a special closed Reeb orbit $P_0=(x_0,T_0)$ so that $x_0(\R)$ is an order $p$ rational unknot with self-linking number   {$\frac{-1}{p}$} with respect to some $p$-disk, which has monodromy relatively prime with $p$ by Lemma~\ref{lemma_top}. We always consider $x_0(\R)$ oriented by $\lambda$. Moreover, every closed Reeb orbit $P \subset M\setminus x_0(\R)$ which is contractible in $M$ and satisfies $\mu_{CZ}(P,\beta_{\rm disk})=2$, is not contractible in $M\setminus x_0(\R)$ or has action larger than $C_0$. Here $C_0$ is the constant~\eqref{constant_C_0}.

We recall some of the arguments so far. Applying Proposition~\ref{prop_nice_disk} we find a special oriented $p$-disk $u_0 : \D\to M$ for $P_0$ having a unique singular point $e$ for its characteristic distribution, which is nicely elliptic and serves as a starting point for a Bishop family of pseudo-holomorphic disks. Such a family of disks is non-empty for some $J \in \J_+(\xi)$, as explained in \S~\ref{filling_disks}. We may further assume that $J$ belongs to the special set $\J_{\rm reg}$ given by Proposition~\ref{prop_generic_J}, and that the Bishop family contains disks arbitrarily $C^0$-close to the constant $(0,e) \in \R\times M$. It follows from Theorem~\ref{bishop_disks_dont_intersect} that such Bishop disks never intersect $x_0(\R)$.

By Proposition~\ref{prop_limit_bishop_disks} we find a sequence of solutions of~\eqref{bvp_bishop} $\util_k$ satisfying  condition (C) described in \S~\ref{def_bishop_family}, and also satisfying $$ \util_k \to \util_\infty \ \text{ in } \ C^\infty_{\rm loc}(\D\setminus \{0\}). $$ Here $\util_\infty:\D\setminus\{0\}\to\R\times M$ is the map defined by $\util_\infty(e^{2\pi(s+it)}) = (T_0s,x_0(T_0t))$. Thus $0$ is a bubbling-off point for the sequence $\util_k$ with mass $T_0$.

Proceeding as in \S~\ref{secsoft} we can do ``soft-rescaling'' at the puncture $0$. This yields sequences $z_k\to0$, $c_k\in \R$, $R_k\to\infty$ and $\delta_k\to0^+$ such that $R_k\delta_k\to 0$ and, up to a subsequence, the sequence of maps
\begin{equation}\label{germ_seq_below_bishop}
\vtil_k = (b_k,v_k) : B_{R_k}(0) \to \R\times M
\end{equation}
defined by
\begin{equation*}
\begin{array}{ccc} b_k(z) = a_k(z_k+\delta_kz)-c_k & & v_k(z) = u_k(z_k+\delta_kz) \end{array}
\end{equation*}
converges in $C^\infty_{\rm loc}(\C\setminus \Gamma_0)$ to some non-constant finite-energy $\jtil$-holomorphic map
\begin{equation}\label{root_below_bishop}
\vtil = (b,v) : \C\setminus \Gamma_0 \to \R\times M.
\end{equation}
Here $\Gamma_0 \subset \D$ is a finite set which consists of negative punctures of $\vtil$, and $\Gamma_0\neq\emptyset \Rightarrow 0\in\Gamma_0$, see Remark~\ref{remper}. Moreover, $\vtil$ has a unique positive puncture at $\infty$ where it is asymptotic to $P_0$. As explained in Remark~\ref{remper}, the arguments from~\cite{fols} tell us that if $\Gamma_0 = \{0\}$ then $\pi \circ dv$ does not vanish identically, where $\pi$ is the projection~\eqref{proj_along_Reeb}.

\begin{lemma}\label{lemma_1_puncture_2}
$\int_{\C\setminus\Gamma_0} v^*d\lambda > 0$.
\end{lemma}

\begin{proof}
If $\pi\circ dv$ vanishes identically then, as observed above, we have $\#\Gamma_0\geq 2$. Moreover, by the similarity principle we find a polynomial $g(z)$ of degree at least $2$ such that $\Gamma_0 = g^{-1}(0)$ and $\vtil = F\circ g$ where $F:\C \setminus\{0\} \to \R\times M$ is the map $F(e^{2\pi(s+it)})=(T_{\rm min}s,x_0(T_{\rm min}t))$, and $T_{\rm min}>0$ is the minimal period of $x_0$. The proof now follows the same pattern as that of Lemma~\ref{lemma_1_puncture_1}. Fixing $z_* \in \Gamma_0$ and $r>0$ small enough, the loop $\gamma_\infty:\R/\Z \to M$ given by $\gamma_\infty(t) = v(z_*+re^{i2\pi t})$ is a reparametrization of $P'=(x_0,p'T_{\rm min})$ along the Reeb vector field, for some $1\leq p'<p$. The loops $\gamma_k(t) = v_k(z_*+re^{i2\pi t})$ converge in $C^\infty$ to $\gamma_\infty$ since $\vtil_k \to \vtil$ in $C^\infty_{\rm loc}(\C\setminus\Gamma_0)$. Since each $\gamma_k$ is contractible we conclude that so is $\gamma_\infty$, contradicting Lemma~\ref{lemma_no_p'_disk} because $p'<p$.
\end{proof}

Now we adapt arguments from~\cite{hryn,HS} to our present situation. We will denote by $F_0$ the embedded open disk $F_0 = u_0(\D\setminus \partial\D)$, oriented by requiring that the map $u_0|_{\D\setminus \partial\D}$ is orientation preserving when $\D$ is equipped with its usual orientation, and we will choose a non-vanishing section
\begin{equation}\label{non_vanish_section_p_disk}
Z : F_0 \to \xi|_{F_0}.
\end{equation}

\begin{lemma}\label{lemma_no_zeros_bishop}
The sections $\pi\circ du_k$ never vanish on $\D$ when $k$ is large enough.
\end{lemma}

\begin{proof}
We denote by $\partial_r$ and $\partial_\theta$ partial derivatives of maps defined on subdomains of $\D\setminus\{0\}$ with respect to standard polar coordinates $(r,\theta) \simeq re^{i\theta}$.

First note that $\pi\circ du_k$ never vanishes on $\partial\D$ when $k$ is large enough. In fact, by the hypotheses of Proposition~\ref{prop_existence_fast} there exists $\epsilon>0$ small enough such that the embedded strip $S_\epsilon := u_0(\{1-\epsilon\leq |z|<1\})$ is transverse to $R$. By Proposition~\ref{prop_limit_bishop_disks} the loops $t\mapsto u_k(e^{i2\pi t})$ $C^\infty$-converge to the orbit $P_0$, so $u_k(\partial\D)\subset S_\epsilon$ when $k$ is large enough. If there are arbitrarily large values of $k$ such that $\exists z$ satisfying $|z|=1$ and $\pi\circ du_k|_z = 0$ then $\partial_\theta u_k(z)$ is a multiple of the Reeb vector field $R|_{u_k(z)}$ which is also tangent to $S_\epsilon$. But the properties of $S_\epsilon$ tell us that this is possible only if $\partial_\theta u_k(z)=0$. However, the strong maximum principle will tell us that $\partial_ra_k(z) >0$, and then the Cauchy-Riemann equations will tell us that $\lambda \cdot \partial_\theta u_k(z) > 0$, a contradiction.

Now we choose a smooth vector field $V$ on $F_0 := u_0(\D\setminus \partial\D)$ parametrizing the characteristic distribution of $F_0$, which has a nondegenerate source at the singularity $e$ and points away from $e$. The next step is to show that $V\circ u_k(e^{i2\pi t})$ and $\pi \cdot \partial_\theta u_k(e^{i2\pi t})$ are linearly independent vectors in $\xi|_{u_k(e^{i2\pi t})}$, for every $t\in\R/\Z$. Arguing indirectly, let $t\in\R/\Z$ and $c_1,c_2\in\R$ be such that $$ c_1 V\circ u_k(e^{i2\pi t}) + c_2 \pi \cdot \partial_\theta u_k(e^{i2\pi t}) = 0. $$ If $c_2=0$ then $c_1 V\circ u_k(e^{i2\pi t})=0$ which implies $c_1=0$. Assuming that $c_2\neq 0$ then $$ c_1 V\circ u_k(e^{i2\pi t}) + c_2 \pi \cdot \partial_\theta u_k(e^{i2\pi t}) = \pi \cdot (c_1 V\circ u_k(e^{i2\pi t}) + c_2 \partial_\theta u_k(e^{i2\pi t})) = 0 $$ implying that, when $k$ is large enough, the vector $c_1 V\circ u_k(e^{i2\pi t}) + c_2 \partial_\theta u_k(e^{i2\pi t})$ is tangent to $S_\epsilon$ and parallel to the Reeb vector field. Consequently we obtain $$ \partial_\theta u_k(e^{i2\pi t}) = -\frac{c_1}{c_2} V\circ u_k(e^{i2\pi t}) $$ which proves that $\lambda \cdot \partial_\theta u_k(e^{i2\pi t}) =0$. As explained above this is a contradiction to the maximum principle.
 
The vector field $V$ is simultaneously a section of two vector bundles over $F_0$, namely, $TF_0$ and $\xi|_{F_0}$. Since $e$ is a nondegenerate zero of $V$ seen as section of $TF_0$, it is also a nondegenerate zero of $V$ seen as a section of $\xi|_{F_0}$. It contributes with $+1$ to the algebraic count of zeros of $V$ seen as a section of $\xi|_{F_0}$ since $e$ is a positive singular point of the characteristic distribution. These remarks and standard degree theory shows that
\begin{equation*}
\wind(t\mapsto V\circ u_k(e^{i2\pi t}), t\mapsto Z\circ u_k(e^{i2\pi t})) = +1
\end{equation*}
where the   {winding number} is computed orienting $\xi$ by $d\lambda$. Note here that the loop $t\mapsto u_k(e^{i2\pi t})$ bounds a closed disk in $F_0$ containing $e$ in its interior.

Let us equip $\D$ with complex coordinates $x+iy$. Clearly we have
\begin{equation*}
\wind(t\mapsto \pi \cdot \partial_\theta u_k(e^{i2\pi t}), t\mapsto \pi\cdot \partial_x u_k(e^{i2\pi t}) = +1
\end{equation*}
since $\pi\circ du_k$ does not vanish on $\partial\D$ and the vector field $\partial_\theta$ winds once relatively to $\partial_x$ along $\partial\D$.

Combining all these facts we arrive at
\[
\begin{aligned}
\wind( & \pi \cdot \partial_xu_k(e^{i2\pi t}), Z\circ u_k(e^{i2\pi t})) \\
& = \wind(\pi \cdot \partial_xu_k(e^{i2\pi t}),\pi \cdot \partial_\theta u_k(e^{i2\pi t})) \\
& + \wind(\pi \cdot \partial_\theta u_k(e^{i2\pi t}),V\circ u_k(e^{i2\pi t})) \\
& + \wind(V\circ u_k(e^{i2\pi t}),Z\circ u_k(e^{i2\pi t})) \\
& = -1 + 0 + 1 = 0
\end{aligned}
\]
when $k$ is large enough. Note that $\wind(\pi \cdot \partial_\theta u_k(e^{i2\pi t}),V\circ u_k(e^{i2\pi t}))=0$ because, as proved before, $\pi \cdot \partial_\theta u_k(e^{i2\pi t})$ and $V\circ u_k(e^{i2\pi t})$ are pointwise linearly independent when $k$ is large enough.

Now we conclude by noting that condition (C) described in \S~\ref{def_bishop_family}, which is part of the definition of the Bishop family, implies that the section $Z|_{u_k(\partial\D)}$ viewed as a section of $(u_k|_{\partial\D})^*\xi \to \partial\D$ can be extended to a non-vanishing section of $u_k^*\xi \to \D$. Thus the above calculation of winding numbers tells us that the algebraic count of zeros of $\pi\cdot \partial_x u_k$ vanishes, when $k$ is large enough. A Cauchy-Riemann type equation satisfied by $\pi \circ du_k$ implies all zeros count positively. Consequently $\pi\cdot \partial_xu_k$ has no zeros at all and, in particular, $\pi \circ du_k$ does not vanish when $k\gg1$.
\end{proof}

Consider $P_z=(x_z,T_z) \in \P(\lambda)$ the asymptotic limit of $\vtil$ at each negative puncture $z\in\Gamma_0$. All these orbits are contractible in $M$ because the resulting pieces of the bubbling-off tree given by applying Proposition~\ref{propbubtree} to the germinating sequence $\vtil_k$~\eqref{germ_seq_below_bishop} yield disks for each $P_z$. Thus for every $z\in\Gamma_0$ we can choose a smooth capping disk $D_z:\D\to M$ for $P_z$ satisfying $D_z(e^{i2\pi t}) = x_z(T_zt)$ $\forall t\in\R/\Z$. Choose also a smooth capping disk $D_\infty$ for $P_0$. As observed in \S~\ref{sec_index_estimates}, non-vanishing sections $\sigma_z$ of $D_z^*\xi$ and $\sigma_\infty$ of $D_\infty^*\xi$ extend to a non-vanishing section $\sigma_{\rm disks}$ of $v^*\xi$. This follows from the fact that $c_1(\xi)$ vanishes on $\pi_2(M)$.

\begin{lemma}\label{lem_wind_infty_below_bishop}
$\wind_\infty(\vtil,\infty,\sigma_{\rm disks}) = 1$.
\end{lemma}

\begin{proof}
In view of Lemma~\ref{lemma_1_puncture_2} and Theorem~\ref{thm_precise_asymptotics} that $\exists R\gg1$ such that $\pi\circ dv$ does not vanish on $\{z\in\C : |z|\geq R\}$ and $\wind_\infty(\vtil,\infty,\sigma_{\rm disks})$ is well-defined. In particular, from the definition of $\wind_\infty$ given in \S~\ref{sec_alg_invs} we have  
\[
\wind_\infty(\vtil,\infty,\sigma_{\rm disks}) = \wind(t\mapsto \pi\cdot \partial_r v(Re^{i2\pi t}), t\mapsto \sigma_{\rm disks}(Re^{i2\pi t})).
\]
Here $r$ denotes the radial coordinate associated to polar coordinates on $\C$ centered at the origin. As always winding numbers are computed orienting $\xi$ by $d\lambda$.

Consider the loop $\gamma:\R/\Z\to M$, $\gamma(t) = v(Re^{i2\pi t})$. Then the sequence of loops $\gamma_k(t) = v_k(Re^{i2\pi t})$ converges in $C^\infty$ to $\gamma$ because $\Gamma_0\subset \D$. By the same reason the vector field $t\mapsto \pi \cdot \partial_r v(Re^{i2\pi t})$ along $\gamma$ and the vector field $t\mapsto \pi\cdot \partial_rv_k(Re^{i2\pi t})$ along $\gamma_k$ are arbitrarily $C^\infty$-close when $k$ is large enough. Note that these vector fields define non-vanishing sections of $\gamma^*\xi$ and of $\gamma_k^*\xi$, respectively. Now consider a $C^\infty$-small smooth homotopy $h_{k,\tau}(t)$, $\tau\in[0,1]$ satisfying $h_{k,0}(t)=\gamma_k(t)$ and $h_{k,1}(t) = \gamma(t)$. Then, when $k$ is large enough, the vector fields $\pi \cdot \partial_r v(Re^{i2\pi t})$ and $\pi\cdot \partial_rv_k(Re^{i2\pi t})$ extend smoothly to a non-vanishing section of $h_k^*\xi$.

For every $k$ large enough, consider a non-vanishing section $Z_k$ of $u_k^*\xi$. Note that $\gamma_k(t) = u_k(z_k+R\delta_ke^{i2\pi t})$. The sections $Z_k|_{\overline B_k}$ of $(u_k|_{\overline B_k})^*\xi$, with $B_k = B_{R\delta_k}(z_k)$, extend to a non-vanishing section $\mathcal Z_k$ of $\xi$ over a piecewise smooth capping disk $\mathcal D_k$ for $\gamma$ defined by attaching $h_k$ to $u_k|_{\overline B_k}$. We can arrange $\mathcal Z_k$ to be smooth over $\gamma = \partial\mathcal D_k$. Note that $t\mapsto \sigma_{\rm disks}(Re^{i2\pi t})$ is homotopic to $\mathcal Z_k|_\gamma$ through non-vanishing sections of $\gamma^*\xi$ since both come from capping disks for $\gamma$ and $c_1(\xi)$ vanishes on $\pi_2(M)$. Now we can compute
\[
\begin{aligned}
& \wind(t\mapsto \pi\cdot \partial_r v(Re^{i2\pi t}), t\mapsto \sigma_{\rm disks}(Re^{i2\pi t})) \\
& = \wind(t\mapsto \pi\cdot \partial_r v(Re^{i2\pi t}), t\mapsto \mathcal Z_k|_{\gamma(t)}) \\
& = \wind(t\mapsto \pi\cdot \partial_r v_k(Re^{i2\pi t}), t\mapsto \mathcal Z_k|_{\gamma_k(t)}) \\
& = \wind(t\mapsto \pi\cdot \partial_\rho u_k(z_k+\delta_k Re^{i2\pi t}), t\mapsto  Z_k(z_k+\delta_k Re^{i2\pi t})) \\
& = \wind(t\mapsto \pi\cdot \partial_\rho u_k(z_k+\delta_k Re^{i2\pi t}), t\mapsto \pi\cdot \partial_x u_k(z_k+\delta_k Re^{i2\pi t})) \\
& + \wind(t\mapsto \pi\cdot \partial_x u_k(z_k+\delta_k Re^{i2\pi t}), t\mapsto Z_k(z_k+\delta_k Re^{i2\pi t})) = 1 + 0.
\end{aligned}
\]
Above $k$ is to be taken large enough, $\rho$ denotes the radial coordinate associated to polar coordinates centered at $z_k$, and $x+iy$ is a global complex coordinate on the domain $\D$ of $u_k$. Lemma~\ref{lemma_no_zeros_bishop} was strongly used in the last equality.
\end{proof}

\begin{lemma}\label{lemma_plane_below_bishop}
$\Gamma_0=\emptyset$.
\end{lemma}

\begin{proof}
By lemmas~\ref{lemma_1_puncture_2} and~\ref{lem_wind_infty_below_bishop}, if $\Gamma_0\neq\emptyset$ then the germinating sequence $\vtil_k$~\eqref{germ_seq_below_bishop} and its limit $\vtil$~\eqref{root_below_bishop} satisfy the hypotheses of Proposition~\ref{main_prop_compactness}. As a consequence we find a finite-energy $\jtil$-holomorphic plane $\util_*=(a_*,u_*)$ asymptotic to a closed Reeb orbit $P_*=(x_*,T_*)$ satisfying $\mu_{CZ}(P_*,\beta_{\rm disk})=2$ and $u_*(\C) \cap x_0(\R) = \emptyset$. If $x_*=x_0$ then $T_*<T_0$ because $\mu_{CZ}(P_0,\beta_{\rm disk})\geq 3$, and we find $1\leq p'<p$ such that $T_* = p'T_{\rm min}<pT_{\rm min}= T_0$. This means that $u_*$ gives a disk for the $p'$-th iterate of $x_0(\R)$, contradicting Lemma~\ref{lemma_no_p'_disk}. We proved that $P_*$ and $P_0$ are geometrically distinct. Consequently, $P_*$ is contractible in $M\setminus x_0(\R)$. In view of~\eqref{crucial_energy_estimate} and the definition of $\vtil_k$ in~\eqref{germ_seq_below_bishop} we have $E(\vtil_k)\leq C_0$, $\forall k$, where $C_0$ is the positive constant~\eqref{constant_C_0}. Proposition~\ref{main_prop_compactness} implies that $T_*=\int_{P_*}\lambda \leq C_0$. This contradicts the hypotheses of Proposition~\ref{prop_existence_fast}. Thus the assumption $\Gamma_0\neq\emptyset$ is false.
\end{proof}

Combining Lemma~\ref{lem_wind_infty_below_bishop} with Lemma~\ref{lemma_plane_below_bishop} we conclude that $\vtil$ is a finite-energy plane asymptotic to $P_0$ and satisfying $\wind_\pi(\vtil)=0$. It follows that $\vtil$ is an immersion. To show that it is an embedding consider the set $$ E = \{ (z_1,z_2) \in \C\times \C \mid z_1\neq z_2 \ \text{and} \ \vtil(z_1)=\vtil(z_2) \}. $$ If $E$ has a limit point away from the diagonal then, using the similarity principle, one finds a $\jtil$-holomorphic map $\wtil:\C\to \R\times M$ and a polynomial $Q:\C\to\C$ of degree at least $2$ such that $\vtil = \wtil \circ Q$. But the zeros of $Q'$ will force zeros of $d\vtil$, contradicting the fact that $\vtil$ is an immersion. We showed that $E$ is discrete in the complement of the diagonal in $\C\times \C$. But if $E$ is not empty then stability and positivity of intersections of pseudo-holomorphic immersions will force self-intersections of the maps $\vtil_k$, and consequently also of $\util_k$, which is impossible since the Bishop disks are embeddings. Thus $E=\emptyset$ and $\vtil$ is an embedding.

To complete the proof of Proposition~\ref{prop_existence_fast} it remains only to show that $P_0$ is $p$-unknotted in the homotopy class of $v$. This follows from the fact that extending $v$ continuously to the compactification $\C\sqcup S^1$ of $\C$ given by adding a circle at $\infty$, we obtain a capping disk for $P_0$ which is homotopic to the $p$-disk $u_0$ modulo boundary.

\section{Constructing global surfaces of section}\label{sec_const_global_sections}

Here we prove the following statement.

\begin{proposition}\label{proposition_global_sections2}
Let $\lambda$ be a defining contact form for a tight closed connected contact $3$-manifold $(M,\xi)$ satisfying $c_1(\xi)|_{\pi_2(M)}=0$, and let $K\subset M$ be a $p$-unknotted prime closed $\lambda$-Reeb orbit satisfying $\sl(K)=  {\frac{-1}{p}}$ and $\mu_{CZ}(K^p)\geq 3$. Consider $\P^*\subset\P(\lambda)$ the set of contractible closed Reeb orbits $P'\subset M\setminus K$ satisfying $\rho(P')=1$, and let $u_0$ be a $p$-disk which is special robust for $(\lambda,K)$. By definition, $u_0$ has precisely one singular point $e$ for its characteristic foliation. Consider also a sequence of smooth functions $f_n:M\to (0,+\infty)$ satisfying $f_n|_K\equiv1$, $df_n|_K \equiv0$, $f_n\to 1$ in $C^\infty$, $f_n|_V\equiv1$ on an open neighborhood $V$ of $e$, and such that $\lambda_n:= f_n\lambda$ is nondegenerate $\forall n$. If every orbit $P'\in \P^*$ satisfies one of 
\begin{itemize}
\item[a)] $P'$ is not contractible in $M\setminus K$, or
\item[b)] $\int_{P'}\lambda > C(\lambda,K,u_0) := 1+\int_\D |u_0^*d\lambda|$,
\end{itemize}
then one finds $n_0$ such that for every $n\geq n_0$ there exists   {a rational} open book decomposition $(K,\pi_n)$ with disk-like pages of order $p$ adapted to~$\lambda_n$.
\end{proposition}

\begin{remark}\label{seqlambdak}
Sequences $f_n$ as in the above statement always exist if $V$ is taken small enough. This is kind of standard, see~\cite{convex} for instance. When $\lambda$ is nondegenerate we may take $f_n\equiv1$, $\forall n$.
\end{remark}

Throughout this section we will be occupied with the proof of Proposition~\ref{proposition_global_sections2}, and we will fix a closed connected co-oriented tight contact $3$-manifold $(M,\xi)$ such that $c_1(\xi)$ vanishes on $\pi_2(M)$, and a defining contact form $\lambda$ for $\xi$ inducing the {\it a priori} given co-orientation.

As explained in \S~\ref{sssec_CZ_orbits}, associated to every contractible closed Reeb orbit $P = (x,T) \in \P(\lambda)$ there is a distinguished homotopy class $\beta_{\rm disk}$ of oriented trivializations of  $(x_T)^*\xi$ associated to a capping disk for $P$. Namely, if $g:\D\to M$ is a smooth map satisfying $g(e^{i2\pi t}) = x_T(t)$ and $Z$ is a non-vanishing section of $g^*\xi$ then $Z$ can be completed to a oriented trivialization representing $\beta_{\rm disk}$. Here $x_T$ denotes the map~\eqref{map_x_T} and $\xi$ is always oriented by $d\lambda$. The transverse rotation number $\rho(P) = \rho(P,\beta_{\rm disk})$ is defined as in \S~\ref{sssec_transv_rot_number}.

Let $K \subset M$ be a knot as in Proposition~\ref{proposition_global_sections2}, and denote by
\begin{equation}\label{special_set_P_*}
\P^* \subset \P(\lambda)
\end{equation}
the set of the closed orbits contained in $M\setminus K$ which are contractible in $M$ and satisfy $\rho(P^*,\beta_{\rm disk}) = 1$. We will denote by $$ x_0 : \R\to M $$ a Reeb trajectory satisfying $x_0(\R) = K$ and by $T_{\rm min}>0$ its minimal period. Of course, the map $x_0$ is only defined up to translation by a real constant. The period of its $p$-th iterate is denoted by $T_0 = pT_{\rm min}$.

\subsection{An auxiliary lemma: the nondegenerate case}\label{foliations_nondegenerate}

\begin{lemma}\label{lemma_aux}
Assume that $\lambda$ is nondegenerate, and let $C_0>0$ be the constant provided by Proposition~\ref{prop_existence_fast} applied to $\lambda$, $P_0=(x_0,T_0)$ and a fixed $p$-disk which is special for $(\lambda,K)$. Consider $\P^*\subset\P(\lambda)$ the set of contractible closed Reeb orbits $P'\subset M\setminus K$ satisfying $\rho(P')=1$. If every orbit $P'\in \P^*$ satisfies 
\begin{itemize}
\item $P'$ is not contractible in $M\setminus K$, or
\item $\int_{P'} \lambda > C_0$,
\end{itemize}
then there exists   {a rational} open book decomposition $(K,\pi)$ of disk-like pages of order $p$ adapted to $\lambda$.
\end{lemma}

Let us prove Lemma~\ref{lemma_aux}. Since $\lambda$ is nondegenerate then $\rho(P,\beta_{\rm disk})=1$ if, and only if, $\mu_{CZ}(P,\beta_{\rm disk}) = 2$, as one easily shows. Let $C_0>0$ be the constant provided by Proposition~\ref{prop_existence_fast}, in particular, $C_0\geq T_0$.


Suppose that every $P\in \P^*$ has period larger than $C_0$ or is non-contractible in $M\setminus K$. Thus, by Proposition~\ref{prop_existence_fast} we find $J\in\J_+(\xi)$ and an embedded fast finite-energy plane
\begin{equation}\label{initial_plane_v_0}
\vtil_0 = (b_0,v_0) :\C \to \R\times M
\end{equation}
asymptotic to $P_0$. Moreover, $P_0$ is $p$-unknotted in the homotopy class of $v_0$, which implies that $\vtil_0(\C) \subset \R\times (M\setminus K)$ and $v_0:\C\to M\setminus K$ is a proper embedding; see Theorem~\ref{thm_fast_embedded}.

By Theorem~\ref{thm_fred_theory} $\vtil_0$ belongs to a $2$-dimensional family that fills up an open set in the symplectization. Identity $\wind_\pi(\vtil_0)=0$ implies that the Reeb vector field is transverse to the embedding $v_0$. Combining these two facts and Poincar\'e's recurrence, there is no loss of generality to assume that $v_0(0)$ is a recurrent point of the Reeb flow.   {Since $v_0(0)$ is a recurrent point and $v_0(\C)$ is transverse to the Reeb vector field} we find a minimal positive time $T_*$ such that $\phi_{T_*}(v_0(0)) \in v_0(\C)$. Define
\begin{equation}\label{path_gamma}
\begin{array}{ccc} \gamma:[0,T_*] \to M\setminus K & \text{by} & \gamma(\tau) = \phi_\tau(v_0(0)) \end{array}
\end{equation}
and a compact set
\begin{equation*}
H := \{ (0,\gamma(\tau)) \in \R\times M \mid \tau \in[0,T_*] \} \subset \{0\}\times M\setminus K
\end{equation*}
Theorem~\ref{thm_comp_fast} gives $C^\infty_{\rm loc}$-compactness of the special set $\Lambda(H,P_0,\lambda,J)$ of embedded fast planes~\eqref{set_of_fast_planes}. After a reparametrization we obtain $\vtil_0 \in \Lambda(H,P_0,\lambda,J)$.


\begin{lemma}\label{lemma_dichotomy_fast_planes}
Let $\util=(a,u)$ and $\vtil=(b,v)$ be planes in $\Lambda(H,P_0,\lambda,J)$. Then either $u(\C) \cap v(\C) = \emptyset$ or $u(\C) = v(\C)$. In the latter we find $c\in\R$, $A,B\in\C$, $A\neq0$, such that $a(Az+B)+c=b(z)$ and $u(Az+B)=v(z)$ for every $z\in\C$.
\end{lemma}

\begin{proof}
The dichotomy $u(\C) \cap v(\C) = \emptyset$ or $u(\C) = v(\C)$ follows from the proof of Theorem~4.11 from~\cite{props2}. In fact, $(\util(\C) \cup \vtil(\C)) \cap (\R\times x_0(\R)) = \emptyset$ by Theorem~\ref{thm_fast_embedded}. Since $c_1(\xi)$ is assumed to vanish on $\pi_2(M)$, all assumptions of~\cite[Theorem~4.11]{props2} are satisfied, except for the assumption $\mu(\util)=\mu(\vtil)\leq 3$. Inspecting its proof this is only used to guarantee that $\wind_\pi(\util)=\wind_\pi(\vtil)=0$, which is true by assumption in our situation. Thus the same argument goes through. The constants $A,B$ are obtained by the similarity principle, since $u(\C)=v(\C)$ implies that $\vtil$ is a holomorphic reparametrization of $\util$, modulo translation in the $\R$-coordinate.
\end{proof}

  {The above proof is a piece of the intersection theory for pseudo-holomorphic curves in symplectizations initiated in~\cite{props2} and further developed in~\cite{siefring}.}

Consider the set $S\subset [0,T_*]$ consisting of numbers $T$ for which there exists a smooth map $(\tau,z) \in [0,T]\times \C \mapsto \vtil_\tau(z) \in \R\times M$ such that $\vtil_\tau = (b_\tau,v_\tau) \in \Lambda(H,P_0,\lambda,J)$, $\vtil_\tau(0) = (0,\gamma(\tau))$ $\forall \tau$, and $\vtil_0$ is the plane~\eqref{initial_plane_v_0}.

\begin{lemma}\label{lemma_maximal_family}
$S = [0,T_*]$.
\end{lemma}

\begin{proof}
Let $\{T_n\} \subset S$ be any sequence. We will show that a limit point $T_\infty$ of $\{T_n\}$ is an interior point of $S$, thus proving that $S$ is open and closed. Up to a subsequence, assume $T_n \to T_\infty$. Abbreviating $\Lambda = \Lambda(H,P_0,\lambda,J)$, we find $\vtil_n \in \Lambda$ such that $\vtil_n(0)=(0,\gamma(T_n))$. By $C^\infty_{\rm loc}$-compactness of $\Lambda$, up to a subsequence, there exists $\vtil_\infty\in\Lambda$ such that $\vtil_n\to\vtil_\infty$. By Theorem~\ref{thm_fred_theory} we find $\epsilon>0$ and a smooth map $$ (\tau,z) \in (T_\infty-\epsilon,T_\infty+\epsilon)\times \C \mapsto \wtil_\tau(z) \in \R\times M $$ satisfying $\wtil_\tau \in \Lambda$, $\wtil_\tau(0)=(0,\gamma(\tau))$, $\forall \tau \in (T_\infty-\epsilon,T_\infty+\epsilon)$ and $\wtil_{T_\infty} = \vtil_\infty$. If $n$ is large then $|T_n-T_\infty|<\epsilon$ and we find a smooth map $(\tau,z)  \mapsto \vtil_\tau(z) \in\R\times M$ defined on $[0,T_n]\times \C$ such that $\vtil_0$ is the plane~\eqref{initial_plane_v_0}, $\vtil_\tau(0) = (0,\gamma(\tau))$ and $\vtil_\tau\in\Lambda$, $\forall \tau\in[0,T_n]$; this follows from the definition of $S$. Using Lemma~\ref{lemma_dichotomy_fast_planes} and smoothness of the embedding given by Theorem~\ref{thm_fred_theory} we can suitably reparametrize the map $(\tau,z)\mapsto \wtil_\tau(z)$ in the $z$-variable in such way that it agrees with the map $(\tau,z)\mapsto \vtil_\tau(z)$ on $(T_\infty-\epsilon,T_n)\times \C$. This allows us to extend the map $\vtil_\tau(z)$ to $[0,T_\infty+\epsilon)\times \C$ with the desired properties, proving that $T_\infty$ is an interior point of $S$.
\end{proof}

By Lemma~\ref{lemma_maximal_family} we get a smooth map
\begin{equation}\label{maximal_family}
(\tau,z) \in [0,T_*]\times \C \mapsto \vtil_\tau(z) \in \R\times M
\end{equation}
such that $\vtil_0$ is the plane~\eqref{initial_plane_v_0}, $\vtil_\tau(0)=(0,\gamma(\tau))$, and $\vtil_\tau \in \Lambda(H,P_0,\lambda,J), \ \forall \tau$.

\begin{lemma}
Denote by $\vtil_\tau = (b_\tau,v_\tau)$ the components of the map~\eqref{maximal_family}. Then $v_\tau(\C) \subset M\setminus K \ \forall \tau$ and
\begin{equation}\label{pre_open_book}
(\tau,z) \in (0,T_*)\times \C \mapsto v_\tau(z) \in M\setminus (v_0(\C)\cup K)
\end{equation}
is a diffeomorphism.
\end{lemma}

\begin{proof}
The first step is to show that~\eqref{pre_open_book} is injective. By Lemma~\ref{lemma_dichotomy_fast_planes} we get $v_0(\C) = v_{T_*}(\C)$. For a fixed $\tau_0 \in (0,T_*)$ consider $\mathcal C_{\tau_0} = \{ \tau'\in (0,T_*) \mid \phi_{\tau'}(v_0(0)) \in v_{\tau_0}(\C) \}$. Clearly $1\leq \#\mathcal C_\tau <\infty \ \forall \tau$. We claim that $\#\mathcal C_{\tau_0}=1$. To prove this, assume by contradiction that $\#\mathcal C_{\tau_0}\geq 2$. Consider a closed loop $\bar\gamma$ in $M\setminus K$ obtained by following $\tau\mapsto \gamma(\tau) = \phi_\tau(v_0(0))$ for $\tau\in[0,T_*]$ and then following a path from $\phi_{T_*}(v_0(0))$ to $v_0(0)$ inside $v_0(\C)$. Clearly, by construction, the algebraic intersection number of the map $v_0:\C\to M$ with $\bar\gamma$ is $1$. Since $\wind_{\pi}(\vtil_{\tau_0})=0$, every intersection point of the loop $\bar\gamma$ with the embedded disk $v_{\tau_0}(\C)$ is transverse and positive. Consequently the algebraic count of intersections of $\bar\gamma$ with $v_{\tau_0}$ is equal to $\#\mathcal C_{\tau_0}\geq 2$. Since $v_{\tau_0}$ is homotopic to  $v_0$ through proper embeddings into $M\setminus K$, we conclude that the algebraic intersection number of $\bar\gamma$ with $v_0$ is $\geq 2$, a contradiction. Having proved that $\#\mathcal C_\tau=1$ for every $\tau\in(0,T_*)$, we proceed to show that~\eqref{pre_open_book} is injective. Assume that $(\tau_0,z_0), (\tau_1,z_1) \in (0,T_*)\times \C$ satisfy $v_{\tau_0}(z_0)=v_{\tau_1}(z_1)$. If $\tau_0=\tau_1$ then by Theorem~\ref{thm_fast_embedded} we get $z_0=z_1$. If $\tau_0\neq\tau_1$ then $\#\mathcal C_{\tau_0}\geq 2$ since by Lemma~\ref{lemma_dichotomy_fast_planes} we have $v_{\tau_0}(\C)=v_{\tau_1}(\C)$, an absurd. Thus the map~\eqref{pre_open_book} in injective.

If $\tau\in (0,T_*)$ then $v_\tau(\C) \cap v_0(\C) = \emptyset$ since, otherwise, with the help of Theorem~\ref{thm_fast_embedded} we would get a contradiction with the definition of the number $T_*$. Thus the map $(\tau,z)\in (0,T_*)\times \C\mapsto v_\tau(z)$ indeed takes values in $M\setminus (v_0(\C)\cup K)$. 

Next we show that the image of the map~\eqref{pre_open_book} is closed in $M\setminus (v_0(\C)\cup K)$. Fix a point $q \in M\setminus (u_0(\C)\cup K)$ and assume that $(\tau_n,z_n) \in (0,T_*)\times \C$ satisfies $v_{\tau_n}(z_n)\to q$. Let $\mathcal N$ be a small open tubular neighborhood of $K$ such that $q\not\in \mathcal N$. Using Lemma~\ref{lemlong} and the normalization conditions~\eqref{normalization_fast_planes} we find $R_0>1$ such that $|z|\geq R_0 \Rightarrow v_\tau(z) \in \mathcal N$, $\forall \tau\in[0,T_*]$. Hence $\limsup_n |z_n|<\infty$ and, up to a subsequence, we may assume that $z_n \to z_*$ for some $z_*\in\C$. Up to a further subsequence we may also assume that $\tau_n\to \tau_* \in [0,T_*]$. Hence $q = v_{\tau_*}(z_*)$. We must have $\tau_* \not\in \{0,T_*\}$ since $q\not\in v_0(\C)$, concluding the argument.

Finally we claim that~\eqref{pre_open_book} is an immersion. If not we find $\tau_0 \in (0,T_*)$, $z_0\in\C$ and $\zeta_0 \in T_{z_0}\C$ such that $\dot v_{\tau_0}(z_0) = dv_{\tau_0}(z_0) \zeta_0$, where the dot denotes the derivative with respect to $\tau$. We used that $z\mapsto v_{\tau_0}(z)$ is an immersion. Consider a smooth curve $\tau \mapsto c_\tau \in \R$, defined for $\tau$ near $\tau_0$, satisfying $c_{\tau_0} = 0$ and $\dot c_{\tau_0} = db_{\tau_0}(z_0) \zeta_0 - \dot b_{\tau_0}(z_0)$. Define $\wtil_\tau = c_\tau * \vtil_{\tau}$, where $*$ denotes the $\R$-action. Note that $\wtil_{\tau_0}=\vtil_{\tau_0}$. We compute
\begin{equation}\label{singular_derivative}
\dot{\wtil}_{\tau_0}(z_0) = (db_{\tau_0}(z_0) \zeta_0, dv_{\tau_0}(z_0)  \zeta_0)  = d\wtil_{\tau_0}(z_0)  \zeta_0.
\end{equation}
Applying Theorem~\ref{thm_fred_theory} to $\wtil_{\tau_0}$ we find a smooth embedding $F:B\times \C\to \R\times M$, where $B\subset \C$ is the open unit ball, satisfying $F(0,z)=\wtil_{\tau_0}(z)$ and all the other conclusions of Theorem~\ref{thm_fred_theory}. Then there exists a smooth curve $\tau\mapsto \tilde z(\tau) \in B$, defined for $|\tau-\tau_0|$ small and satisfying $\tilde z(\tau_0)=0$, and $\wtil_\tau(z) = F(\tilde z(\tau),A_\tau z+B_\tau)$ for suitable constants $A_\tau,B_\tau$. Since $\dot v_{\tau}(0)$ is a non-trivial multiple of the Reeb vector we conclude that $\dot{\tilde z}(\tau_0)\neq 0$, but this is a contradiction to~\eqref{singular_derivative} since $F$ is an embedding.

Thus the image of the injective immersion~\eqref{pre_open_book} is closed and open in the connected set $M\setminus (v_0(\C)\cup K)$.
\end{proof}

It follows from the lemma above that $\{v_\tau(\C) \mid \tau\in (0,T_*) \}$ is a smooth foliation of $M\setminus (v_0(\R)\cup K)$. Using the embedding $F$ obtained by applying Theorem~\ref{thm_fred_theory} to $\vtil_0$ we conclude that $$ \mathcal L = \{v_\tau(\C) \mid \tau\in [0,T_*) \} $$ is a smooth foliation of $M\setminus K$. Let $q\in M\setminus K$ be any point, denote by $\omega(q)$ its omega limit set with respect to the Reeb flow. Clearly, by compactness and transversality of the Reeb vector field with the leafs of $\mathcal L$, if $\omega(q)\cap K = \emptyset$ then the trajectory through $q$ will hit every leaf infinitely many times in the future. We claim that the same is also true even if $\omega(q)\cap K \neq\emptyset$. This is proved as in~\cite[Section~5]{convex}. The idea is that if $\omega(q)\cap K \neq\emptyset$ then $y(t) := \phi_t(q)$ gets arbitrarily close to $K$ for arbitrarily large values of $t$. Consequently, since $K$ is a periodic Reeb orbit, for every tubular neighborhood $\mathcal N$ of $K$ and every integer $k\geq 1$ we find $t_n>k$ such that $y([t_n,t_n+k]) \subset \mathcal N$, i.e., $y(t)$ spends arbitrarily long intervals of time in any arbitrarily small neighborhood of $K$. Hence, in these long intervals of time the corresponding piece of the trajectory $y$ can be well approximated using the linearized flow along $\xi|_K$. Since $\mu_{CZ}(P_0,\beta_{\rm disk})\geq 3 \Leftrightarrow \rho(P_0,\beta_{\rm disk}) >1$ and $\sl(P_0)=  {\frac{-1}{p}}$, we conclude from Lemma~\ref{lemma_self_link_props} and Corollary~\ref{cor_rel_windings} that $\rho(P_0,\beta_{v_\tau})>0$ where $\beta_{v_\tau}$ is the homotopy class of $d\lambda$-positive trivializations of $({x_0}_{T_0})^*\xi$ induced by the outward normal derivative of the disk $\bar v_\tau$ along its boundary. Here $\bar v_\tau$ is the smooth capping disk for $P_0$ obtained from $v_\tau$ by attaching a circle at $\infty$ and $\tau$ is arbitrary. It follows from $\rho(P_0,\beta_{v_\tau})>0$ that over a long interval of time the linearized flow along $K$ rotates more than the normal of the disk $\bar v_\tau$, forcing the trajectory $y$ to hit $v_\tau(\C)$ infinitely often in the future, as desired. For more details in the case $p=1$ see~\cite[Lemma~6.9]{hryn} .

All the corresponding assertions for times close to $-\infty$ are proved analogously replacing $\omega(q)$ by the $\alpha$-limit set $\alpha(q)$. It follows that all leaves of $\mathcal L$ are global surfaces of section for the Reeb dynamics, concluding the proof of Lemma~\ref{lemma_aux}.

\subsection{Proof of Proposition~\ref{proposition_global_sections2}}\label{deg_case_subsection}


Recall that $(M,\xi=\ker \lambda)$ is a co-oriented tight contact manifold with first Chern class $c_1(\xi)$ vanishing on $\pi_2(M)$. Every contractible closed orbit $P\in \P(\lambda)$ has a well-defined transverse rotation number $\rho(P)$ computed with respect to a capping disk $\dcal$ for $P$. The assumption $c_1(\xi)|_{\pi_2(M)} \equiv 0$ implies that $\rho(P)$ does not depend on the choice of $\dcal$. We shall make use of the following standard facts about transverse rotation numbers
\begin{eqnarray}  
\label{rot1} \rho(P)=1 \Rightarrow \rho(P^l)=l,\forall l\in \N^*, \\
\label{rot2} \rho(P)>1 \Rightarrow \rho(P^l)>l,\forall l\in \N^*.
\end{eqnarray} 
As before $K\subset M$ denotes an order $p$ rational unknot, which corresponds to a simple closed orbit for the Reeb flow of $\lambda$. It satisfies $\mu_{CZ}(K^p) \geq 3$, $\sl(K) =  {\frac{-1}{p}}$ and $\mon(K) = -q \mod p$, where the integers $q,p$ are relatively prime and $1\leq q \leq p$. Again, the Conley-Zehnder index $\mu_{CZ}(K^p)$ is computed with respect to a capping disk $\dcal$ for $K^p$ and does not depend on the choice of $\dcal$. It is well-known that
\begin{equation}\label{muczrho} 
\mu_{CZ}(K^p) \geq 3 \Leftrightarrow \rho(K^p)>1.\end{equation}

Consider the $p$-disk $u_0$ for $K=x_0(\R)$ which is special robust for $(\lambda,K)$ given in Proposition~\ref{proposition_global_sections2}. Recall the set $\P^*$~\eqref{special_set_P_*} defined in the beginning of section~\ref{sec_const_global_sections}, and assume that every orbit $P'\in\P^*$ is not contractible inside $M\setminus K$, or $\int_{P'}\lambda>C(\lambda,K,u_0)$. Using Proposition~\ref{prop_perturb_special_robust} we can perturb $u_0$ into a new $p$-disk $u_0'$ for $K$ which is special for $(\lambda_n,K)$, for all $n$ large enough. In particular
\begin{equation}
\int_\D |(u_0')^*d\lambda_n| \leq C(\lambda,K,u_0)
\end{equation}
for all $n$ large enough.

\begin{lemma}\label{seqlambdak2}
Denote by $\P^*_n \subset \P(\lambda_n)$ the set of closed orbits of $\lambda_n$ in $M\setminus K$, which are contractible in $M$ and have transverse rotation number equal to $1$. There exists $n_1 \in \N$ such that if $n>n_1$ and $P\in \P^*_n$ then $P$ is non-contractible in $M \setminus K$ or $\int_P \lambda_n > C=C(\lambda,K,u_0)$.  
\end{lemma}

\begin{proof}
Arguing indirectly assume the existence of a subsequence $n_j \to +\infty$ as $j\to \infty$ such that $\P^*_{n_j}$ contains a closed orbit $P_j\subset M\setminus K,$ satisfying both conditions 
\begin{itemize}
\item[(i)] $P_j$ is contractible in $M\setminus K,\forall j.$  
\item[(ii)] $\int_{P_j} \lambda_{n_j} \leq C,\ \forall j$.
\end{itemize} 
It follows from  (ii) and Arzel\`a-Ascoli's theorem that, up to extraction of a subsequence, we can assume that 
\begin{equation}\label{propp3} 
P_j \stackrel{C^\infty}{\longrightarrow} P\in \P(\lambda) \ \mbox{ as } \ j\to \infty.
\end{equation} 
We must have $\rho(P)=1$ since $\rho(P_j)=1$, $\forall j$. If $P\subset M \setminus K$, then (i), (ii) and~\eqref{propp3} imply that $P\in\P^*$, $P$ is contractible in $M\setminus K$ and $\int_P\lambda\leq C$, a contradiction.  So we can assume that $P = K^{p_0}$ for some $p_0\in \N^*$. Since $P_j$ is contractible in $M$ it follows from \eqref{propp3} that $K^{p_0}=P$ is contractible in $M$ and $\rho(K^{p_0})=1$ as well. This last equality and \eqref{rot1} imply that 
\begin{equation}\label{pp00} 
\rho(K^{p_0p}) =p.
\end{equation} 
However, from the assumption $\mu_{CZ}(K^p)\geq 3$, we have from \eqref{muczrho} that $\rho(K^p) >1$, which, in view of \eqref{rot2}, implies that \begin{equation}\label{pp0} \rho(K^ {p p_0}) >p_0.\end{equation}  From \eqref{pp00} and \eqref{pp0} we get 
\begin{equation}
\label{pp01}p_0<p.
\end{equation} 
We concluded that $K^{p_0}$ is contractible, with $p_0<p$, contradicting Lemma~\ref{lemma_no_p'_disk}.
\end{proof}

The proof of Proposition~\ref{proposition_global_sections2} ends with a combined application of Lemma~\ref{seqlambdak2} and Lemma~\ref{lemma_aux} for the nondegenerate contact forms $\lambda_n$, with $n\gg1$.


\appendix

\section{Fredholm theory for fast planes \\ with higher covering number}

\begin{theorem}\label{thm_fred_theory}
Let $\lambda$ be a defining nondegenerate contact form on a contact $3$-manifold $(M,\xi)$, $J \in \J_+(\xi)$ be arbitrary, and $\util=(a,u): \C\to \R\times M$ be an embedded fast finite-energy $\jtil$-holomorphic plane asymptotic to a closed Reeb orbit $P_0=(x_0,T_0)$. If $\mu_{CZ}(\util)\geq 3$ then, denoting by $B\subset \C$ the open unit ball, there exists a smooth embedding
\begin{equation}\label{map_implicit_function_thm}
F : B \times \C \to \R\times M
\end{equation}
satisfying:
\begin{itemize}
\item[(a)] $F(0,z) = \util(z) \ \forall z\in\C$.
\item[(b)] $z\mapsto F(z',z)$ is an embedded fast finite-energy $\jtil$-holomorphic plane asymptotic to $P_0$, $\forall z' \in B$.
\item[(c)] If $\util_k$ is a sequence of embedded fast finite-energy planes asymptotic to $P_0$ satisfying $\mu_{CZ}(\util_k)=\mu_{CZ}(\util)$ and $\util_k\to\util$ in $C^\infty_{\rm loc}$, then there exist sequences $z'_k\to 0$, $A_k\to 1$ and $B_k \to0$ such that $\util_k(z) = F(z'_k,A_kz+B_k) \ \forall z\in\C$.
\end{itemize}   
\end{theorem}

Here the integer $\mu_{CZ}(\util)$ denotes the Conley-Zehnder index of $P_0$ computed with respect to a $d\lambda$-symplectic frame of $\xi$ along $P_0$ which extends to a frame of $u^*\xi$.

\begin{proof}[Sketch of the proof of Theorem~\ref{thm_fred_theory}]
Let us quickly sketch the reason why it is possible to set-up a Fredholm theory for embedded fast planes asymptotic to $P_0$, with exponential weights carefully chosen in such a way that 
\begin{itemize}
\item[i)] the Fredholm index equal to $2$, and
\item[ii)] there is automatic transversality.
\end{itemize}
This was already carefully done in~\cite{hryn} for the case $P_0$ is prime, see also~\cite{wendl} for the general procedure for higher genus curves.

Let us write the components of $\util$ in $\R\times M$ as $\util=(a,u)$. The complex plane $\C$ can be compactified to a closed disk $D = \C \sqcup \infty S^1$ by the addition of a circle $\infty S^1$ at radius equal to infinity, in a standard fashion. In view of the asymptotic behavior described in Theorem~\ref{thm_precise_asymptotics}, the map $u : Re^{i2\pi t} \in \C \mapsto u(Re^{i2\pi t}) \in M$ extends smoothly to a capping disk $\bar u:D\to M$ for $P_0$ by the formula $$ \bar u(\infty e^{i2\pi t}) = x_0(T_0t), $$ up to a rotation in the domain. Let us denote by $\xi^{P_0}$ the bundle over $S^1$ with fiber $\xi|_{x_0(T_0t)}$ over $e^{i2\pi t}$. Again by Theorem~\ref{thm_precise_asymptotics}, the bundle $u^*\xi$ extends smoothly to a bundle $\bar u^*\xi \to D$ by setting $\bar u^*\xi|_{\infty S^1} =\xi^{P_0}$.

Consider the projection $\pi_M:\R\times M \to  M$ onto the second factor. In view of Theorem~\ref{thm_precise_asymptotics} the $\jtil$-invariant bundle $\util^*\pi_M^*\xi$ is complementary to $T\util$ over points $z\in \C$ satisfying $|z|\gg1$. Then there exists a $\jtil$-invariant subbundle $N_\util \subset \util^*T(\R\times M)$ complementary to $T\util$ and coinciding with $\util^*\pi_M^*\xi = (\pi_M\circ \util)^*\xi = u^*\xi$ over points $z\in \C$ satisfying $|z|\gg1$. Again by Theorem~\ref{thm_precise_asymptotics} this normal bundle $N_\util$ extends to a bundle $\overline N_\util$ over $D$ by attaching the fibers of $\xi^{P_0}$ over the circle at infinity.

Let $W$ be a non-vanishing section of $\overline N_\util$ and let $Z$ be a non-vanishing section of $\bar u^*\xi$. Then we can compare them by calculating the relative winding $\wind(W,Z)$ at the circle at infinity. The answer is $\wind(W,Z) = \chi(D) = 1$, a calculation that is done in~\cite{props3}.

The section $W$ can be completed to a frame by adding the section $\jtil W$. This frame trivializes the normal bundle $N_\util \simeq \C\times \C$, and embedded finite-energy curves asymptotic to $P_0$ near $\util$ are represented by sections of $N_\util$, which are seen as graphs $z\mapsto (z,v(z))$ of $\C$-valued functions $v(z)$ with the help of our trivialization. These functions satisfy a certain (non-linear) Cauchy-Riemann equation.

Any number $\delta<0$ can be used to define certain Banach spaces of H\"older-continuous sections of $N_\util$ that decay, together with their derivatives, asymptotically like $e^{\delta s}$ as $z=e^{2\pi(s+it)} \to \infty$. This exponential decay can be described in a precise way that we will not explain here, see~\cite{props3} for more details. The Cauchy-Riemann operator defines a smooth Fredholm map between these Banach spaces, and linearizing at the zero section we obtain a Fredholm operator $L$. The analytical set-up outlined here depends on the choice of $\delta$, in particular, the Fredholm index of $L$ depends on $\delta$.

The first crucial remark is that the inequality $\mu_{CZ}(\util)\geq 3$ allows us to place the negative number $\delta$ in the spectral gap between the eigenvalues of the asymptotic operator $A_{P_0}$ at $P_0$ which have winding $0$ and winding $1$ with respect to a frame aligned with $W$. This carefully placed exponential weight $\delta$ defines a weighted Conley-Zehnder index $\mu^\delta_{CZ}(P_0)$ which equals $3$ when computed with respect to a frame aligned with $Z$. It turns out that the index formula yields $$ {\rm ind}\ L= \mu^\delta_{CZ}(P_0)-1 = 3-1 = 2, $$ that is, we obtain i).

The second crucial remark is that the algebraic count of zeros of sections in $\ker L$ is equal to winding number of the eigenvalue of $A_{P_0}$ immediately below $\delta$, computed with respect to a trivialization that does not wind with respect to $W$. This extremal winding number is zero, as observed above. Hence, sections in $\ker L$ either never vanish or vanish identically, in view of Carleman's similarity principle. Since $L$ has index $2$ and the normal bundle has real rank $2$, it follows that $\dim\ker L\leq 2$. This easily implies that $$ \begin{array}{ccc} \dim {\rm coker}\ L=0 & \text{and} & \dim \ker L=2 \end{array} $$ or, equivalently, that $L$ is surjetive. We obtained ii).

With the help of the implicit function theorem, a neighborhood of $0$ in $\ker L$ can be used to parametrize the embedded curves nearby $\util$ which are detected by this weighted Fredholm theory. At this point the $\delta$-weighted exponential decay implies two important facts. Firstly, two nearby distinct curves do not intersect each other since, using Carleman's similarity principle as above, an intersection implies that the curves coincide. Secondly, the asymptotic approach of such a nearby curve to $P_0$ follows an asymptotic eigenvector with the same winding of the asymptotic eigenvector of $\util$, and we get that all nearby planes are fast, that is, they satisfy $\wind_\pi=0$. As a final step, following the appendix of~\cite{props3}, we can reparametrize the nearby curves in such a way to obtain a map $F$ with the desired properties. See~\cite{hryn} for all these details in the case $P_0$ is a prime closed Reeb orbit.
\end{proof}

\section{Proof of Lemma~\ref{lemma_int}}\label{app_lemma}

Let $v^k_2$ and $v=(v_1,v_2)$ be as in the statement. The hypotheses on $j$ give
\[
j(z_1,z_2) = \begin{pmatrix} i & 0 \\ 0 & i \end{pmatrix} + \Delta(z_1,z_2) z_2 \ \ \text{where} \ \ \Delta(z_1,z_2)=\int_0^1 D_2j(z_1,\tau z_2)d\tau
\]
and consequently $\partial_sv_2 + i\partial_tv_2 + \beta v_2 = 0$ for some smooth $\beta: B_r^+ \to \mathcal L_\R(\C)$, where $s+it$ is the coordinate on $B^+_r$. Our hypotheses give $v_2(0)=0$, $\nabla v_2(0)=0$. By a version of Carleman's similarity principle after shrinking $r$ we can assume that there is a continuous function $\Phi:B^+_r \to \mathcal L_\R(\C)$ such that $\Phi(z)$ is non-singular $\forall z\in B^+_r$,
\[
f(z) := \Phi(z)v_2(z) \ \text{is holomorphic},
\]
$\Phi(s)$ is diagonal $\forall s\in(-r,r)$, and $\Phi(0)=id$. In particular, $f(s) \in [0,+\infty)$ for $s\in(-r,r)$. Expanding in a power series we find $k_0\geq 2$ such that 
\[
f(z) = \sum_{k\geq k_0} c_kz^k \ \ \text{and} \ \ c_{k_0}\neq 0
\]
with $c_k\in\R$. After shrinking $r$ we can assume that $v,\Phi,f$ and all $v^k_2$ are defined on $D^+_r = \overline{B^+_r}$, $\Phi$ is close to $id$ and
\begin{equation}\label{not_zero_v_2}
v_2(z)\neq0, \ \forall z\in D^+_r \setminus\{0\}.
\end{equation}
Following~\cite{char1}, it is crucial to observe that $k_0$ is even. In fact, if $k_0$ is odd then $f(s)$ changes signs at $0$ when $s$ varies from $-r$ to $r$. But this is impossible since $v_2([-r,r]) \subset [0,+\infty)$.

We claim that $\exists \rho_0>0$ such that $0<\rho<\rho_0 \Rightarrow -\rho \not \in v_2(\partial D^+_r), \ -\rho\not\in f(\partial D^+_r)$ and there is a homotopy from $v_2$ to $f$ that does not attain the value $-\rho$ on $\partial D^+_r$. In fact, if $\rho_0$ is small then $-\rho \not\in v_2(\partial D^+_r)\cup f(\partial D^+_r) \ \forall \rho \in (0,\rho_0)$ in view of~\eqref{not_zero_v_2} and since $v_2,f$ are non-negative on $[-r,r]$. Since $\Phi$ is close to $id$ we find a homotopy $\{\Phi_\tau\}_{\tau\in[0,1]}$ satisfying $\Phi_0=\Phi$, $\Phi_1=id$ and $\Phi_\tau(s) \in \mathcal L^+_\R(\R) \ \forall \tau,s$. Hence $$ \{\Phi_\tau(z) v_2(z) \mid \tau\in[0,1], \ |z|=r, \ \Im z\geq 0 \} $$ misses a neighborhood of $0$ and $\Phi_\tau(s) v_2(s) \geq 0, \ \forall (\tau,s) \in [0,1]\times[-r,r]$. We obtain the desired conclusion.

The Brouwer degree $\deg(v_2,D^+_r,-\rho)$ is well-defined for $\rho\in(0,\rho_0)$. To compute it we use the fact that $f$ is real on $[-r,r]$ to apply the reflection principle and assume that $f$ is defined in $D_r = \overline{B_r} = \{z\in\C \mid |z|\leq r\}$. Then $$ k_0 = \deg(f,D_r,-\rho) = \deg(f,D^+_r,-\rho) + \deg(f,D^-_r,-\rho) = 2\deg(f,D^+_r,-\rho) $$ perhaps after shrinking $\rho_0$, where we denoted $D^-_r = D_r \cap \{z\in D_r \mid \Im z \leq 0 \}$. The deformation described before gives $$ k_0/2 = \deg(v_2,D^+_r,-\rho)\geq 1 $$  by homotopy invariance of the degree.

Now, proceeding indirectly, assume that $0\not\in v^k_2(D_r^+) \ \forall k$. We know that $v_2$ misses a neighborhood of $0$ on $\partial D_r^+\setminus (-r,r)$. Thus so does $v_2^k$ for $k$ large. Taking $\rho^k>0$, $\rho^k\to0$ then the entire homotopy $h_\lambda := \lambda(v_2+\rho^k)+(1-\lambda)v_2^k$, $\lambda\in[0,1]$, misses a neighborhood of $0$ on $\partial D_r^+\setminus (-r,r)$ when $k$ is large enough. Clearly $h_\lambda$ misses $0$ on $[-r,r]$ since $v_2([-r,r])\subset [0,+\infty)$, $v_2^k([-r,r])\subset (0,+\infty)$ and $\rho^k>0$. By homotopy invariance of the degree
\[
\deg(v_2^k,D^+_r,0) = \deg(v_2+\rho^k,D^+_r,0) = \deg(v_2,D^+_r,-\rho^k) \geq 1 \ \text{if} \ k\gg1.
\]
This contradiction concludes the proof.

\end{document}